\documentclass[10pt]{amsart}
\usepackage{appendix}
\usepackage{geometry}
\usepackage{amssymb}
\usepackage{amsmath}
\usepackage{amsthm}
\usepackage{mathtools}
\usepackage{mathrsfs}
\usepackage{stmaryrd}
\usepackage{enumitem}
\usepackage{graphicx}   
\usepackage{tikz}
\usepackage{svg}
\usepackage{tikz-cd}
\DeclareFontFamily{U}{rsfs}{\skewchar\font127 }
\DeclareFontShape{U}{rsfs}{m}{n}{%
   <-6> rsfs5
   <6-8> rsfs7
   <8-> rsfs10
}{}
\usepackage{bm}
\usepackage[linktocpage]{hyperref}
\hypersetup{
}
\usepackage{stackengine}
\stackMath

\usepackage{cleveref}
\usepackage{makecell}
\usepackage{framed}
\usepackage{graphicx}
\usepackage{xcolor}
\usepackage{listings}
\usepackage{multirow}
\usepackage{multicol}
\usepackage{indentfirst}
\usepackage{booktabs}
\usepackage{tikz}
\usepackage{longtable} 
\usepackage{tikz-cd}
\usepackage{float}
\usepackage{algorithm}
\usepackage{algorithmicx}
\usepackage{algpseudocode}
\usepackage{subfigure}
\usepackage{lineno}
\usepackage{comment}
\usepackage{setspace}
\usetikzlibrary{patterns, arrows.meta, calc}

\newcommand*{\be}[1]{\begin{equation}\label{#1}}
\newcommand*{\ee}{\end{equation}}

\newtheorem{theorem}{Theorem}[section]

\newtheorem{claim}{Claim}
\newtheorem{lemma}{Lemma}[section]
\newtheorem{proposition}{Proposition}[section]

\newtheorem{corollary}{Corollary}[section]
\theoremstyle{remark}
\newtheorem{remark}{Remark}[section]

\definecolor{pink}{RGB}{255,45,115}

\DeclareMathOperator{\grad}{grad}
\DeclareMathOperator{\hess}{hess}
\DeclareMathOperator{\curl}{curl}
\DeclareMathOperator{\inc}{inc}

\DeclareMathOperator{\cott}{cott}
\DeclareMathOperator{\Def}{def}
\DeclareMathOperator{\dev}{dev}
\DeclareMathOperator{\sym}{sym}
\DeclareMathOperator{\diverenge}{div}

\DeclareMathOperator{\tr}{tr}
\DeclareMathOperator{\Tr}{Tr}

\newcommand\vskw{\operatorname{vskw}}
\newcommand\mskw{\operatorname{mskw}}

\newcommand\skw{\operatorname{skw}}

\newcommand{\bS}{\mathbb S}
\newcommand{\bT}{\mathbb T}
\newcommand{\bST}{\mathbb S \cap \mathbb T}

\renewcommand{\div}{\diverenge}

\usepackage{geometry}
\renewcommand\ker{\mathcal{N}}
 
 \newcommand\ran{\mathcal{R}}

 \numberwithin{equation}{section}

\usepackage{cite}

\allowdisplaybreaks
\title{Finite elements for symmetric and traceless tensors in three dimensions}
\author{Kaibo Hu}
\address{Mathematical Institute, University of Oxford,
Radcliffe Observatory, Andrew Wiles Building,
Woodstock Rd, Oxford OX2 6GG, UK.}
\email{kaibo.hu@maths.ox.ac.uk}
\author{Ting Lin}
\address{School of Mathematical Sciences, Peking University, Beijing 100871, P. R. China.}
\email{lintingsms@pku.edu.cn}
\author{Bowen Shi}
\address{Oden Institute for Computational Engineering \& Sciences, University of Texas at Austin, Austin, TX 78712}
\email{bowenshi@utexas.edu}
\begin{document}
\begin{abstract}
We construct a family of finite element sub-complexes of the conformal complex on tetrahedral meshes and show their exactness on contractible domains. This complex includes vector fields and symmetric and traceless tensor fields, connected through the conformal Killing operator, the linearized Cotton-York operator, and the divergence operator, respectively. This leads to discrete versions of transverse traceless (TT) tensors, i.e., symmetric, traceless and divergence-free matrix fields, in continuum mechanics and general relativity. We also show the inf-sup stability of the $H(\operatorname{div})$-conforming finite element symmetric and traceless tensors paired with discontinuous vectors.
\end{abstract}
\maketitle 
\section{Introduction}
\label{sec:introduction}

Hilbert complexes and their discretizations are fundamental for the construction and analysis of numerical algorithms within Finite Element Exterior Calculus (FEEC) \cite{arnold2018finite,arnold2006finite}. A Hilbert complex consists of a sequence of linear spaces interconnected by a series of linear operators that possess the property that the composition of any two consecutive operators vanishes, along with various other analytical properties \cite{arnold2006finite,bru1992hilbert}. The de~Rham complex is a basic example, encoding differential structures in electromagnetism and other vector-valued problems. Other examples can be derived from de~Rham complexes using the Bernstein-Gelfand-Gelfand (BGG) construction \cite{arnold2021complexes,vcap2023bgg,vcap2001bernstein}, encoding structures of problems involving tensors with various kinds of symmetries, such as symmetric matrix fields (stress \cite{arnold2006defferential,eastwood2000complex,hu2023nonlinear}, metric, Einstein and Ricci tensors \cite{christiansen2011linearization,hu2025finite,berchenko2025finite,gawlik2025finite}) and traceless matrix fields \cite{gopalakrishnan2020mass}. 

Many problems involve matrix fields which are both symmetric and traceless ($\mathbb{S}\cap \mathbb{T}$).  For incompressible Stokes flows, a stress-variable (symmetric gradient) $\boldsymbol \sigma:=\Def(\boldsymbol u)$ was introduced in    \cite{gopalakrishnan2020mass,gopalakrishnan2020mass2}. The new variable $\boldsymbol \sigma$ is symmetric by definition; the incompressibility of the flow ($\div \boldsymbol u=0$, mass conservation) is translated to the algebraic condition of $\boldsymbol \sigma$ being traceless. In the Einstein-Bianchi formulation \cite{quenneville2015new,choquet2009general} of the Einstein equations, the main variables are Transverse-Traceless (TT), i.e.  symmetric, traceless and divergence-free.  In differential geometry, the Einstein tensor $G_{ij}$ and the Ricci tensor $R_{ij}$ differ by a trace: $G_{ij}=R_{ij}-\tfrac{1}{2}Rg_{ij}$, where the scalar curvature $R:=g^{k\ell}R_{k\ell}$ is the trace of the Ricci tensor.  In the construction of numerical discretizations, the challenge of translating results for $G_{ij}$, which appears in the elasticity complex in three dimensions \cite{arnold2007mixed,arnold2006defferential,eastwood2000complex,hu2023nonlinear}, to $R_{ij}$ often implicitly requires understanding the matrix trace of the symmetric $G_{ij}$. This calls for clarifying discrete versions of $\mathbb{S}\cap \mathbb{T}$ tensors, although tracelessness does not appear explicit.

Symmetric and traceless tensor fields also fit within complexes. 
 In this paper, we focus on the {\it conformal deformation complex} in three dimensions:
\begin{equation} 
\boldsymbol{CK} \xrightarrow{\subset} {H}^1(\Omega;\mathbb{R}^3)\xrightarrow{\dev \Def} {H}(\cott,\Omega;\bST) \xrightarrow[]{\cott} {H}(\div,\Omega; \bST) \xrightarrow[]{\div} {L}^2(\Omega;\mathbb R^3) \xrightarrow{}0.\label{eq:sec1:conformal-complex}
\end{equation}

 In this complex, $\bST$ represents symmetric and traceless matrices. The operator $\dev\Def$ is the conformal Killing operator (traceless part of the symmetric gradient), and its kernel, $\boldsymbol{CK}:=\ker(\dev \Def)$, is the 10-dimensional space of conformal Killing fields \cite{neff2009new, dain2006generalized}.

We use the standard Sobolev space notations to define the spaces in the complex as:
$${H}(\cott,\Omega;\bST):=\{\boldsymbol{\sigma}\in L^2(\Omega;  \bST)   :\cott\boldsymbol{\sigma}\in L^2(\Omega;  \bST) \},$$
$${H}(\div,\Omega;\bST):=\{\boldsymbol{\sigma}\in L^2(\Omega;  \bST)  :\div\boldsymbol{\sigma}\in L^2(\Omega;\mathbb{R}^3)\}.$$
The divergence operator applies row-wise and the $\cott$ operator is a third order differential operator (see precise definition in \eqref{def:cott} in Section \ref{sec:notations}). In geometry, if $g$ is the metric, then $\cott(g)$ is the linearized Cotton-York tensor (around the Euclidean metric). Note that $\cott$ maps symmetric and traceless tensors to symmetric and traceless tensors. 
 
 The notion of the conformal complex, originally introduced by Gasqui and Goldschmidt in their investigation of infinitesimal deformations of conformally flat structures in manifolds of dimensions $n\geq3$ \cite{Gasqui_conformal}, extends its influence to general relativity (GR). In the context of GR, the conformal complex encodes the transverse-traceless (TT) gauge (symmetric, traceless and divergence-free tensor fields, \cite{maggiore2007gravitational}) and York splits. In fact,  TT tensors are ${H}(\div,\Omega; \bST)\cap \ker(\div)$, which is $\ran (\cott)$ if the sequence \eqref{eq:sec1:conformal-complex} is exact, meaning that 
 TT tensors can be expressed by a potential in the form of the linearized Cotton-York tensor \cite{beig1996tt}.
 Beig and Chrusciel extended the use of Cotton-York potentials for the Einstein constraint equations \cite{beig2020linearised} and gravity shielding \cite{beig2017shielding}.  The Hodge decomposition inherent to the conformal complex is referred to as York splits in GR literature \cite{york1973conformally,deser1967covariant}. These splits are related to the study of conformal diffeomorphisms on the space of Riemannian metrics \cite{fischer_marsden_1977}.  Moreover, the conformal complex is relevant to Cosserat elasticity \cite{neff2009new,jeong2009numerical} and the trace-free Korn inequality \cite{feireisl2009singular,fuchs2009application,fuchs2010generalizations} in the context of continuum mechanics.

The conformal complex is a special case of the BGG construction \cite{vcap2001bernstein}, and the cohomology is isomorphic to the de~Rham version. Explicit forms with Sobolev spaces on bounded Lipschitz domains in $\mathbb{R}^3$ such as \eqref{eq:sec1:conformal-complex} can be found in \cite{vcap2023bgg,arnold2021complexes}. 
Beig \cite{beig1996tt} proved the exactness of the conformal complex on conformally flat 3-manifolds when the underlying manifold is simply connected and has vanishing second cohomology. The proof was ad hoc without using the BGG approach.

 
 A major idea in FEEC is to discretize the entire complex and preserve the cohomological structures, rather than discretizing individual spaces. This guarantees the stability, convergence and structure-preserving properties of the discrete problems, and facilitates the construction of solvers. Systematic finite element discretizations exist for the de~Rham complex, generalizing the classical Raviart--Thomas \cite{raviart2006mixed}, N\'ed\'elec \cite{nedelec1980mixed}, Brezzi--Douglas--Marini \cite{brezzi1985two} elements to discrete differential forms \cite{hiptmair1999canonical}, leading to a periodic table \cite{arnold2006finite,arnold2014periodic}. Conforming finite element discretizations for the BGG complexes (including the elasticity, Hessian, and divdiv complexes, incorporating matrix fields that are either symmetric or traceless) are much more challenging. The construction of inf-sup stable finite element pairs for the Hellinger–Reissner principle, which involves symmetric stress tensors paired with vector displacements, was regarded as a major challenge around the 2000s \cite{arnold2002icm}. Subsequently, the idea of discretizing the entire complex has inspired considerable progress in addressing this elasticity problem \cite{arnold2002mixed}, as well as other BGG complexes. For the three fundamental examples of the BGG complexes, i.e., the Hessian, elasticity, and divdiv complexes, conforming finite elements on simplicial meshes in both 2D and 3D have been discussed in various works, such as \cite{arnold2008finite,chen2022finitedivdiv2dim,chen2022finitedivdiv,chen2022finite,christiansen2020discrete,christiansen2023finite,gong2023discrete,hu2021conforming,hu2015family,hu2022conformingdivdiv,hu2023nonlinear,chen2025new,chen2024finite,chen2025complexes}. There have also been results in arbitrary dimensions \cite{bonizzoni2023discrete,chen2022divdivanddiv}. These results can be viewed as an extension of the study of multivariate (simplicial) splines \cite{lai2007spline} from a homological perspective.
 
As reviewed earlier, symmetric traceless tensors and conformal complexes play a crucial role in problems arising from fluid mechanics and general relativity. Relevant numerical applications include discretizing transverse-traceless (TT) tensors within the Einstein-Bianchi formulation \cite{quenneville2015new} of general relativity, as well as mass-conserving schemes for incompressible Stokes flow \cite{gopalakrishnan2020mass,gopalakrishnan2020mass2}. However, to the best of our knowledge, conforming finite element discretizations of symmetric traceless tensors and conformal complexes do not currently exist in the literature. In existing works  \cite{quenneville2015new,gopalakrishnan2020mass,gopalakrishnan2020mass2}, the conditions of symmetry or tracelessness of the matrix fields are either disregarded or imposed weakly via Lagrange multipliers. One may expect that imposing these algebraic conditions strongly on the algebraic level will improve stability and enforce structure-preservation in computation. Moreover, discretizing the entire conformal complex will lead to a discrete version of the York split \cite{york1973conformally}. This inspires us to investigate a conforming finite element discretization of \eqref{eq:sec1:conformal-complex}.

In this paper, we establish the first conforming finite element sub-complex of \eqref{eq:sec1:conformal-complex}. We show the exactness of the finite element complex. As a special instance of the construction, we construct
   a conforming finite element pair $\boldsymbol{\Sigma}_{k,h}^{\div}\times\boldsymbol{V}_{k-1,h}\subset{H}(\div,\Omega; \bST)\times {L}^2(\Omega;\mathbb R^3)$ 
that satisfies the balance condition
\begin{align}
\div\boldsymbol{\Sigma}_{k,h}^{\div}=\boldsymbol{V}_{k-1,h},\label{sec1:eq:balance}
\end{align}
and the discrete inf-sup (Ladyzhenskaya–Babu\v{s}ka–Brezzi) condition \cite{boffi2013mixed,brezzi1974existence} 
\begin{align} 
\label{sec1:eq:inf-sup}\inf_{\boldsymbol{v} \in \boldsymbol{V}_{k-1,h}\backslash\{0\}} \sup _{\boldsymbol{\sigma} \in \boldsymbol{\Sigma}_{k,h}^{\div} \backslash\{0\}} \frac{\int_{\Omega}\operatorname{div} \boldsymbol{\sigma}\cdot \boldsymbol{v}}{\|\boldsymbol{\sigma}\|_{H(\div,\Omega)}\|\boldsymbol{v}\|_{L^2(\Omega)}} \geq C>0,
\end{align} 
with a positive constant $C$ independent of the mesh size $h$. A finite element pair satisfying these conditions would preserve the TT structure exactly in numerical computations. The discrete inf-sup condition \eqref{sec1:eq:inf-sup} generalizes results for well-known questions in Stokes and elasticity problems. 

Moreover, we develop techniques to address challenges in extending existing constructions of conforming finite elements for vectors 
and symmetric or traceless tensors. This includes handling supersmoothness, bubble functions, bubble complexes and trace complexes. These results can be applied to other problems. 

The rest of the paper is organized as follows. In Section \ref{sec:overview}, we provide an overview of the construction in this paper. In Section \ref{sec:notations}, we provide notations and various identities on vector- and matrix-valued functions. We introduce some useful results concerning BGG diagrams and geometric decompositions for $P^{(s)}_k$ (polynomial space with vanishing derivatives at vertices). In Section \ref{trace section}, we identify the traces of the linearized Cotton-York operator $\cott$, establish trace complexes and a sufficient condition for the ${H}(\cott)$-conformity. Section \ref{bubble section} is dedicated to demonstrating the exactness of the smoothest bubble conformal complex via BGG construction.
In Section \ref{div section}, we use the exactness of the bubble complex to conduct a dimension count and prove Theorem~\ref{thm:div-surjective property}. Based on Theorem~\ref{thm:div-surjective property}, we construct a balanced pair of ${H}(\div,\bST)$-$L^2(\mathbb R^3)$ conforming finite element spaces and show it satisfies the inf-sup condition. 
In Section \ref{cinc section}, we present bubble conformal complexes with less smoothness and construct ${H}(\cott)$-conforming finite elements for symmetric and traceless tensors.
Finally, in Section \ref{finite element complex section}, we demonstrate that our finite element spaces can be linked to form an exact sequence. Some technical proofs are given in Appendix \ref{technical proofs appendix}.

\section{Technical Challenges and Overview of the Construction}\label{sec:overview}

{Due to the technical nature of this topic, we provide an overview of our construction. }
A conforming finite element sub-complex of \eqref{eq:sec1:conformal-complex} must simultaneously satisfy the following constraints:
\begin{itemize}
\item Unisolvency: degrees of freedom should match local polynomials.
\item Algebraic constraints: matrix fields should be both symmetric and traceless.
\item Conformity: the finite element spaces are subspaces of certain Sobolev spaces, reflected in the fact that the piecewise polynomials should have certain interelement continuity.
\item Sub-complex: the discrete finite element complex should be a sub-complex of the continuous complex, in the sense that the differential operators map one finite element space to another.
\item Cohomology: the cohomology of the discrete sequence should be isomorphic to the continuous version. 
\end{itemize}
{We start by discussing some general issues in constructing conforming finite elements that satisfy these conditions. These issues already arise in well-known problems from incompressible flows and linear elasticity. We discuss the challenges in generalizing these results to symmetric and traceless tensors and our solutions. In the discussions below, we start with the divergence operator and the spaces it connects, i.e., the last part of the complex. This is because the divergence pair corresponds to widely discussed questions in incompressible flows and linear elasticity; many technical issues, such as bubble spaces and supersmoothness, already appear.  Then we discuss how to extend the results for the divergence pair to the entire complex. This involves the trace complexes. }


\subsection{Review: technical aspects in constructing conforming finite elements}
\subsubsection{Conformity and supersmoothness}
The divergence operator and corresponding inf-sup stable finite element pairs in the form of \eqref{sec1:eq:inf-sup} appear in several classical problems in the finite element theory: 
\begin{enumerate}
\item  In classical de~Rham complex,  $\boldsymbol{\sigma}\in H(\div, \Omega)$ is a vector and ${v}\in L^{2}(\Omega)$ is a scalar. 
\item In Stokes problems, $\boldsymbol{\sigma}\in  H^{1}(\Omega; \mathbb{R}^{3})$ is a vector and ${v}\in L^{2}(\Omega)$ is a scalar.
\item In linear elasticity, $\boldsymbol{\sigma}\in H(\div, \Omega; \mathbb{S})$ is a symmetric matrix field and $\boldsymbol{v}\in L^{2}(\Omega; \mathbb R^{3})$ is a vector.
\item In the setup concerned in this paper, i.e., in the conformal complex, $\boldsymbol{\sigma}\in H(\div, \Omega; \mathbb{S}\cap \mathbb{T})$ is a $\mathbb{S}\cap \mathbb{T}$ tensor and $\boldsymbol{v}\in L^{2}(\Omega; \mathbb R^{3})$ is a vector.
\end{enumerate}
 Except for (1), for which there exists a canonical construction, all other cases present significant challenges. For the Stokes problem, the $H^{1}$-conformity requires each component of $\boldsymbol{\sigma}$ to be continuous across the boundaries of cells; for the elasticity problem, the shape functions should be matrices with algebraic symmetries. From a differential complex perspective, these two problems result in a common challenge as the complexes that they sit in, for example, in 2D, 
\begin{equation}\label{cplx:stokes-2D}
{0} \longrightarrow
H^{2}(\mathbb{R})
\xrightarrow{\curl}
H^{1}(\mathbb{R}^{2})
\xrightarrow{\div}
L^{2}(\mathbb{R})
\longrightarrow
{0},
\end{equation}
 and 
\begin{equation}\label{cplx:elasticity-2D}
{0} \longrightarrow
H^{2}(\mathbb{R})
\xrightarrow{\curl\curl}
H(\div;\mathbb{S})
\xrightarrow{\div}
L^{2}(\mathbb{R}^{2})
\longrightarrow
0,
\end{equation}
 essentially involve $H^{2}$ spaces on triangulations, which require scalar splines or finite elements with continuity higher than $C^{0}$. Simplicial splines still see many open problems \cite{hu2024condition,schenck2016multivariate,billera1988homology}. A general result claims that simplicial splines may exhibit \emph{supersmoothness}: $C^{r}$ piecewise smooth functions may have higher continuity at corners (vertices, edges...) of the mesh \cite{sorokina2010intrinsic,floater2020characterization}. Correspondingly, in the construction of conforming finite elements, such supersmoothness should be incorporated, either as part of the degrees of freedom, or one should use macroelements (further splits of each cell) to resolve the supersmoothness. Such supersmoothness will propagate in complexes such as \eqref{cplx:stokes-2D} and \eqref{cplx:elasticity-2D}. For example, a conforming $C^{1}$ finite element discretization of $H^{2}$ involves $C^{2}$ degrees of freedom at vertices. Consequently, its first order derivatives involve $C^{1}$ at vertices. This is a higher continuity requirement than  $ H^{1}$-conformity ($C^{0}$). This indicates that to obtain a finite element complex, we need to impose ($C^{1}$) continuity at vertices higher than the ($C^{0}$) global smoothness that $ H^{1}$ naturally requires. For more general cases, such as in higher dimensions with more delicate Sobolev spaces, characterizing such conditions remains open. 

 For the Stokes problem, the first work investigating the supersmoothness is Falk and Neilan \cite{falk2013stokes}. They constructed a Stokes pair with $C^{1}$ vertex continuity for the velocity and $C^{0}$ vertex continuity for the pressure, fitting within a discrete version of the complex \eqref{cplx:stokes-2D}. For the elasticity, it is already realized at an early stage \cite{arnold2002mixed} that the stress $\boldsymbol{\sigma}$ should have $C^{0}$ continuity at vertices. Examples include the Arnold-Winther element \cite{arnold2002mixed} and the Hu-Zhang element \cite{hu2014family}, constructed from different perspectives. 
 
Generalizing this idea to three dimensions sees essential challenges, as the complexes become longer: there are more slots to fill in and the supersmoothness is much less clear. For the Stokes problem, a generalization was given by Neilan \cite{neilan2015discrete}, where the velocity has $C^{2}$ continuity at vertices with piecewise polynomials of degree higher than or equal to 7 (this is consistent with the fact that $C^{1}$ scalar simplicial spline in 3D starts with $C^{4}$ vertex continuity with piecewise polynomials of degree 9). For linear elasticity, a 3D generalization of the Arnold-Winther element can be found in \cite{arnold2008finite}, and a 3D elasticity complex containing the Hu-Zhang pair \cite{hu2015family} was given by \cite{chen2022finite}. 

\subsubsection{Bubbles, trace and bubble complexes}\label{subsec:bubbles}

{To motivate the discussions on bubble functions, we recall the fact from Stokes problems that $\div: \boldsymbol H_{0}^{1}\to L^{2}\cap\mathbb{R}^{\perp}$ is onto, where $L^{2}\cap\mathbb{R}^{\perp}:=\{u\in L^{2}: \int_{\Omega} u\, dx=0\}$. This indicates that the $L^{2}$ pressure is almost controlled by the interior modes of the $\boldsymbol H_{0}^{1}$ velocity up to a constant. Similarly, conclusions hold for $\div: H_{0}(\div; \mathbb{X})\to L^{2}(\mathbb{R}^{d})\cap\mathcal{X}^{\perp}$, where $\mathcal{X}:=\ker(\div^{\ast})$ is a finite dimensional space.  On the finite-dimensional level, this indicates a general pattern that on each cell, the pressure is controlled by the interior modes, i.e., {\it bubbles}, of the velocity, up to a constant. The remaining constant is further controlled by one mode per face in the velocity space, mimicking the structure in the Raviart--Thomas element and piecewise constants (see Figure \ref{fig:bubble-decomposition} for an illustration). }

\begin{figure}
\begin{minipage}{0.68\textwidth}
$$
\begin{tikzcd}[row sep = small]
\cdots\arrow{r} & \boldsymbol \Sigma^{\div}_{h}\arrow{r}{\div} & \boldsymbol V_{h} \arrow{r} & 0\quad \quad \mbox{full}\\
&\parallel&\parallel&\\
\cdots\arrow{r} & 0 {\boldsymbol\Sigma}^{\div}_{h}\arrow{r}{\div} & \boldsymbol V_{h}/\ker(\div^{*}) \arrow{r} & 0\quad \quad \mbox{bubble}\\
&+&+&\\
\cdots\arrow{r} &  \tilde{\boldsymbol\Sigma}^{\div}_{h}\arrow{r}{\div} & \ker(\div^{*}) \arrow{r} & 0\quad \quad \mbox{skeleton}\\
\end{tikzcd}
$$
\end{minipage}
\begin{minipage}{0.3\textwidth}
  \parbox{0.59\textwidth}{\vspace{+0.1cm}
\centering
\includegraphics[height=0.3\textheight]{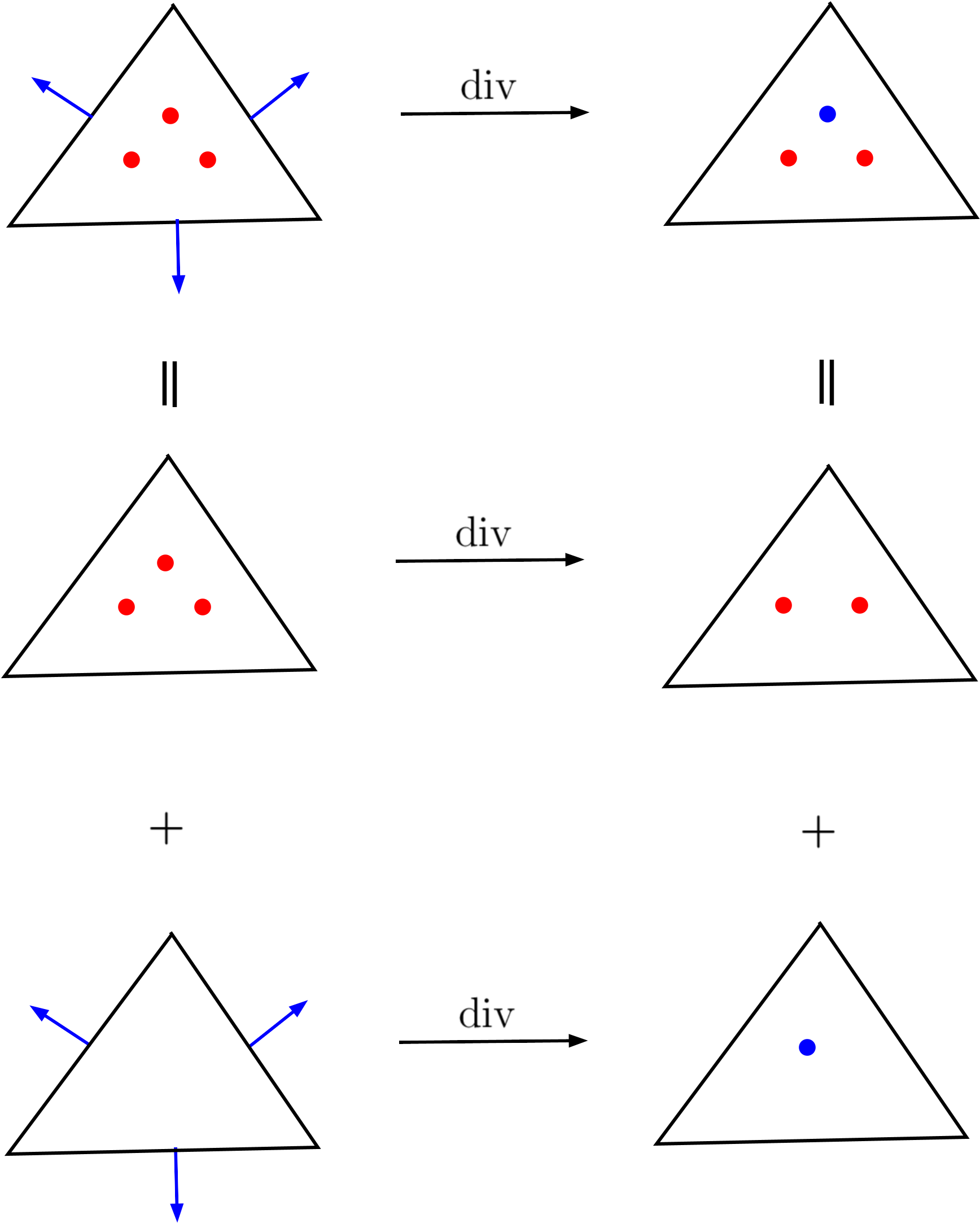}
}
\end{minipage}
\caption{{Illustration of the decomposition of a finite element space into bubbles and the rest (skeleton). At the end of the complex, $\div$ maps bubbles onto piecewise polynomials module a finite dimensional space, which is controlled by face degrees of freedom. In general, the bubbles living on each entity (cells, faces, edges, etc.) form a complex; a finite element space can be decomposed into bubbles in different dimensions plus the rest (skeleton).}}
\label{fig:bubble-decomposition}
 \end{figure}
The main question that we answer in this work (particularly for the last part of the complex) is as follows:
\begin{quote}
    \emph{What smoothness conditions should be imposed to construct a balanced \eqref{sec1:eq:balance} and inf-sup stable \eqref{sec1:eq:inf-sup} finite element pair $\boldsymbol{\Sigma}_{k,h}^{\div}\times\boldsymbol{V}_{k-1,h}\subset{H}(\div,\Omega; \bST)\times {L}^2(\Omega;\mathbb R^3)$, and the full conforming finite element sub-complex of \eqref{eq:sec1:conformal-complex}?}
\end{quote}
Let $T$ be a simplex and $P_{k}(T;\mathbb R)$ be the space of polynomials on $T$ with a degree less than or equal to $k$. A \emph{bubble space} is the subspace of a polynomial space whose \emph{trace} (the terms appearing in the integration by parts) on $\partial T$ vanishes. For instance, the canonical $H^1$ bubble space on a tetrahedron $K$ is given as 
$$
\mathbb{B}^1_k(K;\mathbb R^d)\;:=\;
\bigl\{\,\boldsymbol \sigma\in P_k(K;\mathbb R^d)\;\big|\; \boldsymbol \sigma|_{\partial K}=0\bigr\} = b_KP_{k-4}(K;\mathbb R^d),
$$
where $b_K:=\lambda_{0}\lambda_{1}\lambda_{2}\lambda_{3}$ denotes the bubble function of $K$, which is the product of the four barycentric coordinates.
The
$H(\div)$--type bubble spaces are defined as
\begin{align}
     \mathbb{B}^{\div}_k(K;\mathbb{X})
  :=\bigl\{
        \boldsymbol{\sigma}\in P_k(K;\mathbb{X})
        \;\big|\;
        \boldsymbol{\sigma}\boldsymbol n|_F = \mathbf 0\ \forall
        \text{ face }F
     \bigr\},
  \quad
  \mathbb{X}\in\{\mathbb{R}^d,\mathbb{S},\mathbb{T},\bST\},\label{def:div_bubble}
\end{align}
where $P_k(K;\mathbb{X}):=P_k(K;{\mathbb R})\otimes \mathbb X$ and $\boldsymbol{n}$ represents the face-normal vectors.

{
 The decomposition of a finite element space into bubbles and skeleton not only works for the last part of the complex involving divergence. To introduce the decomposition for the entire complex, we review the notions of {\it trace complexes} and {\it bubble complexes}.}  
 
 For the de~Rham complexes,   the trace spaces also form a \emph{trace complex}, and the trace operators commute with the differential operators. {For example,  for the de~Rham complex in $\mathbb{R}^{3}$, } the following diagram commutes:
\[
\begin{tikzcd}[row sep = large, column sep = large,
               ampersand replacement = \&]   
  H^{1}(\mathbb R)
      \arrow[r,"\grad"]
\& H(\operatorname{curl};\mathbb R^{3})
      \arrow[r,"\operatorname{curl}"]
\& H(\operatorname{div};\mathbb R^{3})
\\[-2pt]
  u
    \arrow[u,phantom,"\mathrel{\rotatebox{90}{$\in$}}"]
    \arrow[r,"\grad"]
    \arrow[d,"\operatorname{tr}"']
\& \boldsymbol{\sigma}
    \arrow[u,phantom,"\mathrel{\rotatebox{90}{$\in$}}"]
    \arrow[r,"\operatorname{curl}"]
    \arrow[d,"\operatorname{tr}"']
\& \boldsymbol{\tau}
    \arrow[u,phantom,"\mathrel{\rotatebox{90}{$\in$}}"]
    \arrow[d,"\operatorname{tr}"']
\\
  u\!\mid_{F}
      \arrow[r,"\grad_F"]
\& (\Pi_F\boldsymbol{\sigma})\!\mid_{F}
      \arrow[r,"\operatorname{rot}_F"]
\& (\boldsymbol{\tau}\!\cdot\!\boldsymbol{n})\!\mid_{F}
\\
  H^{1}(\mathbb R)
      \arrow[r,"\grad_F"]
      \arrow[u,phantom,"\mathrel{\rotatebox{270}{$\in$}}"]
\& H (\operatorname{rot}_{F};\mathbb R^{2})
      \arrow[r,"\operatorname{rot}_F"]
      \arrow[u,phantom,"\mathrel{\rotatebox{270}{$\in$}}"]
\&L^{2}(\mathbb R)
      \arrow[u,phantom,"\mathrel{\rotatebox{270}{$\in$}}"]
\end{tikzcd}
\]
where $\Pi_F:=\mathbf{I}-\boldsymbol{n}\boldsymbol{n}^T$ denotes the surface orthogonal projection.

\begin{remark}
For complexes of Sobolev spaces, $\tr(H^{1})=H^{\tfrac{1}{2}}$, and $\tr H(\operatorname{div};\mathbb R^{3})\subset H^{-\tfrac{1}{2}}(F;\mathbb{R}^{2})$. However, for finite element spaces, the traces have higher regularity and fit in the more regular complex
 $$
 \begin{tikzcd}[row sep = large, column sep = large,
               ampersand replacement = \&]
  H^{1}(\mathbb R)
      \arrow[r,"\grad_F"]
\& H (\operatorname{rot}_{F};\mathbb R^{2})
      \arrow[r,"\operatorname{rot}_F"]
\&L^{2}(\mathbb R).
\end{tikzcd}
$$
\end{remark} 
Consequently, the differential operators map bubble spaces to bubble spaces. By replacing each Sobolev space in the Hilbert complex with its corresponding bubble space, we obtain a new complex, referred to as the \emph{bubble complex}. 
Despite being purely local, bubbles and bubble complexes are powerful tools in the construction of global finite element complexes. For de~Rham complexes, bubbles living on each topological entity (cells, faces, and edges etc.) form an exact sequence; while the cohomology of the global finite element spaces is carried in the lowest order Whitney forms. This leads to a {\it geometric decomposition} of global finite element spaces into cell, face, edge and vertex modes \cite{arnold2006finite,schoberl2005high,christiansen2010finite,   arnold2009geometric, chen2022divdivanddiv}. Therefore, the first step of obtaining global finite element spaces with correct cohomology is to obtain exact bubble sequences. The bubble spaces and complexes also indicate the structure of local degrees of freedom (DOFs).

\begin{remark}[The interface+bubble decomposition]
Instead of the decomposition of finite element functions into bubbles living in all dimensions plus a skeletal part, we consider a coarser decomposition into {\it cell} bubbles and complement. Correspondingly, the DOFs are decomposed into two complementary classes:
(i) \emph{interface DOFs}, attached to lower-dimensional sub-simplices (vertices, edges, faces) to enforce the relevant traces being single-valued, and vanishing of which for a local polynomial implies inclusion in a bubble space; and
(ii) \emph{bubble DOFs}, usually moments against the bubble space (or its $L^2$ dual) to ensure interior completeness.
With the dimension identity,
\[
  \dim(\text{interface}) + \dim(\text{bubble}) = \dim(\text{shape}),
\]
unisolvency can be established. The interface DOFs must be calibrated such that the assembled global space is embedded in the target Sobolev space, and the discrete finite element sequence constitutes a sub-complex. 
\end{remark}

\subsubsection{Divergence pair: inf–sup stability via bubble surjectivity}\label{subsec:infsup-bubbles}

The discrete divergence inf-sup condition for a conforming mixed finite element pair as mentioned earlier
is indicated, in essence, by the \emph{last arrow} of the corresponding \emph{bubble complex}.  
Recall that one needs an {\it elementwise surjection}
\begin{equation}\label{eq:local-surj}
  \div:\mathbb B^{\div}_k\!\bigl(K;\mathbb X\bigr)\;
       \twoheadrightarrow\;
       \mathcal Q_{k-1}(K),\footnote{Here ``\(\twoheadrightarrow\)'' indicates surjectivity.}
\end{equation}
or, in the $H^{1}$–case,
\(
  \div:\mathbb B^{1}_k(K;\mathbb R^{d})\twoheadrightarrow\mathcal Q_{k-1}(K),
\)
where the target space  
\(
  \mathcal Q_{k-1}(K)\subset P_{k-1}(K)
\)
is obtained from
\(\,P_{k-1}\)-polynomials by imposing vertex/edge vanishing conditions (supersmoothness) and $L^2$ orthogonality to a low-degree polynomial (due to Stokes' formula).  

Once a right–inverse of \eqref{eq:local-surj} is known, interface DOFs can be tuned to match with the bubble spaces. Canonical scaling and interpolation arguments already developed in \cite{falk2013stokes,arnold2008finite,neilan2015discrete} then lift the local surjection to a \emph{bounded global} right–inverse, and the discrete inf–sup constant follows.

For the Falk–Neilan Stokes element in 2D, one has  
\begin{equation}\label{eq:stokes-bubble-2D}
  \div:\mathbb B^{1}_{k}(F;\mathbb R^{2})
   \twoheadrightarrow
   \bigl\{q\in P_{k-1}(F)\;:\;q(\delta)=0\ \forall\text{ vertices }\delta,\;
        \int_{F}q=0\bigr\},
\end{equation}
where $\mathbb B^{1}_{k}(F;\mathbb R^{2})$ denotes the $H^1$ bubble space on the face $F$.
This is a result first proved by Vogelius via an explicit local right–inverse \cite{vogelius1983right}.  
The $C^1-C^0$ vertex continuity of the Falk–Neilan pair aligns exactly with the restrictions in \eqref{eq:stokes-bubble-2D}. In three dimensions, supersmoothness propagates to edges; Neilan \cite{neilan2015discrete} showed that the bubble divergence satisfies  
\begin{equation}\label{eq:stokes-bubble-3D}
  \div:\mathbb B^{1}_{k}(K;\mathbb R^{3})
  \twoheadrightarrow
  \bigl\{q\in P_{k-1}(K)\;:\; q|_{e}=0\ \forall\text{ edges }e,\;
        \int_{K}q=0\bigr\},
\end{equation}
without constructing an explicit inverse.  Instead, Neilan established the exactness of a \emph{bubble Stokes complex} which facilitates the dimension count of  
\(\mathbb B^{1}_{k}(K;\mathbb R^{3})\cap\ker(\div)\),
and derived \eqref{eq:stokes-bubble-3D} through dimension counting.  Matching interface degrees of freedom (DOFs) with bubble spaces results in globally stable 3D Stokes pairs, albeit at a high polynomial degree.

For the de~Rham sequence, FEEC provides a classical surjection:  
\[
  \div:\mathbb B_k^{\div}(K;\mathbb R^{d})
        \twoheadrightarrow
        P_{k-1}(K;\mathbb R)\cap\mathbb R^{\perp},
\]
where \(\perp\) denotes the $L^{2}$‑orthogonal complement. When the $H(\div)$ space carries additional algebraic structure, the bubble image changes. In the case of symmetric tensors, the elementwise surjection
\begin{equation}\label{eq:bubble-div-S}
  \div:\;\mathbb{B}_{k}^{\div}(K;\mathbb S)
        \;\twoheadrightarrow\;
        P_{k-1}(K;\mathbb R^{3})\cap\boldsymbol{RM}^{\perp}
\end{equation}
holds, with \(\boldsymbol{RM}\) the space of infinitesimal rigid motions:
\begin{align}
      \boldsymbol{RM}:=\ker(\Def)=\{\boldsymbol a\times\boldsymbol x+\boldsymbol b:\boldsymbol a, \;\boldsymbol b\in \mathbb R^3\}.\label{def:RM}
\end{align}

Since the dimension of $\mathbb{B}_k^{\div}(K;\bS)$ is relatively easy to compute using geometric decompositions, there are two distinct perspectives from which to interpret this result: \emph{working with the kernel $\ker(\div)$} (with bubble complexes) or \emph{working with the range $\ran(\div)$} (without bubble complexes).
In the construction of Arnold–Awanou–Winther \cite{arnold2008finite} elements, the authors demonstrated the exactness of the following \emph{bubble elasticity complexes}:
\begin{align}
0 \;\longrightarrow\;
 b_KP_{k-3}(K;\mathbb R^3)
  \xrightarrow{\Def} 
  \mathbb{B}_{k}^{\inc}(K;\mathbb S)
  \xrightarrow{\inc} 
  \mathbb{B}_{k-2}^{\div}(K;\mathbb S)\cap\ker(\div)
  \xrightarrow{\div} 
0&,\label{sec1:eq:elasticity-bubble}\\
0 \;\longrightarrow\;
 b_K^2P_{k-7}(K;\mathbb R^3)
  \xrightarrow{\Def} 
  b_K\mathbb{B}_{k-4}^{1\inc}(K;\mathbb S)
  \xrightarrow{\inc} 
  \mathbb{B}_{k-2}^{\div}(K;\mathbb S)\cap\ker(\div)
  \xrightarrow{\div} 
0&,\label{sec1:eq:elasticity-bubble-2}
\end{align}
where  
\(\inc\) denotes the second order incompatibility operator (see precise definition in \eqref{def:inc}, Section \ref{sec:notations}); $\mathbb{B}_{k}^{\inc}$ and $b_K\mathbb{B}_{k-4}^{1\inc}$ represent $H(\inc)$ and $H(\inc)\cap H^1$ bubble spaces respectively:
 \begin{align*}    \mathbb{B}^{\inc}_k\left(K;\bS\right)&:=\left\{\boldsymbol{\sigma}\in P_k\left(K;\bS\right):\tr_1^{\inc}\left(\boldsymbol{\sigma}\right)|_F=\tr_2^{\inc}\left(\boldsymbol{\sigma}\right)|_F=0,\;\forall {\text{ face } F\in\mathcal{F}(K)}\right\},\\
    {\mathbb{B}}^{1\inc}_{k-4}\left(K;\bS\right)&:=\left\{{\boldsymbol{\sigma}}\in P_{k-4}\left(K;\bS\right):\tr_1^{\inc}(\boldsymbol{\sigma})|_F=0,\;\forall{\text{ face } F\in\mathcal{F}(K)}\right\}.
\end{align*}

Here \(\tr_{1}^{\inc}\) and \(\tr_{2}^{\inc}\) represent the zeroth- and first-order traces of the incompatibility operator \(\inc\) (their precise definitions are provided in \eqref{inc_trace1} and \eqref{inc_trace2}, Section~\ref{trace section}).  Both traces arise from the integration-by-parts identity for symmetric tensors $\boldsymbol{\sigma}$ and $\boldsymbol{\tau}$
\cite{arnold2008finite,chen2022finite}:
\begin{align*}
    \int_K \inc \boldsymbol{\sigma} : \boldsymbol{\tau} -     \int_K \inc \boldsymbol{\tau} : \boldsymbol{\sigma} &=- \int_{\partial K}\tr_1^{\inc} (\boldsymbol{\sigma}) :\tr_2^{\inc} (\boldsymbol{\tau}) + \int_{\partial K} \tr_2^{\inc} (\boldsymbol{\sigma}) :\tr_1^{\inc} (\boldsymbol{\tau})\\
    & \quad + \text{ edge integrals.}
\end{align*}

The bubble space \(\mathbb{B}^{\inc}_{k}(K;\mathbb{S})\) is
particularly delicate, as two distinct traces must vanish on \(\partial K\). Consequently, a significant portion of ~\cite{arnold2008finite} is dedicated to determining \(\dim\mathbb{B}^{\inc}_{k}(K;\mathbb{S})\).  Once this dimension is established, the exactness of ~\eqref{sec1:eq:elasticity-bubble} immediately yields \(
  \dim\bigl(\mathbb{B}_{k-2}^{\div}(K;\mathbb{S})\cap\ker(\div)\bigr)
\), thereby allowing the surjectivity ~\eqref{eq:bubble-div-S} to follow through a straightforward dimension count.

In \cite{arnold2008finite}, it was shown that the exactness of~\eqref{sec1:eq:elasticity-bubble} further implies the
exactness of~\eqref{sec1:eq:elasticity-bubble-2}.  The bubble
space
\(b_{K}\mathbb{B}_{k-4}^{1\inc}(K;\mathbb{S})
  \subset\mathbb{B}_{k}^{\inc}(K;\mathbb{S})\)
is significantly simpler than \(\mathbb{B}_{k}^{\inc}(K;\mathbb{S})\)
itself and admits explicit geometric decompositions.  Consequently,
the bubble complex~\eqref{sec1:eq:elasticity-bubble-2} provides a computable basis for \(\mathbb{B}_{k-2}^{\div}(K;\mathbb{S})\cap\ker(\div)\), which is a cornerstone in the design of the Arnold–Awanou–Winther element. This is the way of \emph{working with $\ker(\div)$}.

Hu and Zhang \cite{hu2015family} provided a proof of \eqref{eq:bubble-div-S} without employing any exact sequence: they provided an explicit characterization of \(\mathbb{B}_{k}^{\div}(K;\mathbb S)\) and identified its $L^{2}$-dual as precisely \(P_{k-2}(K;\mathbb S)\).  Therefore, for any $\boldsymbol p\in P_{k-1}(K;\mathbb R^3)$ such that $\boldsymbol{p}\perp \ran(\div)$, it holds that $$\Def \boldsymbol{p}\perp \mathbb B_k^{\div}(K;\bS), \quad\text{hence } \boldsymbol p\in \ker(\Def)=\boldsymbol{RM}\text{ and $\div$ is a surjection.} $$This approach avoids the delicate dimension count required for the \(\inc\)-bubble space. Hu and Liang \cite{hu2021conforming} later extended this technique to demonstrate that
an analogous result holds for traceless tensors:
\begin{equation}
      \div:\;\mathbb{B}_{k}^{\div}(K;\mathbb T)
        \;\twoheadrightarrow\;
        P_{k-1}(K;\mathbb R^{3})\cap\boldsymbol{RT}^{\perp}\label{eq:bubble-div-T}
\end{equation}
is surjective, where $\boldsymbol{RT}$ denotes the local Raviart–Thomas space:
  \begin{align}
  \boldsymbol{RT}:=\ker(\dev\grad)=\{a\boldsymbol x+\boldsymbol b: a\in \mathbb{R}, \boldsymbol b\in \mathbb R^{3}\}.\label{def:RT}
  \end{align}
And this corresponds the approach of \emph{working with $\ran(\div)$} directly.
   For any $\boldsymbol{\sigma} \in \mathbb{B}_k^{\operatorname{div}}(K; \mathbb{S})$ or $\boldsymbol{\sigma} \in \mathbb{B}_k^{\operatorname{div}}(K; \mathbb{T})$, it can be demonstrated that
        \begin{itemize}
            \item $\boldsymbol{\sigma} = 0$ on all vertices of $K$.
            \item $\boldsymbol{\sigma} \boldsymbol{n}_{e\pm}= 0$ across all edges of $K$, with $\boldsymbol{n}_{e\pm}$, and the edge tangential vector forming an orthogonal basis.            
\item $\boldsymbol{\sigma} \boldsymbol{n} = 0$ on all faces of $K$.
\end{itemize}
For either the $H(\div,\bS)$-$L^2(\mathbb R^3)$ Hu-Zhang pair \cite{hu2015family} or the $H(\div,\bT)$-$L^2(\mathbb R^3)$ Hu-Liang pair \cite{hu2021conforming}, the design of the interface degrees of freedom (DOFs) is established as follows:
\begin{itemize}
        \item The $P_k(\mathbb S)$/$P_k(\mathbb T)$ $H(\div)$-conforming space is $C^0$ at vertices and has edge-normal continuity along internal edges ($\boldsymbol{\sigma} \boldsymbol{n}_{e\pm}$ being single-valued across each edge), and face-normal continuity across internal interfaces.
        \item The $P_{k-1}(\mathbb R^3)$ displacement space has no continuity requirements on sub-simplices.
    \end{itemize}
This aligns with the smoothness of the bubble spaces, and it is demonstrated that these pairs are inf-sup stable.

These cases share the same idea: verifying the surjection between local bubbles \eqref{eq:local-surj}, designing interface DOFs that enforce proper edge/vertex/face conditions, and lifting  the local right–inverse to obtain a global inf–sup bound.  
\subsection{Overview of the construction of finite element conformal complexes}

Discretizing the conformal complex is considerably more challenging than its Stokes or elasticity counterparts: one must enforce \emph{simultaneously} the algebraic constraints of symmetry and tracelessness, while also accommodating the higher-order operator $\cott$. Although the canonical BGG-based framework developed by Bonizzoni \emph{et al.}  \cite{bonizzoni2023discrete} constructs conforming elements for the $\mathbb S$- and $\mathbb T$-based complexes in arbitrary dimensions, it does not extend in a direct manner to their intersection $\bST$.  This observation already suggests a genuine structural obstruction: a naive intersection of the known $\mathbb S$ and $\mathbb T$ bubble spaces could lead to an over-constraining of the local shape functions. {In this subsection, we discuss some of the technical issues in the previous section for $\mathbb{S}\cap \mathbb{T}$ finite elements and conformal complexes, and provide an overview of the construction of this paper.}

\subsubsection{Inf-sup condition of the divergence pair}

Recall that $\boldsymbol{CK}=\ker(\dev\Def)$ is the ten-dimensional space of conformal Killing fields. It has the form 
\begin{align}\label{def:CK}    \quad\quad\quad\boldsymbol{CK}=\{(\boldsymbol{x}\cdot\boldsymbol{x})\boldsymbol{a}-2(\boldsymbol{a}\cdot\boldsymbol{x})\boldsymbol{x}+\boldsymbol{b}\times\boldsymbol{x}+c\boldsymbol{x}+\boldsymbol{d}:\boldsymbol{a},\boldsymbol{b},\boldsymbol{d}\in\mathbb{R}^3,c\in\mathbb{R}\},
\end{align}
which was shown in \cite{neff2009new}; for completeness, we include a self-contained proof in Appendix \ref{appendix:characterization_of_CK}.
By Stokes’ formula, for any $\boldsymbol{\sigma}\in\mathbb{B}_{k}^{\div}(K;\bST)$ and any $\boldsymbol{q}\in\boldsymbol{CK}$,
\[
  \int_K \div \boldsymbol{\sigma}\cdot \boldsymbol{q}
  \;=\;
    -\int_K \boldsymbol{\sigma} : \dev\Def \boldsymbol{q}
    + \int_{\partial K} \boldsymbol{q} \cdot \boldsymbol{\sigma}\boldsymbol{n}
  \;=\;0.
\]
Motivated by the surjectivity results \eqref{eq:bubble-div-S} and \eqref{eq:bubble-div-T} for the pure $\mathbb S$- and $\mathbb T$-settings, one might conjecture that the map
\begin{equation}\label{eq:fake-surjection}
  \div:\;\mathbb{B}_{k}^{\div}(K;\bST)
  \longrightarrow
  P_{k-1}(K;\mathbb R^{3})\cap\boldsymbol{CK}^{\perp}
\end{equation}
is surjective. However, {we will demonstrate that} 
this statement is incorrect.  While for $\mathbb{B}_k^{\div}(K;\mathbb S)$ and $\mathbb{B}_k^{\div}(K;\mathbb T)$ only the normal components on the edges vanish,    for any $\boldsymbol{\sigma} \in \mathbb{B}_k^{\operatorname{div}}(K; \bST)$, it can be demonstrated that
        \begin{itemize}
            \item $\boldsymbol{\sigma} = 0$ at all vertices of $K$.
            \item $\boldsymbol{\sigma} = 0$ across all edges of $K$.            
\item $\boldsymbol{\sigma} \boldsymbol{n} = 0$ on all faces of $K$.
\end{itemize} This forces $\div\boldsymbol{\sigma}$ to vanish at all vertices (since $D\boldsymbol{\sigma}$ vanish at all vertices). To encode such vanishing conditions, define for $s\ge0$
\begin{equation}\label{def:polynomials-enhanced-smoothness}
  P_{k}^{(s)}(K;\mathbb X)
  :=
  \Bigl\{
    \boldsymbol{\sigma}\in P_{k}(K;\mathbb X)
    :
    D^{\alpha}\boldsymbol{\sigma}(\delta)=0
    \quad\forall\text{ vertices }\delta\in\mathcal V(K),\;
    0\le|\alpha|\le s
  \Bigr\}.
\end{equation}
It is then natural to weaken \eqref{eq:fake-surjection} and investigate if the following map is a surjection:
\begin{equation}\label{eq:fake-surjection-2}
  \div:\;\mathbb{B}_{k}^{\div}(K;\bST)
  \longrightarrow
  P_{k-1}^{(0)}(K;\mathbb R^{3})\cap\boldsymbol{CK}^{\perp}.
\end{equation}
Nevertheless, we shall demonstrate that even \eqref{eq:fake-surjection-2} is \emph{not} a surjection. In fact, by introducing additional vertex smoothness, we define for \(s\ge1\)
\begin{align}
    \mathbb B_{k}^{\div,(s)}(K;\bST)
:=P_{k}^{(s)}(K;\bST)\cap\mathbb B_{k}^{\div}(K;\bST),\label{def:div_bubble_extra_smoothness}
\end{align}
where \(P_{k}^{(s)}(K;\bST)\) was defined in \eqref{def:polynomials-enhanced-smoothness}. Clearly, we have the result \(\mathbb B_{k}^{\div,(1)}(K;\bST)=\mathbb B_{k}^{\div}(K;\bST)\) from edge supersmoothness mentioned earlier. Our main results on $\bST$ surjectivity are as follows:
\begin{theorem}[Bubble stability for the \(H(\div,\bST)\text{–}L^{2}\) pair]%
\label{thm:div-surjective property}
Let \(k\ge 7\).  Then
\[
\div:\;
\mathbb B_{k}^{\div,(s)}(K;\bST)
\;\longrightarrow\;
P_{k-1}^{(s-1)}(K;\mathbb R^{3})\cap\boldsymbol{CK}^{\perp}
\]
is surjective when \(s=3\); however, for \(s<3\) surjectivity fails.
\end{theorem}
\begin{remark}
This theorem shows that bubble surjectivity does not hold for the canonical space 
$\mathbb B_k^{\div}(K;\bST)$ unless additional $C^3$ smoothness at the vertices is imposed. 
This is in sharp contrast with the $\bS$- and $\bT$-settings, where surjectivity holds already for the canonical bubble spaces.  

For $\bS$ and $\bT$, the arguments establishing 
\eqref{eq:bubble-div-S} and \eqref{eq:bubble-div-T} (in the Hu–Zhang and Hu–Liang elements, respectively) rely on an \emph{edge decomposition} of the bubble spaces. Such a decomposition is not available for $\mathbb B_{k}^{\div}(K;\bST)$, since its elements vanish identically on every edge.  

As a consequence, our proof of Theorem~\ref{thm:div-surjective property} necessarily takes a different route. Instead of the \emph{working with $\ran(\div)$} strategy, which analyzes the image of the divergence, we adopt a \emph{working with $\ker(\div)$} approach: we study $\ker(\div)$ and construct a family of exact \emph{bubble   
  conformal complexes}, in the spirit of the Stokes pair and the Arnold–Awanou–Winther element construction.
\end{remark}
Once the bubble complexes are in place, the usual interface–bubble matching techniques developed for the finite element Hessian, elasticity, and divdiv complexes can be extended to assemble a full conforming sub‑complex of the conformal sequence.  The essential challenge is to characterize the $H(\cott)$ bubble spaces; even in elasticity complex, the $H(\inc)$ bubbles associated with the second‑order operator $\inc$ are already notoriously intricate.

\medskip
{Our construction follows the following steps:}
\[
\boxed{\text{continuous complex}}
\;\xrightarrow{}\;
\boxed{\text{trace complex}}
\;\xrightarrow{}\;
\boxed{\text{bubble complex}}
\;\xrightarrow{}\;
\boxed{\text{finite element complex}}.
\]
The results obtained at each stage are summarized below.

\subsubsection{Conformal trace complexes and bubble complexes}
We begin by demonstrating a local integration-by-parts identity for the operator
\(\cott\) on a tetrahedron \(K\) for $\bST$ tensors:
\begin{equation}
    \begin{aligned}
          \int_{K}\cott \boldsymbol{\sigma}: \boldsymbol{\tau}
  - \int_{K}\cott \boldsymbol{\tau}: \boldsymbol{\sigma}
  ={}& -\int_{\partial K}\tr_1^{\cott}(\boldsymbol{\sigma}):\Pi_F\inc\boldsymbol{\tau}\Pi_F
        -\int_{\partial K}\tr_2^{\cott}(\boldsymbol{\sigma}):\tr_{2}^{\inc}(\boldsymbol{\tau}) \\
      &\quad + \int_{\partial K}\tr_3^{\cott}(\boldsymbol{\sigma}):\tr_{1}^{\inc}(\boldsymbol{\tau})
        \;+\; \text{edge integrals}.
    \end{aligned}
    \label{eq:integration-by-parts}
\end{equation}

This formula involves three boundary traces associated with the conformal complex, i.e.,%
\(\tr_{1}^{\cott}\), \(\tr_{2}^{\cott}\), and \(\tr_{3}^{\cott}\), with differential operators of orders~0,~1, and~2, respectively, these traces also form commutative trace complexes: 
\[\tr_1^{\cott}(\dev\Def \boldsymbol{u})=\tfrac{1}{2}\Def_F(\boldsymbol{u}\times\boldsymbol{n})-\tfrac{1}{2}\sym\curl_F(\boldsymbol{u}\Pi_F),\]
\[ \begin{tikzcd}[column sep = large]
\boldsymbol{u} \arrow{r}{\dev\Def} \arrow{d}{} & \boldsymbol{\sigma} \arrow{d}{} \arrow{r}{\cott}&\boldsymbol{\tau}\arrow{d}{}\\%
\boldsymbol{u}\cdot\boldsymbol{n}\arrow{r}{-\Def_F\curl_F}& \tr_2^{\cott}(\boldsymbol{\sigma})\arrow{r}{-\div_F\div_F}&\boldsymbol{n}\cdot\boldsymbol{\tau}\cdot\boldsymbol{n},
\end{tikzcd}
\]
\[ \begin{tikzcd}[column sep = large]
\boldsymbol{u} \arrow{r}{\dev\Def} \arrow{d}{} & \boldsymbol{\sigma} \arrow{d}{} \arrow{r}{\cott}&\boldsymbol{\tau}\arrow{d}{}\\%
\boldsymbol{u}\times\boldsymbol{n}\arrow{r}{\tfrac{1}{2}\operatorname{hess}_F\div_F}& \tr_3^{\cott}(\boldsymbol{\sigma})\arrow{r}{\operatorname{rot}_F}&\boldsymbol{n}\times\boldsymbol{\tau}\cdot\boldsymbol{n}.
\end{tikzcd}
\]  
Precise descriptions are deferred to Section~\ref{trace section}.

Building upon these trace complexes, we construct three nested \(H(\cott)\) bubble spaces and establish the exact sequences stated below. Due to trace complexes, it is clear that the following sequences are complexes:

\begin{theorem}
    [Bubble conformal complexes]\label{thm:bubble_conformal}
Let \(k\ge 10\) and let \(K\) denote a tetrahedron with a quartic bubble function \(b_{K}\).  
The \(H(\cott)\) bubble spaces
 \begin{align*}    \mathbb{B}^{\cott}_k\left(K;\bST\right)&:=\left\{\boldsymbol{\sigma}\in P_k\left(K;\bST\right):\tr_1^{\cott}\left(\boldsymbol{\sigma}\right)|_F=\tr_2^{\cott}\left(\boldsymbol{\sigma}\right)|_F=\tr_3^{\cott}\left(\boldsymbol{\sigma}\right)|_F=0 {\text{ for } F\in\mathcal{F}(K)}\right\},\\
    {\mathbb{B}}^{1\cott}_{k-4}\left(K;\bST\right)&:=\left\{{\boldsymbol{\sigma}}\in P_{k-4}\left(K;\bST\right):b_K\boldsymbol{\sigma}\in \mathbb{B}^{\cott}_k\left(K;\bST\right)\right\},\\
    {\mathbb{B}}^{2\cott}_{k-8}\left(K;\bST\right)&:=\left\{{\boldsymbol{\sigma}}\in P_{k-8}\left(K;\bST\right):b_K^2\boldsymbol{\sigma}\in \mathbb{B}^{\cott}_k\left(K;\bST\right)\right\},
\end{align*}  satisfy
\[
  b_{K}^{2}\mathbb B_{k-8}^{2\cott}(K;\bST)
  \;\subset\;
  b_{K}\,\mathbb B_{k-4}^{1\cott}(K;\bST)
  \;\subset\;
  \mathbb B_{k}^{\cott}(K;\bST)
\]
and the following sequences are exact:
\begin{align}
  0
  \xrightarrow{}
     b_{K}P_{k-3}(K;\mathbb R^{3})
   \xrightarrow{\dev\Def}
     \mathbb B_{k}^{\cott}(K;\bST)
   \xrightarrow{\cott}
     \mathbb B_{k-3}^{\div}(K;\bST)\cap\ker(\div)
   \xrightarrow{\div}
     0&, \label{eq:bubble-seq0}\\
  0
  \xrightarrow{}
     b_{K}^{2}P_{k-7}(K;\mathbb R^{3})
   \xrightarrow{\dev\Def}
     b_{K}\,\mathbb B_{k-4}^{1\cott}(K;\bST)
   \xrightarrow{\cott}
     \mathbb B_{k-3}^{\div}(K;\bST)\cap\ker(\div)
   \xrightarrow{\div}
     0&, \label{eq:bubble-seq1}\\
     \label{eq:bubble-seq2}
  0
  \xrightarrow{}
     b_{K}^{3}P_{k-11}(K;\mathbb R^{3})
   \xrightarrow{\dev\Def}
     b_{K}^{2}\mathbb B_{k-8}^{2\cott}(K;\bST)
   \xrightarrow{\cott}
     \mathbb B_{k-3}^{\div}(K;\bST)\cap\ker(\div)
   \xrightarrow{\div}
     0&. 
\end{align}
\end{theorem}
Proof of Theorem \ref{thm:bubble_conformal} is divided  into three parts: Theorem \ref{thm:conformal-bubble-seq3} (for \eqref{eq:bubble-seq2}), Theorem \ref{thm:bubble-seq-0} (for \eqref{eq:bubble-seq0}) and Corollary \ref{cor:bubble-conformal-seq-2} (for \eqref{eq:bubble-seq1}). The main tool used in the proof of Theorem~\ref{thm:bubble_conformal} is the bubble-complex analog of the BGG construction~\cite{vcap2001bernstein,arnold2021complexes}. Compared with the continuous setting, applying BGG to bubble complexes is more delicate, since special care must be taken to handle vanishing traces and supersmoothness. Our approach begins with the exact bubble de~Rham complex~\cite{arnold2006finite} and bubble Stokes complex~\cite{neilan2015discrete}, together with their smoother variants. In a systematic and structured way, we recover the previously established bubble elasticity complex~\cite{arnold2008finite,hu2015family} and bubble divdiv complex~\cite{chen2022finitedivdiv,hu2022conformingdivdiv} using BGG machinery, and then employ these complexes to construct a BGG diagram proving the exactness of~\eqref{eq:bubble-seq2}. We further establish that $\ker(\cott) = \ran(\dev\Def)$ in~\eqref{eq:bubble-seq0}, which in turn yields the exactness of all three bubble conformal complexes. We also note that, during the revision of this work, related results on BGG constructions for conformal bubble complexes with the Hessian and elasticity complexes have appeared~\cite{huang2025finite}, which differ from our approach but highlight the growing interest in constructing finite element BGG complexes with bubbles.

Among the three spaces, $\mathbb B_{k}^{\cott}(K;\bST)$ is the most intricate, as it consists of piecewise polynomials whose three traces vanish on $\partial K$. By contrast, the smallest space $b_{K}^{2}\mathbb B_{k-8}^{2\cott}(K;\bST)$ has a particularly simple structure: the factor $b_K^2$ automatically annihilates the first two traces ($\tr_1^{\cott}$ and $\tr_2^{\cott}$) on $\partial K$. This simplicity allows a clean geometric decomposition, which facilitates both the explicit construction of a basis and the computation of dimensions. As a direct consequence, we can prove Theorem~\ref{thm:div-surjective property} from the exactness of~\eqref{eq:bubble-seq2}.
\begin{proof}[Sketch of the proof of Theorem \ref{thm:div-surjective property}]
From the supersmoothness of the bubble function $b_K$ at the vertices, we deduce that 
\[
  b_{K}^{2}\mathbb B_{k-8}^{2\cott}(K;\bST)
  \;\subset\;
  P_{k}^{(6)}(K;\bST).
\]
By the exactness of the third bubble conformal complex, it follows that
\[
  \mathbb B_{k-3}^{\div}(K;\bST)\cap\ker(\div)
  \;=\;
  \cott\bigl(b_{K}^{2}\mathbb B_{k-8}^{2\cott}(K;\bST)\bigr)
  \;\subset\;
  P_{k-3}^{(3)}(K;\bST).
\]
In particular, this shows that 
\[
  \mathbb B_{k-3}^{\div,(s)}(K;\bST)\cap\ker(\div)
  \;=\;
  \mathbb B_{k-3}^{\div}(K;\bST)\cap\ker(\div), 
  \qquad 1\le s\le 3.
\]

Consequently, a dimension-counting argument gives,
\begin{align*}
    \dim \ran(\div)
    &= \dim \mathbb B_{k-3}^{\div,(s)}(K;\bST)
      - \dim \ker(\div) \\
    &= \dim \mathbb B_{k-3}^{\div,(s)}(K;\bST)
       - \dim \mathbb B_{k-8}^{2\cott}(K;\bST)
       + \dim P_{k-11}(K;\mathbb R^3),
\end{align*}
where the last equality follows from the exactness of the bubble conformal complex and can be computed explicitly via geometric decompositions.  
The claim then follows by comparing $\dim \ran(\div)$ with $\dim\!\left(P_{k-4}^{(s-1)}(K;\mathbb R^{3})\cap\boldsymbol{CK}^{\perp}\right)$.
\end{proof}

\begin{remark}
The full proof is given in Section \ref{div section}. This result is somewhat unexpected: although the space $\mathbb B_{k-3}^{\div}(K;\bST)$ only vanishes at the vertices up to first order, its divergence-free kernel
\[
  \mathbb B_{k-3}^{\div}(K;\bST)\cap\ker(\div)
\]
in fact vanishes up to third order at vertices. Notably, this supersmoothness cannot be inferred directly from \eqref{eq:bubble-seq0}, since elements of $\mathbb B_k^{\cott}(K;\bST)$ vanish at vertices only to first order (see Lemma~\ref{edge vanish 1}).  
Thus, Theorem~\ref{thm:bubble_conformal} reveals an intrinsic supersmoothness phenomenon arising from the smoothest bubble complex \eqref{eq:bubble-seq2}. In this sense, the inclusion of $\mathbb B_{k-3}^{\div,(3)}(K;\bST)$ into the (exact) bubble conformal complex is natural, as its vertex smoothness precisely matches the structural requirements of the sequence. 
\end{remark}

\subsubsection{Finite element complex and inf–sup stability}
Guided by Theorem~\ref{thm:div-surjective property}, we select
\(\boldsymbol{\Sigma}_{k,h}^{\div}\) and \(\boldsymbol V_{k-1,h}\) in such a manner  
that their interface DOFs correspond  
to the smoothness of \(\mathbb B_{k}^{\div,(3)}(K;\bST)\) and
\(P_{k-1}^{(2)}(K;\mathbb R^{3})\), respectively:
\begin{itemize}
        \item The $P_k(\bST)$ $H(\div)$-conforming space $\boldsymbol{\Sigma}_{k,h}^{\div}$ is $C^3$ at vertices, $C^0$ along internal edges, and has face-normal continuity  across internal interfaces.
        \item The $P_{k-1}(\mathbb R^3)$ displacement space \(\boldsymbol V_{k-1,h}\) has $C^2$ continuity at vertices.
    \end{itemize} Section~\ref{div section} establishes the inf–sup stability of this pair using a
standard Neilan-type argument. Subsequently, we extend the pair to a complete finite element conformal complex:
\begin{claim}
There exists a finite element sub-complex of \eqref{eq:sec1:conformal-complex}
  \begin{align}
            \boldsymbol{CK} \xrightarrow[]{\subset} \boldsymbol{U}_{k+1,h}\xrightarrow{\dev \Def} \boldsymbol{\Sigma}^{\cott}_{k,h} \xrightarrow[]{\cott} \boldsymbol{\Sigma}^{\div}_{k-3,h}\xrightarrow[]{\div} \boldsymbol{V}_{k-4,h}\xrightarrow[]{}0.
            \label{eq:sec1:finite element complex 1}
 \end{align}
The sequence is exact if the domain $\Omega$ is contractible. 
    \label{claim1}
\end{claim}
  The finite element complex \eqref{eq:sec1:finite element complex 1} is constructed via enriching the first bubble complex \eqref{eq:bubble-seq0} (specifically, the variant in 
Theorem~\ref{conformal_theorem_bubble_extra}, which adds extra vertex smoothness) with   
interface DOFs. We note that to ensure the exact sequence property and match the vertex smoothness of the last two spaces, vertex continuity in~\eqref{eq:sec1:finite element complex 1} is designed as
(from left to right):
\[
C^{7}\;\; \to\;\; C^{6}\;\; \to\;\; C^{3}\;\; \to\;\; C^{2},
\]
which matches the vertex smoothness of the bubble complex in Theorem~\ref{conformal_theorem_bubble_extra}. For the \(H^{1}\)-conforming space \(\boldsymbol{U}_{k+1,h}\), we adopt the
Neilan Stokes velocity element~\cite{neilan2015discrete}  with enhanced \(C^{7}\)
vertex smoothness.
Interface DOFs for $\boldsymbol{\Sigma}^{\cott}_{k,h}$ are guided by a sufficient $H(\cott)$-conforming condition following \eqref{eq:integration-by-parts}. Drawing inspiration from the work on finite element elasticity complexes \cite{chen2022finite} and divdiv complexes \cite{hu2022conformingdivdiv,chen2022finitedivdiv}, we design the interface DOFs to ensure $\tr_1^{\cott}$, $\tr_2^{\cott}$, and $\tr_3^{\cott}$ are single-valued across faces by employing the interface+bubble paradigm once more for two-dimensional bubbles and elements, which are guided by trace complexes we discussed earlier.

To prove exactness on contractible domain, we show $\boldsymbol{\Sigma}^{\cott}_{k,h}\cap \ker(\cott)=\dev\Def\boldsymbol{U}_{k+1,h}$ by dealing with extra smoothness on vertices and edges. Since the last two spaces are balanced by design, the exactness of the full finite element complex is proven via checking Euler's identity.

\section{Preliminaries and Notations}
\label{sec:notations}
\subsection{Triangulations}
We denote $\mathcal{T}_h$ as a shape-regular tetrahedral triangulation of a contractible domain $\Omega$ with Lipschitz boundary $\partial \Omega$ in $\mathbb{R}^3$. For a tetrahedron $K \in \mathcal{T}_h$, let $\mathcal{F}(K)$, $\mathcal{E}(K)$, and $\mathcal{V}(K)$ denote the sets of its faces, edges, and vertices, respectively. Analogously, for a face $F \in \mathcal{F}(K)$, we let $\mathcal{E}(F)$ and $\mathcal{V}(F)$ denote the sets of its edges and vertices.

On a face $F$, we denote $\boldsymbol{n}$ as the unit outward normal vector. $\boldsymbol{t}_{F,e}$ is the unit tangential vector on $e$, with its direction induced by the orientation of $\boldsymbol{n}$. We also use $\boldsymbol{t}_e$ when the direction is not an important factor. Denote by $\boldsymbol{n}_{e\pm}$ the two orthogonal unit vectors perpendicular to the tangential vector $\boldsymbol{t}_e$ of an edge $e$. Let $\boldsymbol{n}_{F,e}$ be the outward pointing face unit normal vector on $e$. It can also be written as $\boldsymbol{n}_{F,e}=\boldsymbol{t}_{F,e}\times\boldsymbol{n}$. See Figure \ref{tetrahedron} for an illustration.
\begin{figure}[H]
    \centering
    \begin{tikzpicture}[line join = round, line cap = round, scale=0.5]
\pgfmathsetmacro{\factor}{1/sqrt(2)}; 
\coordinate  (A) at (3,0,0);
\coordinate  (B) at (-3,0,0);
\coordinate (C) at (0,4,0);
\coordinate (D) at (2,0,4*\factor);
\coordinate (MCD) at (1,2,2*\factor);
\coordinate (NF) at (-0.8133,
   +0.4880,
   +1.2651);
\coordinate (N) at ( 1.2172  ,  0.9129    ,0.4304);
\coordinate (Normal) at (2.8838,    2.2463  ,  3.7302);
\coordinate (ABC) at (barycentric cs:A=1,B=1,C=1);
\coordinate (ABD) at (barycentric cs:A=1,B=1,D=1);
\coordinate (ACD) at (barycentric cs:A=1,C=1,D=1);

\coordinate [label=right:$\boldsymbol{n}$] (ACD_4) at (barycentric cs:A=1,C=0.9,D=-0.1);
\coordinate [label=below:$\boldsymbol{n}_{F,e}$] (BCD) at (barycentric cs:B=1,C=1.5,D=1);
\coordinate  (ACD_1) at (barycentric cs:A=0,C=4,D=4);
\coordinate [label=below:$\boldsymbol{t}_{F,e}$] (ACD_5) at (barycentric cs:A=-1,C=1.5,D=5);
\coordinate (DC_1) at (0.5976,   -1.1952  ,  0.8452);
\foreach \i in {A,B}
    \draw[dashed,line width=0.3pt] (0,0)--(\i);
\draw[-, fill=black!10, opacity=0.5,line width=0.3pt] (A) --(D)--(C)--cycle;
\draw[-, opacity=5,line width=0.3pt] (A) --(D)--(C)--cycle;
\draw[-, opacity=5,line width=0.3pt] (B)--(D)--(C)--cycle;

\draw[-latex, line width=1.5pt] (ACD) -- ++(N);
\draw[-latex,line width=1.5pt] (MCD) -- ++(NF);
\draw[-latex,line width=1.5pt] (ACD_1) -- ++(DC_1);
\coordinate [label=right:$F$] (ACD_2) at (barycentric cs:A=1,C=0.7,D=1);
\end{tikzpicture}

    \caption{Vector notations on a tetrahedron $K$}
    \label{tetrahedron}
\end{figure}
Denote $\lambda_F$ as the corresponding barycentric coordinate of a face $F$. The tetrahedron bubble is defined as $b_K = \prod_{F\in\mathcal{F}(K)}\lambda_F$. For face $F$, the face bubble is given by $b_F = b_K/\lambda_F$. In the case of an edge $e = F_+ \cap F_-$, the edge bubble is defined as $b_e = b_K/(\lambda_{F_+}\lambda_{F_-})$. Note that $\operatorname{grad}\left(\lambda_F\right)=-\boldsymbol{n}/h_F$, where $h_F$ is the distance between a face $F$ and its opposed vertex. 
\subsection{Tensor operations}
Let $\mathbb{M}$ denote the space of $3 \times 3$ real matrices. Define the subspaces:
\begin{itemize}
    \item $\mathbb{K} \subset \mathbb{M}$: skew-symmetric matrices,
    \item $\mathbb{S} \subset \mathbb{M}$: symmetric matrices,
    \item $\mathbb{T} \subset \mathbb{M}$: traceless matrices,
    \item $\mathbb{S} \cap \mathbb{T}$: symmetric, traceless matrices.
\end{itemize}

We summarize some operators in the following table:
\renewcommand{\arraystretch}{1.3}
\begin{longtable}{lll}
\caption{Algebraic operators on matrices and vectors} \\
\toprule
\textbf{Operator} & \textbf{Domain $\to$ Codomain} & \textbf{Definition} \\
\midrule
\endfirsthead

\multicolumn{3}{l}{\emph{(continued from previous page)}} \\
\toprule
\textbf{Operator} & \textbf{Domain $\to$ Codomain} & \textbf{Definition} \\
\midrule
\endhead

\midrule
\multicolumn{3}{r}{\emph{(continued on next page)}} \\
\endfoot

\bottomrule
\endlastfoot

$\Tr$         & $\mathbb{M} \to \mathbb{R}$     & $\Tr(\boldsymbol{U}) = \sum_i \boldsymbol{U}_{ii}$ \\
$\sym$        & $\mathbb{M} \to \mathbb{S}$     & $\sym(\boldsymbol{U}) = \tfrac{1}{2}(\boldsymbol{U} + \boldsymbol{U}^T)$ \\
$\skw$        & $\mathbb{M} \to \mathbb{K}$     & $\skw(\boldsymbol{U}) = \tfrac{1}{2}(\boldsymbol{U} - \boldsymbol{U}^T)$ \\
$\mskw$       & $\mathbb{R}^3 \to \mathbb{K}$   & $[\mskw(\boldsymbol{a})]_{ij} = -\varepsilon_{ij\ell} \boldsymbol{a}_\ell$ \\
$\vskw$       & $\mathbb{M} \to \mathbb{R}^3$   & $\vskw = \mskw^{-1} \circ \skw$ \\
$\dev$           & $\mathbb{M} \to \mathbb{T}$     & $\dev(\boldsymbol{U}) = \boldsymbol{U} - 1/3(\Tr \boldsymbol{U}) \mathbf I$ \\
$S$           & $\mathbb{M} \to \mathbb{M}$     & $S(\boldsymbol{U}) = \boldsymbol{U}^T - (\Tr \boldsymbol{U}) \mathbf I$ \\
$\iota$           & $\mathbb{R} \to \mathbb{M}$     & $\iota({a}) =  a \mathbf I$ \\
\end{longtable}
We note that throughout the paper, the boundary trace operators are always in lowercase notation $\tr$ with appropriate superscripts and subscripts, while the matrix trace operator remains capitalized as $\Tr$. 

We use the following tensor-vector operations:
\begin{itemize}
    \item Double contraction: $\boldsymbol{a} \cdot \boldsymbol{U} \cdot \boldsymbol{b} := \boldsymbol{a}^T \boldsymbol{U} \boldsymbol{b}$,
    \item Column-wise cross product: $\boldsymbol{a} \times \boldsymbol{U}:=\mskw(\boldsymbol{a})\, \boldsymbol{U}$,
    \item Row-wise cross product: $\boldsymbol{U} \times \boldsymbol{a}:=\boldsymbol{U}\, \mskw(\boldsymbol{a})=- (\boldsymbol{a} \times \boldsymbol{U}^T)^T$.
\end{itemize}

We adopt standard notation for differential operators. For vector fields $\boldsymbol{a}$, the row-wise gradient is defined as $\grad \boldsymbol{a} := \boldsymbol{a} \nabla^T$ and the symmetric gradient is defined $\Def \boldsymbol{a} := \sym\grad \boldsymbol{a}$. The operators $\div$ and $\curl$ act row-wise on matrix fields. In our notations, $\curl \mathbf {U}=-\boldsymbol{U}\times \nabla$ for matrices.

$\Pi_F := \mathbf I - \boldsymbol{n} \boldsymbol{n}^T=-\mskw(\boldsymbol{n})^2$ denotes the surface orthogonal projection. Surface differential operators such as $\grad_F$, $\div_F$, and $\Def_F$ follow the same form, with $\nabla$ replaced by the tangential gradient $\nabla_F := \Pi_F \nabla$.

The tangential curl is defined as $\nabla_F^\perp := \boldsymbol{n} \times \nabla$ and $\curl_F \boldsymbol{a} := \boldsymbol{a} \nabla_F^{\perp T}$. The surface rot is denoted by $\operatorname{rot}_F \boldsymbol{a} := \nabla_F^{\perp}\cdot\boldsymbol{a}$. Analogous to $\div_F$, $\operatorname{rot}_F$ acts row-wise on matrices. The surface Hessian is denoted by $\hess_F(a) := \nabla_F^2 a$.

A key identity relating to double cross products is:
\begin{align}
    \boldsymbol{n} \times (\nabla \times \boldsymbol{a}) 
    = \nabla(\boldsymbol{n} \cdot \boldsymbol{a}) - \partial_{\boldsymbol{n}} \boldsymbol{a}
    = \nabla_F(\boldsymbol{n} \cdot \boldsymbol{a}) - \partial_{\boldsymbol{n}}(\Pi_F \boldsymbol{a}).
    \label{cross_product_identity_1}
\end{align}
\subsection{BGG diagrams} 
The elasticity complex and divdiv complex can be obtained from the following {\emph{Bernstein--Gelfand--Gelfand}} (BGG) diagram \cite{arnold2021complexes}:
\begin{equation}\label{eq:bgg-construction-elasticity}
    \begin{tikzcd}[column sep=2.5em,row sep=2.5em]
0 \arrow[r] & H^q\!\otimes\!\mathbb{R}^3 \arrow[r,"\operatorname{grad}"] &
H^{q-1}\!\otimes\!\mathbb{M} \arrow[r,"\operatorname{curl}"] &
H^{q-2}\!\otimes\!\mathbb{M} \arrow[r,"\operatorname{div}"] &
H^{q-3} \otimes \mathbb R^3\arrow[r] & 0 \\[2pt]
0 \arrow[r] & H^{q-1}\!\otimes\!\mathbb{R}^3 \arrow[r,"\operatorname{grad}"]
\arrow[ru,"\mskw"] &
H^{q-2}\!\otimes\!\mathbb{M} \arrow[r,"\operatorname{curl}"] \arrow[ru,"S"] &
H^{q-3}\!\otimes\!\mathbb{M} \arrow[r,"\operatorname{div}"] \arrow[ru,"-2\operatorname{vskw}"] &
H^{q-4}\!\otimes\!\mathbb{R}^3 \arrow[r] & 0 .\\
0 \arrow[r] & H^{q-2}\!\otimes\!\mathbb{R}\arrow[r,"\operatorname{grad}"]
\arrow[ru,"\iota"] & H^{q-3}\!\otimes\!\mathbb{R} 
\arrow[r,"\operatorname{curl}"] \arrow[ru,"\mskw"] & H^{q-4}\!\otimes\!\mathbb{R}^3 
 \arrow[r,"\operatorname{div}"] \arrow[ru,"\operatorname{id}"] &
H^{q-5}\!\otimes\!\mathbb{R} \arrow[r]& 0.
\end{tikzcd}
\end{equation}
A crucial property is that each block of this diagram anti-commutes. For example, the elasticity follows from a ``zigzag'' path of the top two rows, where $\mskw$ is an injection, $S$ is a bijection and $-2\vskw$ is a surjection. If we restrict our third slot in the top row as $\bS$ ($\ran(\mskw)^{\perp}$) and, the fourth slot in the second row also as $\bS$ ($\ker(\vskw)$), we have (see details at \cite{arnold2021complexes}):
\begin{equation}
    \begin{tikzcd}[column sep=2.5em,row sep=2.5em]
0 \arrow[r] & H^q\!\otimes\!\mathbb{R}^3 \arrow[r,"\Def"] &
H^{q-1}\!\otimes\!\mathbb{S} \arrow[r,"\operatorname{curl}"] &
{} &
 &  \\[2pt]
& &{}
\arrow[r,"\operatorname{curl}"] \arrow[ru,"S"] &
H^{q-3}\!\otimes\!\mathbb{S} \arrow[r,"\operatorname{div}"]  &
H^{q-4}\!\otimes\!\mathbb{R}^3 \arrow[r] & 0 ,
\end{tikzcd}
\end{equation}
and from which we get the elasticity complex: 
\begin{equation}
    \begin{tikzcd}
        0 \arrow[r] & H^{q}\!\otimes\!\mathbb{R}^3 \arrow[r,"\operatorname{def}"]
&
H^{q-1}\!\otimes\!\mathbb{S} \arrow[r,"\operatorname{inc}"]  &
H^{q-3}\!\otimes\!\mathbb{S} \arrow[r,"\operatorname{div}"]  &
H^{q-4}\!\otimes\!\mathbb{R}^3 \arrow[r] & 0 .
    \end{tikzcd}
\end{equation}
Here,\begin{align}
         \inc\, \boldsymbol{U} := \curl\, S^{-1} \curl\, \boldsymbol{U}.\label{def:inc}
 \end{align}
    denotes the \emph{incompatibility operator}. If $\boldsymbol{U}$ is symmetric, then $\Tr(\curl \boldsymbol{U}) = -\Tr( \boldsymbol{U}\mskw(\nabla))= 0$. So $\inc\, \boldsymbol{U} = \curl\left((\curl \boldsymbol{U})^T\right)$ remains symmetric. If $\boldsymbol{U}$ is skew-symmetric, i.e., $\boldsymbol{U} = \mskw(\boldsymbol{u})$, then from anti-commutativity of \eqref{eq:bgg-construction-elasticity}: $\curl \mskw=-S\grad$, $\inc\, \boldsymbol{U} = -\curl\grad \boldsymbol{u} =0$. Thus, $\inc$ maps general matrix fields to symmetric tensors, with skew-symmetric tensors lying in its kernel.

If we iterate the BGG construction for  divdiv complex and elasticity complex, we get another diagram:
\begin{equation}\label{eq:bgg-diagrams-conformal}
    \begin{tikzcd}[column sep=2.5em,row sep=2.5em]
0 \arrow[r] & H^q\!\otimes\!\mathbb{R}^3 \arrow[r,"\operatorname{dev\,grad}"] &
H^{q-1}\!\otimes\!\mathbb{T} \arrow[r,"\operatorname{sym\,curl}"] &
H^{q-2}\!\otimes\!\mathbb{S} \arrow[r,"\operatorname{div\,div}"] &
H^{q-4} \otimes \mathbb R\arrow[r] & 0 \\[2pt]
0 \arrow[r] & H^{q-1}\!\otimes\!\mathbb{R}^3 \arrow[r,"\operatorname{def}"]
\arrow[ru,"\mskw"] &
H^{q-2}\!\otimes\!\mathbb{S} \arrow[r,"\operatorname{inc}"] \arrow[ru,"S"] &
H^{q-4}\!\otimes\!\mathbb{S} \arrow[r,"\operatorname{div}"] \arrow[ru,"\Tr"] &
H^{q-5}\!\otimes\!\mathbb{R}^3 \arrow[r] & 0 .
\end{tikzcd}
\end{equation}
Applying the BGG machinery, we get
\begin{equation}
    \begin{tikzcd}[column sep=2.5em,row sep=2.5em]
0 \arrow[r] & H^q\!\otimes\!\mathbb{R}^3 \arrow[r,"\dev\Def"] &
H^{q-1}\!\otimes\!\bST \arrow[r,"\sym\operatorname{curl}"] & {}
  &
 &  \\[2pt]
& &
{}\arrow[r,"\operatorname{inc}"] \arrow[ru,"S"] &
H^{q-4}\!\otimes\!\bST \arrow[r,"\operatorname{div}"]  &
H^{q-5}\!\otimes\!\mathbb{R}^3 \arrow[r] & 0 .
\end{tikzcd}
\end{equation}
This yields the conformal deformation complex:
\begin{equation}
    \begin{tikzcd}
        0 \arrow[r] & H^{q}\!\otimes\!\mathbb{R}^3 \arrow[r,"\dev\operatorname{def}"]
&
H^{q-1}\!\otimes\!\bST \arrow[r,"\operatorname{cott}"]  &
H^{q-4}\!\otimes\!\bST \arrow[r,"\operatorname{div}"]  &
H^{q-5}\!\otimes\!\mathbb{R}^3 \arrow[r] & 0 ,
    \end{tikzcd}
\end{equation}
where $\cott$ denotes the \emph{linearized Cotton--York tensor}:
    \begin{align}
            \cott\, \boldsymbol{U} :
    = \inc\, S^{-1} \sym \curl\, \boldsymbol{U} = \curl\, S^{-1} \curl\, S^{-1} \curl\, \boldsymbol{U}.\label{def:cott}
    \end{align}
 If $\boldsymbol{U}$ is skew-symmetric (in $\ker(\inc)$) or a scalar multiple of the identity (in $\ker(\sym\curl)$), then $\cott\, \boldsymbol{U} = 0$. If $\boldsymbol{U}$ is symmetric and traceless, then $\sym \curl \boldsymbol{U}$ is also symmetric and traceless. From anti-commutativity of \eqref{eq:bgg-diagrams-conformal}, we have $\Tr(\cott\, \boldsymbol{U}) = -\div\div \sym\curl \boldsymbol{U} = 0$. Hence, $\cott$ maps symmetric and traceless tensors to symmetric and traceless tensors, with kernel containing all skew-symmetric tensors and scalar multiples of the identity matrix. This operator also occurs as $\operatorname{cinc}$ \cite{arnold2021complexes} and $\operatorname{cot}$ \cite{vcap2023bgg, bonizzoni2023discrete} in the literature.
\subsection{Polynomials with vanishing derivatives at vertices} 

Since we shall work with finite elements with high-order vertex regularity, 
we introduce the geometric decomposition of polynomials with vanishing derivatives at vertices of $K$ to aid the construction of finite element spaces. Denote by $P_k(e;\mathbb{X})$ and $P_k(F;\mathbb{X})$ polynomial subspaces of $P_k(K;\mathbb{X})$ spanned by the canonical Lagrange basis of $P_k(K;\mathbb{X})$ on $e\in\mathcal{E}(K)$ and $F\in\mathcal{F}(K)$ respectively.
\begin{proposition}\label{prop:geometric_decompositions}
    Let $s$ and $k$ be non-negative integers such that $k\geq 2s+1$. Then $ P_k^{(s)}(K;\mathbb{X})$ can be decomposed geometrically as follows:
    \begin{equation}
    \begin{aligned}
         P_k^{(s)}(K;\mathbb{X})=&\oplus_{e\in\mathcal{E}(K)}b_e^{s+1}P_{k-2s-2}(e;\mathbb{X})\\
  &\oplus_{F\in\mathcal{F}(K)}b_F P_{k-3}^{(s-2)}(F;\mathbb{X})\\
  &\oplus b_K P_{k-4}^{(s-3)}(K;\mathbb X),
    \end{aligned}
    \label{geo decomp}
\end{equation}
where $P^{(s)}_k(F;\mathbb{X})$ is defined as:
    \begin{equation}
    \begin{aligned}        P^{(s)}_k(F;\mathbb{X}):=&\oplus_{e\in\mathcal{E}(F)}b_e^{s+1}P_{k-2s-2}(e;\mathbb{X})\\
&\oplus b_FP_{k-3}^{(s-2)}(F;\mathbb{X}).
    \end{aligned}
    \label{geo decomp F}
\end{equation}
Here we set $P_{k}^{(q)}(F;\mathbb X):=P_{k}(F;\mathbb X)$ and $P_{k}^{(q)}(K;\mathbb X):=P_{k}(K;\mathbb X)$ for $q< 0$. 
\end{proposition}
\begin{proof}
The inclusion ``$\supset$'' follows directly from the recursive construction. For the reverse inclusion, we perform a dimension count. By induction, it can be shown that:
\begin{align}
\label{eq:dim_face_polynomials_vanishing_at_vertices} \dim P_k^{(s)}(F;\mathbb{X}) &= \dim P_k(F;\mathbb{X}) - \dim\mathbb{X} \cdot 3\binom{s+2}{2}, \\
 \label{eq:dim_rhs}\dim\text{(RHS of \eqref{geo decomp})} &= \dim P_k(K;\mathbb{X}) - \dim\mathbb{X} \cdot 4\binom{s+3}{3}.
\end{align}
This matches the dimension of $P_k^{(s)}(K;\mathbb{X})$, since vertex derivatives $D^\alpha \boldsymbol{\sigma}(\delta)$ are linearly independent for $|\alpha|\le s$ when $k \ge 2s + 1$ \cite{hu2023construction}.
\end{proof}
The following proposition shows that when restricted to face $F$, the recursive definition of $P_k^{(s)}(F;\mathbb{X})$ also characterizes the space of polynomials in $P_k(F;\mathbb{X})$ whose facial derivatives vanish to order $s$ at each vertex of $F$.
\begin{proposition}
  For $k\geq 2s+1$, it holds that 
\begin{align}
        P_k^{(s)}(F;\mathbb{X})|_F=\{\boldsymbol{\sigma}\in P_k(F;\mathbb{X})|_F: D_{F}^{\alpha}\boldsymbol{\sigma}(\delta)=0,\,\forall\delta\in\mathcal{V}(F)\text{ for }\left|\alpha\right|\leq s\}.\label{geometric_decomp_F}
\end{align}
\end{proposition}
\begin{proof}
The inclusion ``$\subset$'' follows from the recursive definition. To show equality, we match dimensions using \eqref{eq:dim_face_polynomials_vanishing_at_vertices}. The vertex derivative DOFs $D_F^\alpha \boldsymbol{\sigma}(\delta)$ for $|\alpha| \le s$ are linearly independent for $k \ge 2s + 1$ \cite{hu2023construction}, completing the argument.
\end{proof}
\begin{remark}
     The spaces \( P_k(e;\mathbb{X}) \) and \( P_k(F;\mathbb{X}) \) denote the subspaces of \( P_k(K;\mathbb{X}) \) via the Lagrange basis from \( K \) with nodes lying on the relevant sub-simplex. Therefore, they are globally defined on the whole $K$. By recursive definition of $P_k^{(s)}(F;\mathbb{X})$, it is also defined on the whole $K$ and is a subspace of $P_k^{(s)}(K;\mathbb{X})$.
\end{remark}
\section{Traces and Trace Complexes}
\label{trace section}
In this section, we are going to analyze the linearized Cotton-York operator $\cott$. We provide trace complexes and establish a sufficient condition for the ${H}(\cott)$-conformity. 

The traces of the linearized Cotton-York operator for $\bST$ tensors are defined as follows: 
 \begin{align}
    & \tr_1^{\cott}(\boldsymbol\sigma):=\sym(\Pi_F\boldsymbol\sigma\times\boldsymbol n),\\
    &\tr_2^{\cott}(\boldsymbol\sigma):= \sym((2\operatorname{def}_F(\boldsymbol{n}\cdot\boldsymbol{\sigma}\Pi_F)-\Pi_F\partial_n\boldsymbol{\sigma}\Pi_F)\times\boldsymbol{n}),\\
    &\tr_3^{\cott}(\boldsymbol\sigma):= 2\operatorname{def}_F(\boldsymbol{n}\cdot\sym\curl\boldsymbol{\sigma}\Pi_F)-\Pi_F\partial_n(\sym\curl\boldsymbol{\sigma})\Pi_F.
\end{align}
\begin{remark}
Two symmetric traces are introduced for the incompatibility operator $\inc$ for symmetric tensors $\boldsymbol \tau$\cite{arnold2008finite, chen2022finite}:
\begin{align}
    &{\tr}^{\inc}_1(\boldsymbol{\tau}):=\boldsymbol{n}\times\boldsymbol{\tau}\times\boldsymbol{n},\label{inc_trace1}\\
&{\tr}^{\inc}_2(\boldsymbol{\tau}):=2\operatorname{def}_F(\boldsymbol{n}\cdot\boldsymbol{\tau}\Pi_F)-\Pi_F\partial_n\boldsymbol{\tau}\Pi_F.\label{inc_trace2}
\end{align}

The newly derived traces for the linearized Cotton-York tensor $\cott$ can be represented as
\begin{align}
    & \tr_1^{\cott}(\boldsymbol{\sigma})=\sym(\tr^{\inc}_1(\boldsymbol{\sigma})\times\boldsymbol{n}),\label{new_representation_1}\\
    & \tr_2^{\cott}(\boldsymbol{\sigma})=\sym(\tr^{\inc}_2(\boldsymbol{\sigma})\times\boldsymbol{n}),\label{new_representation_2}\\
    &
    \tr_3^{\cott}(\boldsymbol{\sigma})=\tr^{\inc}_2(\sym\curl\boldsymbol{\sigma}).\label{new_representation_3}
\end{align}
Since $\tr_1^{\inc}(\boldsymbol{\sigma})$ is symmetric, $\Tr(\tr_1^{\inc}(\boldsymbol{\sigma})\times \boldsymbol{n})= \Tr(\tr_1^{\inc}(\boldsymbol{\sigma})\operatorname{mskw}(\boldsymbol{n}))=\Tr(\operatorname{mskw}(\boldsymbol{n})\tr_1^{\inc}(\boldsymbol{\sigma}))$. But $(\tr_1^{\inc}(\boldsymbol{\sigma})\operatorname{mskw}(\boldsymbol{n}))^T=-\operatorname{mskw}(\boldsymbol{n})\tr_1^{\inc}(\boldsymbol{\sigma})$. Therefore, $\Tr(\tr_1^{\cott}(\boldsymbol{\sigma}))=0$ and similarly $\Tr(\tr_2^{\cott}(\boldsymbol{\sigma}))=0$. 
We conclude that $\tr_1^{\cott}(\boldsymbol{\sigma})$ and $\tr_2^{\cott}(\boldsymbol{\sigma})$ are both symmetric and traceless. Moreover, Lemma~\ref{trace_couple} shows that $$\Tr(\tr_3^{\cott}(\boldsymbol{\sigma}))=\div_F\div_F\tr_1^{\cott}(\boldsymbol{\sigma}).$$
Thus if $\tr_1^{\cott}(\boldsymbol{\sigma})|_F=0$, $\tr_3^{\cott}(\boldsymbol{\sigma})|_F$ is also both symmetric and traceless. 
\end{remark}
\begin{remark}
There is a useful duality to observe here.  
For symmetric tensors $\boldsymbol{\tau}$ one has
\[
\big(S^{-1}(\boldsymbol{\tau}\times\boldsymbol{n})\big)\times\boldsymbol{n}
= -\,\tr^{\inc}_1(\boldsymbol{\tau}),
\]
which may be viewed as the boundary analog of the identity $\inc=\curl S^{-1}\curl$, where the differential operator $\mskw(\nabla)$ is replaced by the algebraic operator $\mskw(\boldsymbol{n})$.  
In the same spirit, for $\bST$ tensor $\boldsymbol{\sigma}$ we find
\[
\boldsymbol{n}\times\big(S^{-1}\sym(\boldsymbol{\sigma}\times\boldsymbol{n})\big)\times\boldsymbol{n}
= \tr^{\cott}_1(\boldsymbol{\sigma}),
\]
again reflecting the same substitution for $\cott =\inc S^{-1}\sym\curl$. 
\end{remark}

The following trace complexes parallel the classical elasticity trace complex \cite{arnold2008finite, chen2022finite}, while we are dealing with symmetric and traceless tensors.
 \begin{theorem}{\rm (Trace Complexes and Commutative Diagrams)}
 \label{trace complexes}
 If $\boldsymbol{\sigma}=\dev\Def\boldsymbol{u}$, It holds that
        \begin{align*}    \tr_1^{\cott}(\boldsymbol{\sigma})=&\tfrac{1}{2}\Def_F(\boldsymbol{u}\times\boldsymbol{n})-\tfrac{1}{2}\sym\curl_F(\boldsymbol{u}\Pi_F).
        \end{align*} 
     Moreover, the following diagrams 
\[ \begin{tikzcd}[column sep = large]
\boldsymbol{u} \arrow{r}{\dev\Def} \arrow{d}{} & \boldsymbol{\sigma} \arrow{d}{} \arrow{r}{\cott}&\boldsymbol{\tau}\arrow{d}{}\\%
\boldsymbol{u}\cdot\boldsymbol{n}\arrow{r}{-\Def_F\curl_F}& \tr_2^{\cott}(\boldsymbol{\sigma})\arrow{r}{-\div_F\div_F}&\boldsymbol{n}\cdot\boldsymbol{\tau}\cdot\boldsymbol{n}
\end{tikzcd}
\]
and
\[ \begin{tikzcd}[column sep = large]
\boldsymbol{u} \arrow{r}{\dev\Def} \arrow{d}{} & \boldsymbol{\sigma} \arrow{d}{} \arrow{r}{\cott}&\boldsymbol{\tau}\arrow{d}{}\\%
\boldsymbol{u}\times\boldsymbol{n}\arrow{r}{\tfrac{1}{2}\operatorname{hess}_F\div_F}& \tr_3^{\cott}(\boldsymbol{\sigma})\arrow{r}{\operatorname{rot}_F}&\boldsymbol{n}\times\boldsymbol{\tau}\cdot\boldsymbol{n}
\end{tikzcd}
\]
commute.{\label{trace_complex_commutative}}
 \end{theorem}
 \begin{proof}
     By the tensor identities in Lemma~\ref{trace complex identities part 1} and Lemma~\ref{trace complex identity part 2}.
 \end{proof}
We also provide a sufficient condition for a piecewise polynomial symmetric traceless tensor field to be ${H}(\cott)$-conforming. This is formalized in the following theorem:
\begin{theorem}
    Suppose $\boldsymbol{\sigma}$ is a function defined on a mesh $\mathcal{T}_h$ such that $\boldsymbol{\sigma}|_K\in P_k(K;\bST)$. Then, $\boldsymbol{\sigma}$ is in ${H}(\cott)$ if the following conditions hold:
\begin{itemize}
    \item $\boldsymbol\sigma|_e$ is single-valued for each edge $e\in$ $\mathcal{E}(\mathcal{T}_h);$
    \item $\left.\tr_1^{\cott}(\boldsymbol{\sigma})\right|_F, \left.\tr_2^{\cott}(\boldsymbol{\sigma})\right|_F, \left.\tr_3^{\cott}(\boldsymbol{\sigma})\right|_F$ are single-valued for each face $F \in \mathcal{F}(\mathcal{T}_h)$. To be specific, suppose two tetrahedrons $K$ and $K'$ share a face F, then $\tr_1^{\cott}(\boldsymbol{\sigma})|_{F,K}=-\tr_1^{\cott}(\boldsymbol{\sigma})|_{F,K'}$, $\tr_2^{\cott}(\boldsymbol{\sigma})|_{F,K}=\tr_2^{\cott}(\boldsymbol{\sigma})|_{F,K'}$ and
    $\tr_3^{\cott}(\boldsymbol{\sigma})|_{F,K}=-\tr_3^{\cott}(\boldsymbol{\sigma})|_{F,K'}$.
\end{itemize}
\label{sufficient conforming condition}
\end{theorem}
This theorem provides a concrete criterion for constructing conforming finite element spaces within the complex involving the linearized Cotton–York operator. The condition is motivated by the integration by parts formula for $\cott$,
\begin{align*}
           \int_K\cott \boldsymbol{\sigma}:\boldsymbol{\tau}-\int_K\cott \boldsymbol{\tau}:\boldsymbol{\sigma}\nonumber
        = &-\int_{\partial K}\tr_1^{\cott}(\boldsymbol{\sigma}):\Pi_F\inc\boldsymbol{\tau}\Pi_F-\int_{\partial K}\tr_2^{\cott}(\boldsymbol{\sigma}):\tr_2^{\inc}(\boldsymbol{\tau})\\
    &\quad+\int_{\partial K}\tr_3^{\cott}(\boldsymbol{\sigma}):\tr_1^{\inc}(\boldsymbol{\tau})+\text{ edge integrals}\nonumber.
\end{align*} While the face continuity conditions follow directly from the structure of the surface integrals, the technical challenge lies in handling the edge terms, which involve subtle coupling between face and edge geometry. Dealing with these terms requires careful analysis of tensor trace identities and compatibility conditions across shared edges.

\begin{remark}
In fact, we can derive various forms of Green's identity for the $\cott$ operator and establish different conforming conditions. The trace complexes in Theorem~\ref{trace complexes} serve as essential guides in selecting the appropriate traces that could fit into a finite element complex. The discretization of trace complexes in two dimensions aids in outlining continuity conditions for finite elements on sub-simplices, as elaborated in Section \ref{cinc section}.
Similarly, trace complexes are also derived in the de~Rham, Hessian, divdiv, and elasticity complexes, and they play a pivotal role in shaping the design of degrees of freedom for complex conforming elements.
\end{remark}
     
\subsection{Proofs} 
\begin{lemma}\label{trace_couple}
 For any sufficiently smooth symmetric and traceless tensor $\boldsymbol{\sigma}$, it holds that 
    \begin{align*}
\Tr(\tr_3^{\cott}(\boldsymbol{\sigma}))=\div_F\div_F\tr_1^{\cott}(\boldsymbol{\sigma}).
    \end{align*}
\end{lemma}
\begin{proof}
    Notice that \begin{align*}
        \Tr({\tr_3}(\boldsymbol{\sigma}))&=2\nabla_F\cdot({\sym\curl\boldsymbol{\sigma}})\cdot\boldsymbol{n}+\boldsymbol{n}\cdot\partial_{\boldsymbol{n}}(\sym\curl\boldsymbol{\sigma})\cdot\boldsymbol{n}\\    &=\nabla_F\cdot\left(\nabla\times\boldsymbol{\sigma}-\boldsymbol{\sigma}\times\nabla\right)\cdot\boldsymbol{n}+\boldsymbol{n}\cdot\partial_{\boldsymbol{n}}\left(\nabla\times\boldsymbol{\sigma}\right)\cdot\boldsymbol{n}.
    \end{align*}
    Since $\nabla_F\cdot\left(\nabla\times\boldsymbol{\sigma}\right)\cdot\boldsymbol{n}+\boldsymbol{n}\cdot\partial_{\boldsymbol{n}}\left(\nabla\times\boldsymbol{\sigma}\right)\cdot\boldsymbol{n}=\nabla\cdot\left(\nabla\times\boldsymbol{\sigma}\right)\cdot\boldsymbol{n}=0
    $, we have
    \begin{align*}       \Tr(\tr_3^{\cott}(\boldsymbol{\sigma}))=\nabla_F\cdot\boldsymbol{\sigma}\cdot\nabla_F^{\perp}=\nabla_F\cdot\left(\Pi_F\boldsymbol{\sigma}\times\boldsymbol{n}\right)\cdot\nabla_F=\div_F\div_F\tr_1^{\cott}(\boldsymbol{\sigma}).
    \end{align*}
    This proves the lemma.
\end{proof}
The following lemma provides useful forms for $\tr_2^{\cott}(\boldsymbol{\sigma})$ and $\tr_3^{\cott}(\boldsymbol{\sigma})$.
\begin{lemma}
\label{new_representation_traces_corollary}
For any sufficiently smooth, symmetric and traceless tensor $\boldsymbol{\sigma}$, it holds that
\begin{align}
{\tr_2^{\cott}}(\boldsymbol{\sigma})  &=\boldsymbol{n} \times(\sym\curl\boldsymbol{\sigma})\times\boldsymbol{n}-\sym\curl_F\left(\Pi_F \boldsymbol{\sigma} \cdot \boldsymbol{n}\right),\label{new_form_trace_2} \\
    {\tr_3^{\cott}}(\boldsymbol{\sigma}) & =\boldsymbol{n} \times\left(\nabla \times\left(\sym\curl \boldsymbol{\sigma}\right)\right )\Pi_F+\left(\Pi_F \left(\sym\curl\boldsymbol{\sigma} \right)\cdot \boldsymbol{n}\right) \nabla_F^T. \label{new_form_trace_3}
\end{align}
\end{lemma}
\begin{proof}
    By definition of $\tr_2^{\inc}$ in \eqref{inc_trace2}, we have
    \begin{align}
        \tr_2^{\inc}(\boldsymbol{\sigma})&=\nabla_F(\boldsymbol{n}\cdot\boldsymbol{\sigma}\Pi_F) +\left(\Pi_F\boldsymbol{\sigma}\cdot\boldsymbol{n}\right)\nabla_F^T -\Pi_F\partial_{\boldsymbol{n}}\boldsymbol{\sigma}\Pi_F\nonumber\\&=\boldsymbol{n}\times\left(\nabla\times\boldsymbol{\sigma}\right)\Pi_F+\left(\Pi_F\boldsymbol{\sigma}\cdot\boldsymbol{n}\right)\nabla_F^T,\label{eq:elasticity_identity}
    \end{align}
    which follows from \eqref{cross_product_identity_1}.
    This identity is derived in the discretizations of the elasticity complex in \cite[Lemma 4.1]{chen2022finite}. Using \eqref{eq:elasticity_identity}, \eqref{new_form_trace_2} and \eqref{new_form_trace_3} are simple consequences of \eqref{new_representation_2} and \eqref{new_representation_3}.
\end{proof}
Next, we link the traces of ${H}(\cott)$ space with well-understood traces of ${H}^1$ and ${H}(\div)$. 
 \begin{lemma}
 \label{trace complex identities part 1}
     For any sufficiently smooth $\mathbb R^3$-valued function $\boldsymbol u$, set $\boldsymbol{\sigma}=\dev\Def\boldsymbol{u}$, it holds that
     \begin{align}\sym\curl\boldsymbol{\sigma}=&\tfrac{1}{2}\Def(\curl\boldsymbol{u})\label{trace identity symcurl},\\
         \tr_1^{\cott}(\boldsymbol{\sigma})=&\tfrac{1}{2}\Def_F(\boldsymbol{u}\times\boldsymbol{n})-\tfrac{1}{2}\sym\curl_F(\boldsymbol{u}\Pi_F),\label{trace_complex_trace_1}\\
\tr_2^{\cott}(\boldsymbol{\sigma})=&-\Def_F\curl_F(\boldsymbol{u}\cdot\boldsymbol{n}),\label{trace_complex_trace_2}\\
\tr_3^{\cott}(\boldsymbol{\sigma})=&\tfrac{1}{2}\operatorname{hess}_F\left(\div_F\left(\boldsymbol{u}\times\boldsymbol{n}\right)\right).\label{trace_complex_trace_3}
     \end{align} 
 \end{lemma}
 \begin{proof}
 Let $\mathbf{I}$ be the identity matrix. It follows from $\sym\curl (a\mathbf{I})=0$ for any scalar $a\in \mathbb R$ that \begin{align*}
     \sym\curl\boldsymbol{\sigma}&=\sym\curl\Def\boldsymbol{u}\\
     &=\tfrac{1}{2}\sym\left(-\left(\boldsymbol{u}\nabla+\nabla\boldsymbol{u}\right)\times\nabla\right)\\
     &=\tfrac{1}{2}\Def(\curl\boldsymbol{u}).
 \end{align*}
 Similarly, we have$$\sym(\Pi_F\mathbf{I}\times\boldsymbol{n})=-\boldsymbol{n}\times\sym(\boldsymbol{n}\times \mathbf{I})\times\boldsymbol{n}=0.$$ Thus we obtain
     \begin{align*}       \tr_1^{\cott}(\boldsymbol{\sigma})&=\sym\left(\Pi_F\left(\dev\Def\boldsymbol{u}\right)\times \boldsymbol{n}\right)\\
        &= \sym\left(\Pi_F\left(\Def\boldsymbol{u}\right)\times \boldsymbol{n}\right)\\
        &=
        \tfrac{1}{2}\sym(\nabla_F(\boldsymbol{u}\times\boldsymbol{n})-\nabla_F^{\perp}(\boldsymbol{u}\Pi_F))\\
        &=\tfrac{1}{2}\Def_F(\boldsymbol{u}\times\boldsymbol{n})-\tfrac{1}{2}\sym\curl_F(\boldsymbol{u}\Pi_F).
     \end{align*}
      Following Lemma~\ref{new_representation_traces_corollary} and \eqref{cross_product_identity_1}, we have
      \begin{align*}
        \tr_2^{\cott}(\boldsymbol{\sigma})&=\boldsymbol{n} \times\left(\sym\curl\left(\dev\Def\boldsymbol{u}\right)\right)\times\boldsymbol{n}-\sym\curl_F\left(\Pi_F 
    \left(\dev\Def\boldsymbol{u}\right)\cdot \boldsymbol{n}\right)\\
        &=\tfrac{1}{2}\boldsymbol{n} \times\left(\Def\left(\curl\boldsymbol{u}\right)\right)\times\boldsymbol{n}-\tfrac{1}{2}\sym\curl_F\left(\nabla_F\left(\boldsymbol{u}\cdot\boldsymbol{n}\right)+\partial_{\boldsymbol{n}}\left(\Pi_F\boldsymbol{u}\right)\right)\\
        &=\tfrac{1}{2}\sym\curl_F\left(\curl\boldsymbol{u}\times\boldsymbol{n}-\nabla_F\left(\boldsymbol{u}\cdot\boldsymbol{n}\right)-\partial_{\boldsymbol{n}}\left(\Pi_F\boldsymbol{u}\right)\right)\\
        &=-\Def_F\curl_F(\boldsymbol{u}\cdot\boldsymbol{n}),
      \end{align*}
      \begin{align*}
          \tr_3^{\cott}(\boldsymbol{\sigma})&= \boldsymbol{n} \times\left(\nabla \times\left(\tfrac{1}{2}\Def\left(\curl\boldsymbol{u}\right)\right)\right )\Pi_F+\left(\Pi_F \left(\tfrac{1}{2}\Def\left(\curl\boldsymbol{u}\right)\right)\cdot \boldsymbol{n}\right) \nabla_F^T \\
          &=\tfrac{1}{4}\boldsymbol{n}\times\left(\nabla\times\left(\curl\boldsymbol{u}\right)\nabla_F^T\right) +\tfrac{1}{4}\left(\nabla_F^2\left(\boldsymbol{n}\cdot\curl\boldsymbol{u}\right)+\partial_{\boldsymbol{n}}\left(\Pi_F\curl\boldsymbol{u}\right)\nabla_F^T\right)\\
          &=\tfrac{1}{4}\left(\nabla_F^2\left(\boldsymbol{n}\cdot\curl\boldsymbol{u}\right)-\partial_{\boldsymbol{n}}\left(\Pi_F\curl\boldsymbol{u}\right)\nabla_F^T\right)\\
          &\quad+\tfrac{1}{4}\left(\nabla_F^2\left(\boldsymbol{n}\cdot\curl\boldsymbol{u}\right)+\partial_{\boldsymbol{n}}\left(\Pi_F\curl\boldsymbol{u}\right)\nabla_F^T\right)\\
          &=\tfrac{1}{2}\nabla_F^2\left(\boldsymbol{n}\cdot\curl\boldsymbol{u}\right)\\
          &=\tfrac{1}{2}\operatorname{hess}_F\left(\operatorname{div}_F\left(\boldsymbol{u}\times\boldsymbol{n}\right)\right).
      \end{align*}
      Hence we complete the proof.
 \end{proof}
\begin{lemma}
For any sufficiently smooth $\bST$-valued function $\boldsymbol{\sigma}$, set $\boldsymbol{\tau}=\cott\boldsymbol{\sigma}$, it holds that
\label{trace complex identity part 2}
\begin{align}
    \boldsymbol{n}\cdot\boldsymbol{\tau} \cdot\boldsymbol{n}&=-\div_F\div_F\tr_2^{\cott}\left(\boldsymbol{\sigma}\right),\label{divdiv trace 2}\\       \boldsymbol{n}\times\boldsymbol{\tau} \cdot\boldsymbol{n}&=\operatorname{rot}_F\tr_3^{\cott}\left(\boldsymbol{\sigma}\right).\label{rot trace 3}
\end{align}
\end{lemma}
\begin{proof}
    It follows from \eqref{new_form_trace_2} that \begin{align*}
    \div_F\div_F\tr_2^{\cott}(\boldsymbol{\sigma})&=\div_F\div_F\left(\boldsymbol{n} \times(\sym\curl\boldsymbol{\sigma})\times\boldsymbol{n}\right)\\
    &=\nabla_F\cdot \left(\boldsymbol{n} \times(\sym\curl\boldsymbol{\sigma})\times\boldsymbol{n}\right)\cdot\nabla_F\\
    &=\boldsymbol{n}\cdot
\left(\nabla\times\left(\sym\curl\boldsymbol{\sigma}\right)\times\nabla\right)\cdot\boldsymbol{n}\\
&=-\boldsymbol{n}\cdot\cott\boldsymbol{\sigma}\cdot\boldsymbol{n}.
    \end{align*}
    And from \eqref{new_form_trace_3}, we get
    \begin{align*}
\boldsymbol{n}\times\cott\boldsymbol{\sigma}\cdot\boldsymbol{n}&=-\boldsymbol{n}\times\left(\nabla\times\left(\sym\curl\boldsymbol{\sigma}\right)\times\nabla\right)\cdot\boldsymbol{n}\\
&=\boldsymbol{n}\times\left(\nabla\times\left(\sym\curl\boldsymbol{\sigma}\right)\Pi_F\right)\cdot\nabla_F^{\perp}\\
&=\operatorname{rot}_F\tr_3^{\cott}(\boldsymbol{\sigma}),
    \end{align*}
where the last identity follows from $\nabla_F^{\perp}\cdot\nabla_F=0$.
\end{proof}

Next we proceed to prove Theorem~\ref{sufficient conforming condition}. We shall analyze the linearized Cotton-York operator through integration by parts. For this purpose, we introduce the following lemma. The proof can be found in Appendix \ref{integration by parts appendix}.
\begin{lemma}
    \label{integration by parts}
    Suppose $\boldsymbol{\sigma}$ and $\boldsymbol{\tau}$ are sufficiently smooth symmetric and traceless tensors in three dimensions. Then, it holds that:
        \begin{align}
        \quad&\int_K\cott \boldsymbol{\sigma}:\boldsymbol{\tau}-\int_K\cott \boldsymbol{\tau}:\boldsymbol{\sigma}\nonumber\\
        = &-\int_{\partial K}\tr_1^{\cott}(\boldsymbol{\sigma}):\Pi_F\inc\boldsymbol{\tau}\Pi_F+\int_{\partial K}\tr_3^{\cott}(\boldsymbol{\sigma}):\boldsymbol{n}\times\boldsymbol{\tau}\times\boldsymbol{n}\nonumber\\
    &-\int_{\partial K}\tr_2^{\cott}(\boldsymbol{\sigma}):2\operatorname{def}_F(\boldsymbol{n}\cdot\boldsymbol{\tau}\Pi_F)-\Pi_F\partial_n\boldsymbol{\tau}\Pi_F
    \nonumber\\
        &+
    \sum_{F\in\mathcal{F}(K)}\sum_{e\in \mathcal{E}(F)}\int_e\left(\boldsymbol{n}\cdot\partial_{\boldsymbol{t}_{F,e}}\boldsymbol{\sigma}\Pi_F\right) \cdot\left(\boldsymbol{n}\cdot\boldsymbol{\tau}\Pi_F\right)\label{int edge 1} \\
    &+
    \sum_{F\in\mathcal{F}(K)}\sum_{e\in \mathcal{E}(F)}\int_e\left(\boldsymbol{n}\cdot(\nabla\times\boldsymbol{\sigma})\cdot\boldsymbol{n}_{F,e}\right)\left(\boldsymbol{t}_{F,e}\cdot\boldsymbol{\tau}\cdot\boldsymbol{t}_{F,e}\right) \label{int edge 2}\\
    &-\sum_{F\in\mathcal{F}(K)}\sum_{e\in \mathcal{E}(F)}\int_e\left(\boldsymbol{t}_{F,e}\cdot(\nabla\times\boldsymbol{\sigma})\cdot\boldsymbol{t}_{F,e}\right)\left(\boldsymbol{n}_{F,e}\cdot\boldsymbol{\tau}\cdot\boldsymbol{n}\right)\label{int edge 3}\\
    &- \sum_{F\in\mathcal{F}(K)}\sum_{e\in \mathcal{E}(F)}\int_e\left(\boldsymbol{n}_{F,e}\cdot\boldsymbol{\sigma}\cdot\boldsymbol{t}_{F,e}\right)\left(\boldsymbol{n}\cdot(\nabla\times\boldsymbol{\tau})\cdot\boldsymbol{t}_{F,e}\right)\label{int edge 4}
    \\
    &+ \sum_{F\in\mathcal{F}(K)}\sum_{e\in \mathcal{E}(F)}\int_e\left(\boldsymbol{n}_{F,e}\cdot\boldsymbol{\sigma}\cdot\boldsymbol{n}\right)\left(\boldsymbol{t}_{F,e}\cdot(\nabla\times\boldsymbol{\tau})\cdot\boldsymbol{t}_{F,e}\right). \label{int edge 5}
    \end{align}
\end{lemma}
We are in the position of proving Theorem~\ref{sufficient conforming condition}.
\begin{proof}[Proof of Theorem~\ref{sufficient conforming condition}]
    Choose $\boldsymbol{\tau}\in C_0^{\infty}(\Omega;\bST)$. 
    We consider \begin{align}
    \int_{\Omega}\cott\boldsymbol{\tau}:\boldsymbol{\sigma}-\sum_{K\in \mathcal{T}_h}\int_K\cott\boldsymbol{\sigma}:\boldsymbol{\tau}.\label{eq:integration_by_parts}
\end{align}
From Lemma~\ref{integration by parts}, we have a set of integrals on internal edges and faces of $\mathcal{T}_h$, and we aim to show they cancel out.

    We note that if two elements $K$ and $K'$ share a face $F$ and edge $e$, it holds that
    $$\boldsymbol{n}=-\boldsymbol{n}',\quad\boldsymbol{n}_{F,e}=\boldsymbol{n}_{F,e}',\quad\boldsymbol{t}_{F,e}=-\boldsymbol{t}_{F,e}',$$
    where $'$ is used to distinguish vectors in $K$ and $K'$.
    Using these properties and the fact that $\tr_1^{\cott}(\boldsymbol{\sigma})|_F$, $\tr_2^{\cott}(\boldsymbol{\sigma})|_F$ and $\tr_3^{\cott}(\boldsymbol{\sigma})|_F$ are single-valued, we can deduce that 
     all the terms on the faces cancel out. It remains to show the edge terms concerning \eqref{int edge 1}-\eqref{int edge 5} vanish.
    
     For terms concerning \eqref{int edge 1}, by changing the order of summation, it follows from the fact that $\boldsymbol{\sigma}|_e$ is single-valued that
     \begin{align}
        &\quad -\sum_{K\in\mathcal{T}_h}
    \sum_{F\in\mathcal{F}(K)}\sum_{e\in \mathcal{E}(F)}\int_e\left(\boldsymbol{n}\cdot\partial_{\boldsymbol{t}_{F,e}}\boldsymbol{\sigma}\Pi_F\right) \cdot\left(\boldsymbol{n}\cdot\boldsymbol{\tau}\Pi_F\right)\nonumber\\
    &=-\sum_{e\in\mathcal{E}(\mathcal{T}_h)}\sum_{F\in \mathcal{F}(e)}\int_e \left[(\boldsymbol{n}\cdot\partial_{\boldsymbol{t}_{F,e}}\boldsymbol{\sigma}|_K\Pi_F)-(\boldsymbol{n}'\cdot\partial_{\boldsymbol{t}'_{F,e}}\boldsymbol{\sigma}|_{K'}\Pi_F)\right]\cdot(\boldsymbol{n}\cdot\boldsymbol{\tau}\Pi_F)\label{change order}\\
    &=0.\nonumber
     \end{align}
Here, we use $\mathcal{T}_h(e)$ and $\mathcal{F}(e)$ denotes internal edges of $\mathcal T_h$ and faces that share an edge $e$. A similar argument shows that terms concerning \eqref{int edge 4} and \eqref{int edge 5} vanish.

For any vector $\boldsymbol{a}$, $\boldsymbol{b}\in \mathbb R^3$, it holds that $\boldsymbol{a}\cdot (\boldsymbol{b}\times \boldsymbol{\sigma}) = (\boldsymbol{a}\times\boldsymbol{b})\cdot \boldsymbol{\sigma}$ and $(\boldsymbol{\sigma}\times\boldsymbol{a})\cdot \boldsymbol{b} = \boldsymbol{\sigma}\cdot(\boldsymbol{a}\times\boldsymbol{b} )$.    Using this property, we have
        \begin{align} \boldsymbol{t}_{F,e}\cdot\tr_1^{\cott}(\boldsymbol{\sigma})|_F\cdot\boldsymbol{t}_{F,e}&=\boldsymbol{t}_{F,e}\cdot(\boldsymbol{\sigma}\times \boldsymbol{n})\cdot \boldsymbol{t}_{F,e}\nonumber\\
        &=-\boldsymbol{n}_{F,e}\cdot\boldsymbol{\sigma}\cdot\boldsymbol{t}_{F,e},\label{eq:t_tr_1_t}\\
\boldsymbol{n}_{F,e}\cdot\tr_2^{\cott}(\boldsymbol{\sigma})|_F\cdot\boldsymbol{n}_{F,e} & =\boldsymbol{n}_{F,e}\cdot 2\Def_F(\boldsymbol{n}\cdot\boldsymbol{\sigma}\Pi_F)\cdot\boldsymbol{t}_{F,e} - \boldsymbol{n}_{F,e}\cdot(\partial_{\boldsymbol{n}}\boldsymbol{\sigma}\times \boldsymbol{n})\cdot\boldsymbol{n}_{F,e}\nonumber\\
& = \partial_{\boldsymbol{n}_{F,e}}(\boldsymbol{n}\cdot\boldsymbol{\sigma}\cdot\boldsymbol{t}_{F,e}) + \partial_{\boldsymbol{t}_{F,e}}(\boldsymbol{n}\cdot\boldsymbol{\sigma}\cdot\boldsymbol{n}_{F,e})-\partial_{\boldsymbol{n}}(\boldsymbol{n}_{F,e}\cdot\boldsymbol{\sigma}\cdot\boldsymbol{t}_{F,e}),\\
\boldsymbol{n}\cdot(\nabla\times\boldsymbol{\sigma})\cdot\boldsymbol{n}_{F,e}&=-\nabla_F\cdot(\boldsymbol{n}\times\boldsymbol{\sigma}\cdot\boldsymbol{n}_{F,e})\nonumber\\&=
-\partial_{\boldsymbol{t}_{F,e}}(\boldsymbol{t}_{F,e}\cdot (\boldsymbol{n}\times\boldsymbol{\sigma})\cdot\boldsymbol{n}_{F,e})-\partial_{\boldsymbol{n}_{F,e}}(\boldsymbol{n}_{F,e}\cdot (\boldsymbol{n}\times\boldsymbol{\sigma})\cdot\boldsymbol{n}_{F,e})\nonumber\\
&=-\partial_{\boldsymbol{t}_{F,e}}\left(\boldsymbol{n}_{F,e}\cdot\boldsymbol{\sigma}\cdot\boldsymbol{n}_{F,e}\right)-\partial_{\boldsymbol{n}_{F,e}}\left(\boldsymbol{t}_{F,e}\cdot\tr_1^{\cott}(\boldsymbol{\sigma})|_F\cdot\boldsymbol{t}_{F,e}\right),\label{eq:identity_for_edge_2}\\
\boldsymbol{t}_{F,e}\cdot(\nabla\times\boldsymbol{\sigma})\cdot\boldsymbol{t}_{F,e}&=-\nabla\cdot(\boldsymbol{t}_{F,e}\times\boldsymbol{\sigma}\cdot\boldsymbol{t}_{F,e})\nonumber\\
&= -\partial_{\boldsymbol{n}}(\boldsymbol{n}\cdot(\boldsymbol{t}_{F,e}\times\boldsymbol{\sigma}\cdot\boldsymbol{t}_{F,e}))-\partial_{\boldsymbol{n}_{F,e}}(\boldsymbol{n}_{F,e}\cdot(\boldsymbol{t}_{F,e}\times\boldsymbol{\sigma}\cdot\boldsymbol{t}_{F,e}))\nonumber\\&=\partial_{\boldsymbol{n}}\left(\boldsymbol{n}_{F,e}\cdot\boldsymbol{\sigma}\cdot\boldsymbol{t}_{F,e}\right)-\partial_{\boldsymbol{n}_{F,e}}\left(\boldsymbol{n}\cdot\boldsymbol{\sigma}\cdot\boldsymbol{t}_{F,e}\right)\nonumber\\&
    =-\boldsymbol{n}_{F,e}\cdot\tr_2^{\cott}(\boldsymbol{\sigma})|_F\cdot\boldsymbol{n}_{F,e}+\partial_{\boldsymbol{t}_{F,e}}(\boldsymbol{n}\cdot\boldsymbol{\sigma}\cdot\boldsymbol{n}_{F,e}).\label{eq:identity_for_edge_3}
    \end{align}
By \eqref{eq:identity_for_edge_2} and \eqref{eq:identity_for_edge_3}, terms concerning \eqref{int edge 2} and \eqref{int edge 3} vanish by changing the order of summation, analogous to \eqref{change order}. Therefore \eqref{eq:integration_by_parts} vanishes, which establishes that $\boldsymbol{\sigma}$ is ${H}(\cott)$-conforming. This completes the proof.
\end{proof}
\begin{remark}
\label{rmk:distributional}
    The edge continuity is necessary for the ${H}(\cott)$-conformity. We illustrate this argument by constructing a piecewise constant tensor $\boldsymbol{\sigma}$ with single-valued $\tr_i(\boldsymbol{\sigma})|_F$, which fails to be $H(\cott)$-conforming. The counterexample can be constructed using the traceless part of the linearized Regge element function \cite{christiansen2011linearization}.
    
    Let $\boldsymbol{\sigma'}$ be a piecewise constant symmetric tensor (Regge metric), determined by the tangential-tangential components on each edge: $$ \int_e\boldsymbol{t}_e\cdot\boldsymbol{\sigma'}\cdot\boldsymbol{t}_e,\,e\in \mathcal{E}(K).$$
    Take $\boldsymbol{\sigma}=\dev\boldsymbol{\sigma'}$.
    Since $\boldsymbol{\sigma'}$ satisfies the tangential-tangential continuity, i.e. $\tr^{\inc}_1(\boldsymbol{\sigma'})|_F$ is single-valued for all the faces $F\in\mathcal{F}(\mathcal{T}_h)$,  by \eqref{new_representation_1}, $\tr_1^{\cott}(\boldsymbol{\sigma})|_F=\sym\left(\tr^{\inc}_1(\boldsymbol{\sigma'})|_F\times\boldsymbol{n}\right)$, $\tr_2^{\cott}(\boldsymbol{\sigma})|_F=0$ and $\tr_3^{\cott}(\boldsymbol{\sigma})|_F=0$ are single-valued, but there is no full edge continuity.
    
    For such $\boldsymbol{\sigma}$ constructed above, by \eqref{eq:t_tr_1_t} we obtain that terms concerning \eqref{int edge 4} vanish since $\tr_1^{\cott}(\boldsymbol{\sigma})|_F$ is single-valued on any face $F$ of $\mathcal{T}_h$. Terms concerning \eqref{int edge 1}-\eqref{int edge 3} also vanish, since $\boldsymbol{\sigma}$ is constant on each tetrahedron. Therefore, we have
\begin{align*}  \int_\Omega\cott\boldsymbol{\tau}:\boldsymbol{\sigma}&=\sum_{K\in\mathcal{T}_h}\int_K\cott\boldsymbol{\sigma}:\boldsymbol{\tau}-\sum_{K\in\mathcal{T}_h} \sum_{F\in\mathcal{F}(K)}\sum_{e\in \mathcal{E}(F)}\int_e\left(\boldsymbol{n}_{F,e}\cdot\boldsymbol{\sigma}\cdot\boldsymbol{n}\right)\left(\boldsymbol{t}_{F,e}\cdot(\nabla\times\boldsymbol{\tau})\cdot\boldsymbol{t}_{F,e}\right)\\
    &=\int_{\Omega} \cott\boldsymbol{\sigma}:\boldsymbol{\tau}-\sum_{e\in\mathcal{E}(\mathcal{T}_h)}\int_e\left(\sum_{F\in\mathcal{F}(e)}\llbracket\boldsymbol{n}_{F,e}\cdot\boldsymbol{\sigma}\cdot\boldsymbol{n}\rrbracket\right)\left(\boldsymbol{t}_{e}\cdot\sym\curl\boldsymbol{\tau}\cdot\boldsymbol{t}_{e}\right),
\end{align*}
where $\llbracket\boldsymbol{n}_{F,e}\cdot\boldsymbol{\sigma}\cdot\boldsymbol{n}\rrbracket$ denotes the jump $\boldsymbol{n}_{F,e}\cdot\boldsymbol{\sigma}|_K\cdot\boldsymbol{n}+\boldsymbol{n}_{F,e}'\cdot\boldsymbol{\sigma}|_{K'}\cdot\boldsymbol{n}',$  $ e\in K\cap K'=F$. Since there is no continuity condition on edges, this establishes that $\boldsymbol{\sigma}$ is not $H(\cott)$-conforming in general.
\end{remark}
 \label{traces and trace complexes}
\section{Bubbles and Bubble Complexes}\label{bubble section}
The goal of this section is to prove the exactness of the smoothest bubble conformal complex \eqref{eq:bubble-seq2}.

We begin by recalling the $H(\cott)$ bubble spaces (see also Section~\ref{sec:overview} for context):
\begin{align*}
\mathbb{B}^{\cott}_k\!\left(K;\bST\right)
&=\{\boldsymbol{\sigma}\in P_k\!\left(K;\bST\right):
\tr_1^{\cott}\!\left(\boldsymbol{\sigma}\right)|_F=
\tr_2^{\cott}\!\left(\boldsymbol{\sigma}\right)|_F=
\tr_3^{\cott}\!\left(\boldsymbol{\sigma}\right)|_F=0
\ \ \forall\,F\in\mathcal{F}(K)\},\\[0.25em]
{\mathbb{B}}^{1\cott}_{k-4}\!\left(K;\bST\right)
&=\{{\boldsymbol{\sigma}}\in P_{k-4}\!\left(K;\bST\right):\ 
b_K\,\boldsymbol{\sigma}\in \mathbb{B}^{\cott}_k\!\left(K;\bST\right)\},\\[0.25em]
{\mathbb{B}}^{2\cott}_{k-8}\!\left(K;\bST\right)
&=\{{\boldsymbol{\sigma}}\in P_{k-8}\!\left(K;\bST\right):\ 
b_K^2\,\boldsymbol{\sigma}\in \mathbb{B}^{\cott}_k\!\left(K;\bST\right)\}.
\end{align*}
By construction, these spaces satisfy the inclusion hierarchy
\begin{equation}\label{eq:inclusion-hierarchy}
b_K^2\,{\mathbb{B}}^{2\cott}_{k-8}\!\left(K;\bST\right)\ \subset\
b_K\,{\mathbb{B}}^{1\cott}_{k-4}\!\left(K;\bST\right)\ \subset\
\mathbb{B}^{\cott}_k\!\left(K;\bST\right).
\end{equation}
We also recall the $H(\div)$ bubble spaces $\mathbb{B}^{\div}_k$ in \eqref{def:div_bubble} and their smoother variants $\mathbb{B}^{\div,(s)}_k$ in \eqref{def:div_bubble_extra_smoothness} introduced in Section~\ref{sec:introduction}.

The conformal bubble complexes built on $\mathbb{B}^{\cott}_k$ and $\mathbb{B}^{1\cott}_{k-4}$ are delicate: their definitions involve higher-order trace constraints and their dimensions are not readily accessible (cf.\ the dimension analysis for $\inc$-bubble spaces in~\cite{arnold2008finite}). Consequently, these sequences are not well suited for a clean characterization of the divergence-free subspace
\[
\mathbb{B}^{\div}_{k-3}\!\left(K;\bST\right)\cap\ker(\div).
\]
In contrast, the space $b_K^2\,\mathbb{B}^{2\cott}_{k-8}\!\left(K;\bST\right)$ admits a simpler boundary structure. Indeed, for any $\boldsymbol{\sigma}\in\mathbb{B}^{2\cott}_{k-8}\!\left(K;\bST\right)$,
\[
\tr_1^{\cott}\!\bigl(b_K^2\boldsymbol{\sigma}\bigr)|_F
=\tr_2^{\cott}\!\bigl(b_K^2\boldsymbol{\sigma}\bigr)|_F=0
\qquad\forall\,F\in\mathcal{F}(K),
\]
because the operators $\tr_1^{\cott}$ and $\tr_2^{\cott}$ involve only zeroth- and first-order boundary differential terms, which are annihilated by the factor $b_K^2$. Thus, among the three $H(\cott)$ bubble levels appearing in our complexes, the space $b_K^2\,\mathbb{B}^{2\cott}_{k-8}\!\left(K;\bST\right)$ is the most tractable: only the third-order conformal trace $\tr_3^{\cott}$ requires explicit analysis. For this reason, our proof of exactness will focus on the bubble conformal complex whose middle spaces are built from $b_K^2\,\mathbb{B}^{2\cott}_{k-8}\!\left(K;\bST\right)$.

Our main tool is a \emph{bubble} variant of the BGG construction. Starting from the canonical bubble de~Rham complex~\cite{arnold2006finite} and the bubble Stokes complex~\cite{neilan2015discrete}, together with minor supersmoothness variants, we derive bubble elasticity and bubble divdiv complexes via the BGG machinery. The bubble setting is subtler than the continuous one: the higher interior smoothness forces additional compatibility conditions on traces. Nevertheless, the BGG framework yields streamlined exactness proofs. In particular, it recovers and simplifies results from the smooth bubble elasticity complex~\cite{arnold2008finite} and from the smoother bubble divdiv complexes developed in~\cite{chen2022finitedivdiv,hu2022conformingdivdiv}.

\textbf{Organization.}
We first collect preliminary results on the characterization and geometric decompositions of the bubble spaces needed later. We then construct bubble BGG diagrams for elasticity and divdiv complexes, and finally use these to establish the exactness of \eqref{eq:bubble-seq2}.

\subsection{Characterization and geometric decompositions of bubble spaces}
To begin with, define the polynomial space on $K$ with the first trace vanishing as
\begin{align}
    \mathbb{B}_k^{\tr_1}(K;\bST):=\{\boldsymbol{\sigma}\in P_k\left(K;\bST\right):\tr_1^{\cott}(\boldsymbol{\sigma})|_F=0 \text{ for each } F\in\mathcal{F}(K)\}.
\end{align}
We have the following proposition:
\begin{proposition}
\label{prop:bubbles_in_tr1_bubble}
    It holds that
    \begin{align}
        \mathbb{B}^{\cott}_{k}(K;\bST)&\subset\mathbb{B}_{k}^{\tr_1}(K;\bST),\label{cinc bubble in tr_1}\\
        \mathbb{B}^{1\cott}_{k-4}(K;\bST)&\subset\mathbb{B}_{k-4}^{\tr_1}(K;\bST),\label{cinc1 bubble in tr_1}\\
        \mathbb{B}^{2\cott}_{k-8}(K;\bST)&=\mathbb{B}_{k-8}^{\tr_1}(K;\bST).\label{cinc2 bubble in tr_1}
    \end{align}
\end{proposition}
\begin{proof}
    Note that \eqref{cinc bubble in tr_1} holds by the definition of $\mathbb{B}_k^{\cott}(K;\bST)$.
    
    For each $\boldsymbol{\sigma}_1\in\mathbb{B}_{k-4}^{1\cott}(K;\bST)$, since $b_K\boldsymbol{\sigma}_1|_F=0$ on $F\in\mathcal{F}(K)$, we have
    \begin{align}      0=\tr_2^{\cott}(b_K\boldsymbol{\sigma}_1)|_F&=\sym\left(\left(2\Def_F\left(\boldsymbol{n}\cdot b_K\boldsymbol{\sigma}_1\Pi_F\right)-\Pi_F\partial_{\boldsymbol{n}}\left(b_K\boldsymbol{\sigma}_1\right)\Pi_F\right)\times\boldsymbol{n}\right)|_F\nonumber\\
    &=-\partial_{\boldsymbol{n}}\left(\lambda_F\right)b_F\sym\left(\Pi_F\boldsymbol{\sigma}_1\times\boldsymbol{n}\right)|_F=\tfrac{b_F}{h_F}\tr_1^{\cott}(\boldsymbol{\sigma}_1)|_F.
    \end{align}
    Note that the first term of the right-hand side in the first line vanishes. Thus, \eqref{cinc1 bubble in tr_1} is confirmed to be valid.
    
    For each ${\boldsymbol{\sigma}_2}\in P_{k-8}(K;\bST)$, a direct expansion shows
\begin{align}
\sym\curl\left(b_K^2{\boldsymbol{\sigma}_2}\right)&=\sum_{F\in\mathcal{F}(K)}b_F\sym\left(\operatorname{grad}\lambda_F\times b_K{\boldsymbol{\sigma}_2}\right)+b_K\sym\curl\left(b_K{\boldsymbol{\sigma}_2}\right)\nonumber\\
&=b_K\left(2\sum_{F\in\mathcal{F}(K)}b_F\sym\left(\operatorname{grad}\lambda_F\times {\boldsymbol{\sigma}_2}\right)+b_K\sym\curl{\boldsymbol{\sigma}_2}\right).
\end{align}
It follows that on $F\in\mathcal{F}(K)$,
\begin{align}
\tr_3\left(b_K^2{\boldsymbol{\sigma}_2}\right)|_F&=-\Pi_{F}\partial_{\boldsymbol{n}}\left(\sym\curl\left(b_K^2{\boldsymbol{\sigma}_2}\right)\right)|_F\Pi_{F}\nonumber\\
&=-2\partial_{\boldsymbol{n}}\left(\lambda_F\right)b_F^2\left(\Pi_{F}\sym\left(\operatorname{grad}\lambda_F\times {\boldsymbol{\sigma}_2}\right)\Pi_{F}\right)|_F\nonumber\\
&=-\tfrac{2b_F^2}{h_F^2}\sym\left(\boldsymbol{n}\times{\boldsymbol{\sigma}_2}\Pi_{F}\right)|_F=\tfrac{2b_F^2}{h_F^2}\tr_1\left({\boldsymbol{\sigma}_2}\right)|_{F}.\label{trace3_H2}
\end{align}
Since $\tr_1^{\cott}(b_K^2\boldsymbol{\sigma}_2)|_F=\tr_2^{\cott}(b_K^2\boldsymbol{\sigma}_2)|_F=0$ for each $F\in\mathcal{F}(K)$, \eqref{trace3_H2} shows that ${\boldsymbol{\sigma}_2}\in\mathbb{B}^{2\cott}_{k-8}(K;\bST)$ if and only if $\tr_1^{\cott}({\boldsymbol{\sigma}_2})|_{F}=0$ for each face $F\in\mathcal{F}(K)$. This completes the proof.
\end{proof}

Next, we investigate the characterization of $\mathbb{B}_{k}^{2\cott}(K;\bST)$.
\begin{proposition}
\label{representation of tr_1 bubble}
  $\mathbb{B}_k^{2\cott}(K;\bST)$ can be decomposed geometrically as follows:
    \begin{equation}
        \begin{aligned}
            \mathbb{B}_k^{2\cott}(K;\bST)&=\oplus_{\substack{e\in \mathcal{E}(K)\\e=F_+\cap F_-}}b_eP_{k-2}(e;\mathbb{R})\dev\sym(\boldsymbol{n}_+\boldsymbol{n}_-^T)\\&\quad \oplus_{F\in \mathcal{F}(K)}b_F\dev\sym\left(P_{k-3}(F;\mathbb{R}^3)\boldsymbol{n}^T\right)\\
    &\quad \oplus b_KP_{k-4}(K;\bST).
        \end{aligned}
        \label{representation of tr_1 bubble sum}
    \end{equation}
\end{proposition}
This decomposition is intuitive, for the vanishing of $\tr_1^{\cott}$ removes 2 out of 5 ``directions'' in $\bST$. Therefore, on edges it only carries a single direction. It is important to note that the dimension of $\mathbb{B}^{2\cott}_k(K;\bST)$ can be explicitly determined, as a direct corollary of Proposition~\ref{representation of tr_1 bubble}.
To aid our proof of Proposition~\ref{representation of tr_1 bubble}, we also present two sets of linearly independent basis for $\bST$ in the following two lemmas. The proofs of these lemmas can be found in Appendix \ref{edge basis appendix} and Appendix \ref{face basis appendix}. 
\begin{lemma}{{\rm (Edge basis of $\bST$)}}{\label{edge_basis}}
    Let $e = F_+ \cap F_-$.   $\boldsymbol{n}_+$ and $\boldsymbol{n}_-$ represent the corresponding outward unit normal vectors of $F_+$ and $F_-$, respectively. Then  $\sym(\boldsymbol{t_e}\boldsymbol{n}_+^T)$, $\sym(\boldsymbol{t_e}\boldsymbol{n}_-^T)$, $\dev(\boldsymbol{n}_+\boldsymbol{n}_+^T)$,
    $\dev(\boldsymbol{n}_-\boldsymbol{n}_-^T)$, and
    $\dev\sym(\boldsymbol{n}_+\boldsymbol{n}_-^T)$ are linearly independent and form a basis of $\bST$.
\end{lemma}
\begin{lemma}{\rm (Face basis of $\bST$)}{\label{face_basis}}
    Let $\boldsymbol{t}_{F,1}$ and $\boldsymbol{t}_{F,2}$ be two unit tangent vectors of face $F$ satisfying $\boldsymbol{t}_{F,1}\cdot\boldsymbol{t}_{F,2}=0$. Then $\sym(\boldsymbol{n}\boldsymbol{t}_{F,1}^T)$, $\sym(\boldsymbol{n}\boldsymbol{t}_{F,2}^T)$,
    $\sym(\boldsymbol{t}_{F,1}\boldsymbol{t}_{F,2}^T)$,
    $\boldsymbol{t}_{F,1}\boldsymbol{t}_{F,1}^T-\boldsymbol{t}_{F,2}\boldsymbol{t}_{F,2}^T$, and $\dev\left(\boldsymbol{n}\boldsymbol{n}^T\right)$ are linearly independent and form a basis of $\bST$.
\end{lemma}

Now, we are ready to prove Proposition~\ref{representation of tr_1 bubble}.
\begin{proof}[Proof of Proposition~\ref{representation of tr_1 bubble}]
    It is straightforward to verify that the right-hand side of \eqref{representation of tr_1 bubble sum} is indeed a direct sum, and it forms a subset of $\mathbb{B}_k^{2\cott}(K;\bST)$. It remains to prove the opposite inclusion relation.
    
    For each $\boldsymbol{\sigma}\in \mathbb{B}_k^{2\cott}(K;\bST)$, by \eqref{cinc2 bubble in tr_1} on edge $e\in F_+\cap F_-$ we have $$\sym\left(\Pi_{F_+}\boldsymbol{\sigma}|_e\times\boldsymbol{n}_+\right)=\sym\left(\Pi_{F_-}\boldsymbol{\sigma}|_e\times\boldsymbol{n}_-\right)=0.$$
Hence, the following identities hold on edge $e$:
\begin{equation}
\left\{
\begin{array}{rcl}
\boldsymbol{t}_{F_\pm,e}\cdot\boldsymbol{\sigma}|_e\cdot\boldsymbol{n}_{F_\pm,e} &=& -\boldsymbol{t}_{F_\pm,e}\cdot\tr_1^{\cott}(\boldsymbol{\sigma}|_e)|_{F_\pm}\cdot\boldsymbol{t}_{F_\pm,e} = 0, \\
\boldsymbol{t}_{F_\pm ,e}\cdot\boldsymbol{\sigma}|_e\cdot\boldsymbol{t}_{F_\pm,e} 
- \boldsymbol{n}_{F_\pm,e}\cdot\boldsymbol{\sigma}|_e\cdot\boldsymbol{n}_{F_\pm,e} 
&=& 2\boldsymbol{t}_{F_\pm,e}\cdot\tr_1^{\cott}(\boldsymbol{\sigma}|_e)|_{F_\pm}\cdot\boldsymbol{n}_{F_\pm,e} = 0.
\end{array}
\right.
\label{eq:edge_trace_identity}
\end{equation}

By Lemma~\ref{edge_basis},  $\boldsymbol{\sigma}|_e\in P_{k}(e;\mathbb R)|_e\dev\sym(\boldsymbol{n}_+\boldsymbol{n}_-^T)$. Since each vertex has 3 adjacent edges, $\boldsymbol{\sigma}(\delta)=0$, $\forall\delta\in\mathcal{V}(K)$ and subsequently $\boldsymbol{\sigma}\in P_k^{(0)}(K;\bST)$. Thus, from \eqref{geo decomp}, there exist $p_e\in P_{k-2}(e;\mathbb R)$, $\boldsymbol{U}_{F}\in P_{k-3}(F;\bST)$ and $\boldsymbol{S}\in P_{k-4}(K;\bST)$ such that
$$\boldsymbol{\sigma}=\sum_{e\in\mathcal{E}(K)}b_e p_e\dev\sym(\boldsymbol{n}_+\boldsymbol{n}_-^T)+\sum_{F\in\mathcal{F}(K)}b_F \boldsymbol{U}_{F}+b_K\boldsymbol{S}.$$

On each face $F$, $\tr_1^{\cott}(\boldsymbol{\sigma})|_F=\sym(\Pi_{F}\boldsymbol{\sigma}|_F\times\boldsymbol{n})=0$. This implies on face $F$,
$$\left\{\begin{array}{rcl}
\boldsymbol{t}_{F,1}\cdot\boldsymbol{U}_F\cdot\boldsymbol{t}_{F,2}=0 \\
\boldsymbol{t}_{F,1}\cdot\boldsymbol{U}_F\cdot\boldsymbol{t}_{F,1}-\boldsymbol{t}_{F,2}\cdot\boldsymbol{U}_F\cdot\boldsymbol{t}_{F,2}=0
\end{array} \right..$$

From Lemma~\ref{face_basis}, we see that 
\begin{align*}
    \boldsymbol{U}_{F}&\in P_{k-3}(F;\mathbb R)\otimes\operatorname{span}\{\sym\left(\boldsymbol{n}\boldsymbol{t}_{F,1}^T\right),\,\sym\left(\boldsymbol{n}\boldsymbol{t}_{F,2}^T\right),\,\dev(\boldsymbol{n}\boldsymbol{n}^T)\}.
\end{align*}
Hence there exists $\boldsymbol{p}_F\in P_{k-3}(K;\mathbb R^3)$ such that $\boldsymbol{U}_F=\dev\sym\left(\boldsymbol{n}\boldsymbol{p}_F^T\right)$.
This proves the proposition.
\end{proof}
\begin{remark}
Recall the $H^1\cap H(\inc)$ bubble space with vanishing tangential–tangential trace: 
\begin{align}
{\mathbb{B}}^{1\inc}_{k}\left(K;\bS\right)&=\left\{{\boldsymbol{\sigma}}\in P_{k}\left(K;\bS\right):\tr_1^{\inc}(\boldsymbol{\sigma})|_F= \boldsymbol{n}\times\boldsymbol{\sigma}\times\boldsymbol{n}|_F=0,\;\forall{\text{ face } F\in\mathcal{F}(K)}\right\}.\label{def:1inc_bubble}
\end{align}
Applying the same methodology as Proposition~\ref{representation of tr_1 bubble}, we can show that 
    \begin{equation}
        \begin{aligned}
            \mathbb{B}_k^{1\inc}(K;\bS)&=\oplus_{\substack{e\in \mathcal{E}(K)\\e=F_+\cap F_-}}b_eP_{k-2}(e;\mathbb{R})\sym(\boldsymbol{n}_+\boldsymbol{n}_-^T)\\&\quad \oplus_{F\in \mathcal{F}(K)}b_F\sym\left(P_{k-3}(F;\mathbb{R}^3)\boldsymbol{n}^T\right)\\
    &\quad \oplus b_KP_{k-4}(K;\bS).
        \end{aligned}
\label{eq:1inc_geometric_decompositions}
    \end{equation}
    By Proposition~\ref{representation of tr_1 bubble}, it then holds that
    \begin{align*}
        \mathbb{B}_k^{2\cott}(K;\bST)=\dev \mathbb{B}_k^{1\inc}(K;\bS).
    \end{align*}
\end{remark}
\begin{lemma}
If $\boldsymbol{\sigma}\in\mathbb{B}^{\cott}_k\left(K;\bST\right)$, then $\boldsymbol{\sigma}|_e=0$ for each edge $e\in\mathcal{E}(K)$.\label{edge vanish 1}
\end{lemma}
\begin{proof}  Let $\boldsymbol{\sigma}\in\mathbb{B}^{\cott}_k\left(K;\bST\right)$. Recall \eqref{new_form_trace_2}. On each edge $e=F_+\cap F_-$ it holds that
    \begin{align*}
0&=\boldsymbol{n}_{F_+,e}\cdot\tr_2^{\cott}(\boldsymbol{\sigma})|_{F_+,e}\cdot\boldsymbol{n}_{F_+,e}-\boldsymbol{n}_{F_-,e}\cdot\tr_2^{\cott}(\boldsymbol{\sigma})|_{F_-,e}\cdot\boldsymbol{n}_{F_-,e}\\
&=\left(-\boldsymbol{t}_{F_+,e}\cdot\left(\sym\curl\boldsymbol{\sigma}\right)\cdot\boldsymbol{t}_{F_+,e}+\partial_{\boldsymbol{t}_{F_+,e}}\left(\boldsymbol{n}_{F_+,e}\cdot\boldsymbol{\sigma}\cdot\boldsymbol{n}_{+}\right)\right)\\
&\quad-\left(-\boldsymbol{t}_{F_-,e}\cdot\left(\sym\curl\boldsymbol{\sigma}\right)\cdot\boldsymbol{t}_{F_-,e}+\partial_{\boldsymbol{t}_{F_-,e}}\left(\boldsymbol{n}_{F_-,e}\cdot\boldsymbol{\sigma}\cdot\boldsymbol{n}_{-}\right)\right)\\
&=\partial_{\boldsymbol{t}_{F_+,e}}\left(\boldsymbol{n}_{F_+,e}\cdot\boldsymbol{\sigma}\cdot\boldsymbol{n}_{+}+\boldsymbol{n}_{F_-,e}\cdot\boldsymbol{\sigma}\cdot\boldsymbol{n}_{-}\right).
    \end{align*}
From Proposition~\ref{representation of tr_1 bubble}, $\boldsymbol{\sigma}(\delta)=0$,  $\forall\delta\in\mathcal{V}(K)$ and there exists $p_e\in P_{k-2}(e;\mathbb R)$ such that $\boldsymbol{\sigma}|_e=b_ep_e\dev\sym\left(\boldsymbol{n}_+\boldsymbol{n}_-^T\right)$ for each edge $e$. Hence, we obtain $\boldsymbol{n}_{F_+,e}\cdot\boldsymbol{\sigma}\cdot\boldsymbol{n}_{+}+\boldsymbol{n}_{F_-,e}\cdot\boldsymbol{\sigma}\cdot\boldsymbol{n}_{-}=0$ on $e$. 
This implies, $$\left[\left(\boldsymbol{n}_{F_+,e}\cdot\boldsymbol{n}_{-}\right)+\left(\boldsymbol{n}_{F_-,e}\cdot\boldsymbol{n}_{+}\right)\right]b_ep_e=0\text{ on } e.$$
Thus, $\boldsymbol{\sigma}|_e=0$ for each edge $e\in \mathcal{E}(K)$.
\end{proof}
Recall our definition of $\mathbb{B}_{k}^{\div, (s)}(K;\bST)$ in \eqref{def:div_bubble_extra_smoothness}.
\begin{proposition}
\label{proposition: geometric decompositions for div bubble}
    Assume $s\geq 1$ and $k\geq2s+1$. Then $\mathbb{B}_{k}^{\div, (s)}(K;\bST)$ can be decomposed geometrically as follows:
    \begin{equation}
        \begin{aligned}
            \mathbb{B}_{k}^{\div, (s)}(K;\bST)&= \oplus_{F\in\mathcal{F}(K)}b_FP_{k-3}^{(s-2)}(F;\mathbb{R})\sym(\boldsymbol{t}_{F,1}\boldsymbol{t}_{F,2}^T)\\
    &\quad\oplus_{F\in\mathcal{F}(K)}b_FP_{k-3}^{(s-2)}(F;\mathbb{R})\left(\boldsymbol{t}_{F,1}\boldsymbol{t}_{F,1}^T-\boldsymbol{t}_{F,2}\boldsymbol{t}_{F,2}^T\right)\\
    &\quad\oplus b_K P_{k-4}^{(s-3)}(K;\bST).
        \end{aligned}
        \label{representation of div bubbles}
    \end{equation}
\end{proposition}
\begin{proof}
    It suffices to show the one side inclusion ``$\subset$'', since the converse is straightforward. From Lemma~\ref{edge_basis}, on edge $e=F_+\cap F_-$, we can decompose $\boldsymbol{\tau}\in \mathbb{B}_{k}^{\div}(K;\bST)$ as $\boldsymbol{\tau}|_e=u_1\sym(\boldsymbol{t_e}\boldsymbol{n}_+^T)+u_2\sym(\boldsymbol{t_e}\boldsymbol{n}_-^T)+u_3\dev(\boldsymbol{n}_+\boldsymbol{n}_+^T)+u_4\dev(\boldsymbol{n}_-\boldsymbol{n}_-^T)+u_5\dev\sym(\boldsymbol{n}_+\boldsymbol{n}_-^T)$ with $u_i\in P_{k}(e;\mathbb R)$. By definition, $\boldsymbol{\tau}\cdot\boldsymbol{n}_+|_e=\boldsymbol{\tau}\cdot\boldsymbol{n}_-|_e=0$. Therefore, we have 
$$\left\{\begin{array}{rcl}
   \tfrac{1}{2}\left(u_1+(\boldsymbol{n}_+\cdot\boldsymbol{n}_-)u_2\right)\boldsymbol{t}_e +\left(\tfrac{2}{3}u_3-\tfrac{1}{3}u_4+\tfrac{1}{6}(\boldsymbol{n}_+\cdot\boldsymbol{n}_-)u_5\right)\boldsymbol{n}_+ +\left((\boldsymbol{n}_+\cdot\boldsymbol{n}_-)u_4+\tfrac{1}{2}u_5\right)\boldsymbol{n}_-=0,\\
\tfrac{1}{2}\left(u_2+(\boldsymbol{n}_+\cdot\boldsymbol{n}_-)u_1\right)\boldsymbol{t}_e +\left(\tfrac{2}{3}u_4-\tfrac{1}{3}u_3+\tfrac{1}{6}(\boldsymbol{n}_+\cdot\boldsymbol{n}_-)u_5\right)\boldsymbol{n}_- +\left((\boldsymbol{n}_+\cdot\boldsymbol{n}_-)u_3+\tfrac{1}{2}u_5\right)\boldsymbol{n}_+=0\\
\end{array} \right.$$
on $e$.
Since $\boldsymbol{t}_e$, $\boldsymbol{n}_+$ and $\boldsymbol{n}_-$ are linearly independent, and $-1< \boldsymbol{n}_+\cdot\boldsymbol{n}_-< 1$, we obtain that $u_i=0, 1\leq i\leq 5$ on $e$. As a consequence, for each $\boldsymbol{\tau}\in \mathbb{B}_{k}^{\div}(K;\bST)$, $\boldsymbol{\tau}|_e=0$ for each edge $e\in\mathcal{E}(K)$ and $D^{\alpha}\boldsymbol{\tau}(\delta)=0$ for any multi-indices $|\alpha|\leq 1$, $\forall \delta\in\mathcal{V}(K)$. Since $\boldsymbol{\tau}\cdot\boldsymbol{n}|_F=0$ for each face $F\in\mathcal{F}(K)$, from the geometric decomposition of $P_k^{(s)}(K;\bST)$ \eqref{geo decomp} and face basis of $\bST$ (Lemma~\ref{face_basis}), we obtain the proof.
\end{proof}
\begin{remark}
If $\boldsymbol{\sigma}\in\mathbb{B}^{\div}_k\left(K;\bST\right)$, then $\boldsymbol{\sigma}|_e=0$ for each edge $e\in\mathcal{E}(K)$. Therefore,
     \begin{align}
       \mathbb{B}_{k}^{\div}(K;\bST)\subset P^{(1)}_k(K;\bST), \quad \mathbb{B}_{k}^{\div}(K;\bST)=\mathbb{B}_{k}^{\div, (1)}(K;\bST)\label{div bubble equivalence}
    \end{align}
    due to vertex supersmoothness.
\end{remark}
\subsection{Unified BGG construction of bubble complexes}
We construct the bubble complexes in stages, guided by the BGG framework. Bubble elasticity complexes have been derived in \cite{arnold2008finite}, bubble divdiv complexes in \cite{chen2022finitedivdiv}, and related BGG constructions (including variants with enhanced smoothness) in \cite{chen2025complexes}. Since our BGG construction of the bubble conformal complexes requires bubble elasticity and bubble divdiv complexes with specific smoothness properties, we present a self-contained derivation tailored to this setting. In particular, to clarify the structure of the (often intricate) bubble spaces and to provide intuition for their role in exact sequences, we develop the BGG constructions from scratch, starting with the canonical \emph{bubble de~Rham} complex and the \emph{bubble Stokes} complex, together with their smoother variants.

 Denote by the $H(\curl)$ and the modified $H(\div)$ bubble spaces as    \begin{align*}
            \mathbb{B}_{k}^{\curl}(K;\mathbb{X}) &:=\{\boldsymbol{p}\in P_k(K;\mathbb  X):\; \boldsymbol{p} \times\boldsymbol{n}|_F=0, \;\forall \text{ face }F\},\quad \mathbb X \in\{ \mathbb R^3, \mathbb M\},\\
    \mathbb B_{k}^{\div *}(K;\mathbb{X})&:=\{\boldsymbol{p} \in \mathbb B_{k}^{\div }(K;\mathbb{X}):\boldsymbol{p}|_e =0,\;\forall\text{ edge }e\}, \quad \mathbb X \in\{ \mathbb R^3, \mathbb M\}. 
\end{align*}
We also define auxiliary polynomial spaces that vanish on edges:       \begin{align*}
            P^*_{k-2}(K;\mathbb R) &:= \{p\in P_{k-2}(K;\mathbb R):p|_e=0 \; \forall \text{ edge }e\}.\\
             P^{(2)*}_{k-2}(K;\mathbb R) &:= \{p\in P_{k-2}^{(2)}(K;\mathbb R):p|_e=0 \; \forall \text{ edge }e\}.
        \end{align*}
The following geometric decompositions are straightforward: 
\begin{equation}
        \begin{aligned}
            \mathbb{B}_k^{\curl}(K;\mathbb R^3)=&\oplus_{F\in \mathcal{F}(K)}b_FP_{k-3}(F;\mathbb{R}) \boldsymbol{n}\\
    &\oplus b_KP_{k-4}(K;\mathbb R^3).
        \end{aligned}
        \label{eq:representation of curl bubble sum}
    \end{equation}
    \begin{equation}
        \begin{aligned}
            \mathbb{B}_k^{\div*}(K;\mathbb R^3)=&\oplus_{F\in \mathcal{F}(K)}b_FP_{k-3}(F;\mathbb{R}) \boldsymbol{t}_{F,1}\\
            &\oplus_{F\in \mathcal{F}(K)}b_FP_{k-3}(F;\mathbb{R}) \boldsymbol{t}_{F,2}\\
    &\oplus b_KP_{k-4}(K;\mathbb R^3).
        \end{aligned}
        \label{eq:representation of div* bubble sum}
    \end{equation}
The corresponding decompositions for $\mathbb B_k^{\curl}(\boldsymbol)$ and $\mathbb B_k^{\div*}(\boldsymbol)$ can be obtained by replicating the vector versions for each row.
\begin{lemma}[Bubble de~Rham complex]\label{lemma:bubble-de-rham}
    The following sequence is exact:
\begin{align*}
    {0} \xrightarrow{} b_KP_{k-3}(K;\mathbb{R})\xrightarrow{\grad} \mathbb{B}_{k}^{\curl}(K;\mathbb{R}^3) \xrightarrow[]{\curl} \mathbb{B}^{\div}_{k-1}(K;\mathbb{R}^3)  \xrightarrow[]{\div} P_{k-2}(K;\mathbb R) \cap \mathbb R^\perp \xrightarrow{} 0. 
\end{align*}
\end{lemma}

\begin{corollary}[Smoother bubble de~Rham complex]\label{cor:smoother-de-rham}
    The following sequence is exact:
    \begin{align*}
        {0} \xrightarrow{} b_K^2P_{k-7}(K;\mathbb{R})\xrightarrow{\grad} b_KP_{k-4}(K;\mathbb{R}^3) \xrightarrow[]{\curl} \mathbb B_{k-1}^{\div *}(K;\mathbb{R}^3)\cap\ker(\div)  \xrightarrow[]{\div} 0 .
    \end{align*}
\end{corollary}
\begin{proof}The sequence is clearly a complex. We therefore only focus on the reverse inclusion.

    {\bf Step 1: $\ker(\div)\subset\ran(\curl)$.} For any $\boldsymbol{p}\in \ker(\div)$, from Lemma \ref{lemma:bubble-de-rham}, there exists $\boldsymbol{q} \in \mathbb{B}_{k}^{\curl}$ such that $\boldsymbol{p} =\curl\boldsymbol{q}$. Using geometric decompositions \eqref{eq:representation of curl bubble sum}, there exist $q_F\in P_{k-3}(F;\mathbb R)$ for each face $F$ and $\boldsymbol{r}\in P_{k-4}(K;\mathbb R^3)$ such that $\boldsymbol{q} = \sum_Fb_Fq_F\grad\lambda_F + b_K\boldsymbol{r}$. For each edge $e = F_+\cap F_-$, It holds that
    \begin{align*}
 \curl \boldsymbol{q}|_e = b_e (q_{F_+} -q_{F_-})|_e (\grad \lambda_{F_-}\times \grad \lambda_{F_+})=0.
    \end{align*}
    Therefore, $q_{F_+} = q_{F_-}$ on edge $e$. Hence, there exists ${q}'\in P_{k-3}(K;\mathbb R)$ such that $q'|_F = q_F|_F$ on each face $F$. Then, we obtain that $\boldsymbol q - \grad(b_K q')$ vanishes on all faces, hence in $b_KP_{k-4}(K;\mathbb R^3)$. Since $ \curl \left(\boldsymbol q - \grad(b_K q')\right) = \curl \boldsymbol{q} = \boldsymbol{p}$, we obtain the result.

    {\bf Step 2: $\ker(\curl)\subset\ran(\grad)$.} For any $\boldsymbol{p} \in \ker(\curl)$, from the exactness in Lemma \ref{lemma:bubble-de-rham}, there exists $q\in P_{k-3}(K;\mathbb R)$ such that $\boldsymbol{p} = \grad (b_K q) =\sum_Fb_Fq\grad\lambda_F + b_K\grad q$. Since $\boldsymbol{p}|_F=0$ for each face $F$, $q|_F=0$ for each face $F$, therefore $q\in b_KP_{k-7}(K;\mathbb R)$. This completes the proof.
\end{proof}
\begin{lemma}[Bubble Stokes complex \cite{neilan2015discrete}]
\label{lemma:bubble-stokes-complex}
    The following sequence is exact:
    \begin{align*}
        {0} \xrightarrow{} b_K^2P_{k-7}(K;\mathbb{R})\xrightarrow{\grad} b_K\mathbb{B}_{k-4}^{\curl}(K;\mathbb{R}^3) \xrightarrow[]{\curl} b_KP_{k-5}(K;\mathbb{R}^3)  \xrightarrow[]{\div} P_{k-2}^*(K;\mathbb R) \cap \mathbb R^\perp \xrightarrow{} 0 .   
    \end{align*}
\end{lemma}
\begin{corollary}
    \label{cor:smoother-bubble-stokes-complex}
    The following sequence is exact for $k\geq 7$:
    \begin{align*}
            {0} \xrightarrow{} b_K^2P_{k-7}(K;\mathbb{R})\xrightarrow{\grad} b_K\mathbb{B}_{k-4}^{\curl}(K;\mathbb{R}^3) \xrightarrow[]{\curl} b_KP_{k-5}^{(0)}(K;\mathbb{R}^3)  \xrightarrow[]{\div} P_{k-2}^{(2)*}(K;\mathbb R) \cap \mathbb R^\perp \xrightarrow{} 0 .
    \end{align*}
\end{corollary}
\begin{proof}
By the geometric decomposition \eqref{eq:representation of curl bubble sum}, for every $\boldsymbol{p}\in\mathbb B^{\curl}_k(K;\mathbb R^3)$ we have $\boldsymbol{p}|_e=0$ on each edge $e$. Consequently,
\[
\mathbb B^{\curl}_{k-4}(K;\mathbb R^3)\subset P^{(1)}_{k-4}(K;\mathbb R^3),
\qquad
b_K\mathbb B^{\curl}_{k-4}(K;\mathbb R^3)\subset P^{(4)}_k(K;\mathbb R^3).
\]
Moreover, since $b_KP^{(0)}_{k-5} \subset P^{(3)}_{k-1}$, Lemma~\ref{lemma:bubble-stokes-complex} implies that the displayed sequence is a complex. It remains to show exactness of the right part of the sequence.

Again by Lemma~\ref{lemma:bubble-stokes-complex},
\[
\left(b_KP_{k-5}(K;\mathbb R^3)\right)\cap \ker(\div)
= \curl\left(b_K\mathbb B^{\curl}_{k-4}(K;\mathbb R^3)\right)
\subset b_KP_{k-5}^{(0)}(K;\mathbb R^3),
\]
and hence
\[
\left(b_KP_{k-5}(K;\mathbb R^3)\right)\cap \ker(\div)
= \left(b_KP_{k-5}^{(0)}(K;\mathbb R^3)\right)\cap \ker(\div).
\]
Therefore,
\begin{align*}
\dim \ran(\div)
&= 3\bigl(\dim P_{k-5}(K;\mathbb R)-4\bigr) - \dim \ker(\div)\\
&= 3\bigl(\dim P_{k-5}(K;\mathbb R)-4\bigr) - \dim \mathbb B_{k-4}^{\curl}(K;\mathbb R^3) + \dim P_{k-7}(K;\mathbb R)\\
&= \dim P_{k-2}^{(2)*}(K;\mathbb R) - 1,
\end{align*}
where $\dim P_{k-2}^{(2)*}(K;\mathbb R) = 4\,\dim P_{k-5}^{(0)}(F;\mathbb R) + \dim P_{k-6}(K;\mathbb R)$ for $k\ge 7$ by geometric decompositions (Proposition~\ref{prop:geometric_decompositions}). Thus $\div$ is surjective, and the proof is complete.
\end{proof}
\begin{corollary}
\label{cor:smoothest-bubble-stokes-complex}
        The following sequence is exact for $k\geq 8$:
        \begin{align*}
            0\xrightarrow{} b_K^3P_{k-11}(K;\mathbb{R})\xrightarrow{\grad} b_K^2P_{k-8}(K;\mathbb{R}^3) \xrightarrow[]{\curl} b_K \mathbb B_{k-5}^{\div*}(K;\mathbb{R}^3) \xrightarrow[]{\div} \left(b_K P_{k-6}^{(0)}(K;\mathbb{R})\right) \cap\mathbb R^\perp\xrightarrow{}0.
        \end{align*}
\end{corollary}
\begin{proof}
    It is straightforward to check that the sequence is a complex. $\ker(\curl)=\ran(\grad)$ and $\ker(\div)=\ran(\curl)$ follow from Lemma~\ref{lemma:bubble-stokes-complex} via the same supersmoothness arguments used in Corollary~\ref{cor:smoother-de-rham}. To see $\div$ is a surjection, we use geometric decompositions \eqref{eq:representation of div* bubble sum} to compute
    \begin{align*}
        \dim\ran(\div) &=\dim \mathbb B_{k-5}^{\div*}(K;\mathbb{R}^3) - \dim \ker(\div)\\
        &=8 \dim P_{k-8}(F;\mathbb R)+\dim P_{k-9}(K;\mathbb R^3) - \dim \ran(\curl)\\
        &= 8 \dim P_{k-8}(F;\mathbb R)+\dim P_{k-9}(K;\mathbb R^3) - \dim P_{k-8}(K;\mathbb R^3) + \dim P_{k-11}(K;\mathbb R)\\
        & = \dim P_{k-6}(K;\mathbb R) -5.
    \end{align*}
    This completes the proof.
\end{proof}
\begin{remark}
    We adopt the convention that $P_k(K;\mathbb{X}) = {0}$ for $k < 0$, so bubble spaces of the form $b_K^s P_{k-m}(K;\mathbb{X})$ 
  are trivial whenever $k < m$. For $P_k^{(s)}(K;\mathbb{X})$, we always assume $k \geq 2s+1$. 
\end{remark}
\subsubsection{BGG construction for bubble elasticity complex} 
Recall the symmetric $\inc$ bubble space $\mathbb{B}_{k-4}^{1\inc}(K;\mathbb{S})$ with vanishing tangential-tangential trace \eqref{def:1inc_bubble}. We start by introducing its enrichment to $\boldsymbol$:
\begin{lemma}\label{lemma:decomposition_inc_bubble_bgg}Let $$\mathbb B^{1\inc}_{k}(K;\mathbb M):=\{\boldsymbol{\sigma}\in P_k(K;\mathbb M): \left(S^{-1}(\boldsymbol{\sigma}\times\boldsymbol{n})\right)\times\boldsymbol{n}|_F =0\}.$$
    Then, it holds that \begin{align*}
    \mathbb B^{1\inc}_{k}(K;\mathbb M)=\mathbb B^{1\inc}_k(K;\mathbb S) \,\oplus^{\perp}\,\mskw \, P_k(K;\mathbb R^3).
\end{align*}
\end{lemma}
\begin{proof}
Let $\boldsymbol{\sigma}\in P_k(K;\mathbb M)$ and write $\boldsymbol{\sigma}=\sym\boldsymbol{\sigma}+\skw\boldsymbol{\sigma}$.  
Since $\skw\boldsymbol{\sigma}=\mskw(\boldsymbol{a})$ for some $\boldsymbol{a}\in P_k(K;\mathbb R^3)$, the anti-commutativity of the BGG diagram \eqref{eq:bgg-construction-elasticity} gives (similar to $\curl S^{-1}\curl \mskw = -\curl \grad=0$):
\begin{equation}
\label{eq:anti-commutativity-inc}
    \begin{aligned}
        \big(S^{-1}(\mskw(\boldsymbol{a})\times\boldsymbol{n})\big)\times\boldsymbol{n}= (\boldsymbol{a}\boldsymbol{n}^T)\times\boldsymbol{n} =0.
    \end{aligned}
\end{equation}

Hence,
\[
\big(S^{-1}(\boldsymbol{\sigma}\times\boldsymbol{n})\big)\times\boldsymbol{n}
=-\,\boldsymbol{n}\times\sym\boldsymbol{\sigma}\times\boldsymbol{n},
\]
so that
\[
\boldsymbol{\sigma}\in\mathbb B^{1\inc}_k(K;\mathbb M)
\iff \sym\boldsymbol{\sigma}\in\mathbb B^{1\inc}_k(K;\mathbb S).
\]
In particular, $\mskw P_k(K;\mathbb R^3)\subset\mathbb B^{1\inc}_k(K;\mathbb M)$, giving the inclusion ``$\supset$''.  
Conversely, if $\boldsymbol{\sigma}\in\mathbb B^{1\inc}_k(K;\mathbb M)$, then $\sym\boldsymbol{\sigma}\in\mathbb B^{1\inc}_k(K;\mathbb S)$ and $\skw\boldsymbol{\sigma}\in \mskw P_k(K;\mathbb R^3)$, yielding ``$\subset$''.  
This completes the proof.
\end{proof}
Using bubble de~Rham complex (Lemma~\ref{lemma:bubble-de-rham}) and its smoother version (Corollary~\ref{cor:smoother-de-rham}), we construct BGG diagram for bubble elasticity complex as follows:
\begin{lemma}[BGG construction for bubble elasticity complex]\label{lemma:bgg-elasticity}
    The following two sequences are exact, and the diagram is anti-commutative:
    \[
\begin{tikzcd}[column sep=2.5em,row sep=3.2em]
0 \arrow[r] & b_K^2P_{k-7}(\mathbb R^3)\arrow[r,"\operatorname{grad}"] &
b_K\mathbb B_{k-4}^{1\inc}(\mathbb M)\arrow[r,"\operatorname{curl}"] &
\left(S\,\mathbb B_{k-1}^{\curl}(\mathbb{M})\right)\cap\ker(\div)\arrow[r,"\operatorname{div}"] & 0
\\[2pt]
0 \arrow[r] & b_KP_{k-4}(\mathbb R^3) \arrow[r,"\operatorname{grad}"]
\arrow[ru,"\mskw"] &
\mathbb B_{k-1}^{\curl}(\mathbb M)\cap\ker(\div S)\arrow[r,"\operatorname{curl}"] \arrow[ru,"S"] &
\mathbb{B}^{\div}_{k-2}(\mathbb{S})  \cap\ker(\div)\arrow[r,"\operatorname{div}"] \arrow[ru,"-2\operatorname{vskw}"] &
0.
\end{tikzcd}
\]
Here $\mskw$ is an injection, $S$ is a bijection, and $-2\vskw$ is a surjection.
\end{lemma}
\begin{proof}
$\mskw$ is an injection by Lemma~\ref{lemma:decomposition_inc_bubble_bgg} and $-2\vskw$ is a surjection by definition. $S$ is a surjection by construction. Anti-commutativity follows from the BGG construction \eqref{eq:bgg-construction-elasticity}.

 We first show the second sequence is exact. Since $2\vskw \curl = \div S$, from Lemma~\ref{lemma:bubble-de-rham}, we have
    \begin{align*}
    \mathbb B_{k-1}^{\curl}(K;\mathbb M)\cap\ker(\div S)\cap\ker(\curl) &= \mathbb B_{k-1}^{\curl}(K;\mathbb M)\cap\ker(\curl)\\ &
    = \grad \left(b_K^2 P_{k-8}(K;\mathbb R^3)\right),\\
        \mathbb{B}^{\div}_{k-2}(K;\mathbb{S})  \cap\ker(\div) &= \mathbb{B}^{\div}_{k-2}(K;\mathbb{M}) \cap\ker(\div) \cap\ker(\vskw)  \\
        &=\left(\curl \mathbb B_{k-1}^{\curl}(K;\mathbb M)\right)\cap\ker(\vskw)\\
        & = \curl\left(\mathbb B_{k-1}^{\curl}(K;\mathbb M)\cap\ker(\div S)\right).
    \end{align*}
    Hence the exactness of the second complex is established. Next, we show the first sequence is exact. First, we show 
    $$S\,\mathbb B_{k-1}^{\curl}(K;\mathbb{M}) \subset \mathbb B_{k-1}^{\div *}(K;\mathbb M).$$
    For any $\boldsymbol{\sigma}\in \mathbb B_{k-1}^{\curl}(K;\mathbb M)$, $\boldsymbol{\sigma}|_e=0$ for each edge $e$. Therefore, $S\boldsymbol{\sigma}|_e$ for each edge $e$. Furthermore, the BGG identity $\div S = 2\vskw\curl$ also holds if we replace $\nabla$ with $\boldsymbol{n}$: $$(S\boldsymbol{\sigma})\boldsymbol{n}|_F = -2\vskw(\boldsymbol{\sigma}\times\boldsymbol{n})|_F =0$$ for each face $F$, which implies $S$ is an injection between $\mathbb B_{k-1}^{\curl}(K;\mathbb{M}) $ and $\mathbb B_{k-1}^{\div *}(K;\mathbb M)$. Hence, for any $\boldsymbol{\sigma}\in\mathbb B_{k-1}^{\curl}(K;\mathbb{M})$ and $\div S\boldsymbol{\sigma}=0$, there exists $\boldsymbol{\tau}\in P_{k-4}(\mathbb M)$ such that $S\boldsymbol{\sigma} = \curl(b_K \boldsymbol{\tau})$ from Corollary~\ref{cor:smoother-de-rham}, and therefore, $S^{-1}\curl(b_K \boldsymbol{\tau}) \in \mathbb B^{\curl}_{k-1}(K;\mathbb M)$.
    By definition, for each face $F$ it holds that \begin{align*}
        0 & =\left(S^{-1}\curl(b_K\boldsymbol{\tau})\right)\times \boldsymbol{n}|_F = \tfrac{b_F}{h_F}\left(S^{-1}(\boldsymbol{\tau}\times\boldsymbol{n})\right)\times \boldsymbol{n}|_F.
    \end{align*}
    Hence, $\boldsymbol{\tau}\in \mathbb B_{k-4}^{1\inc}(K;\mathbb M)$, which implies $\ker(\div) \subset\ran(\curl)$. The reverse inclusion $\ran(\curl)\subset\ker(\div)$ follows from a similar expansion.

    Finally, for any $\boldsymbol{p}\in P_{k-7}(K;\mathbb R^3)$, we have$$\grad(b_K^2\boldsymbol{p})=b_K\left(2\sum_{F\in\mathcal{F}(K)}b_F\boldsymbol{p}\grad\lambda_F^T+b_K\grad\boldsymbol{p}\right),$$
and therefore $\grad$ maps to $b_K\mathbb B_{k-4}^{1\inc}(K;\mathbb M)$. Since $\mathbb B_{k-4}^{1\inc}(K;\mathbb M)\subset P_{k-4}(K;\boldsymbol)$, $\ran(\grad)=\ker(\curl)$ from Corollary~\ref{cor:smoother-de-rham}. The first sequence is also exact. This completes the proof.
\end{proof}
\begin{lemma}[Bubble elasticity complex  in \cite{arnold2008finite}]\label{lemma:bubble_elasticity}
    For $k\geq 6$, the following sequence is exact:
    \begin{align*}
        0 \xrightarrow{} b_K^2P_{k-7}(K;\mathbb{R}^3)\xrightarrow{\Def} b_K\mathbb{B}_{k-4}^{1\inc}(K;\mathbb{S}) \xrightarrow[]{\inc} \mathbb{B}^{\div}_{k-2}(K;\mathbb{S})  \cap\ker(\div)\xrightarrow[]{\div} 0.
    \end{align*}
\end{lemma}
\begin{proof}
We use the bubble BGG construction in Lemma~\ref{lemma:bgg-elasticity}. By Lemma~\ref{lemma:decomposition_inc_bubble_bgg}, $\ran(\mskw)^{\perp}$ is exactly $b_K\mathbb{B}_{k-4}^{1\inc}(K;\mathbb{S})$ and $\ker(\vskw)=\mathbb{B}^{\div}_{k-2}(K;\mathbb{S})  \cap\ker(\div)$. Since both sequences in the diagram of Lemma \ref{lemma:bgg-elasticity} are exact, the resulting sequence is also exact by the nature of BGG construction. This completes the proof.
\end{proof}

\subsubsection{BGG construction for the bubble divdiv complex}
\begin{lemma}[BGG construction for the bubble divdiv complex]\label{lemma:bgg-bubble-divdiv}
    For $k\geq 8$, the following two sequences are exact, and the diagram is anti-commutative:
    \[
\begin{tikzcd}[column sep=1.5em,row sep=3.2em]
0 \arrow[r] & b_K^3P_{k-11}(\mathbb R^3)\arrow[r,"\operatorname{grad}"] &
b_K^2P_{k-8}(\mathbb M)\arrow[r,"\operatorname{curl}"] &
b_K\mathbb B_{k-5}^{\div*}(\mathbb{M})\arrow[r,"\operatorname{div}"] & \left(b_K P_{k-6}^{(0)}(\mathbb{R}^3)\right) \cap(\mathbb R^3)^\perp \arrow[r] & 0
\\[2pt]
0 \arrow[r] & b_K^2 P_{k-8}(\mathbb R) \arrow[r,"\operatorname{grad}"]
\arrow[ru,"\iota"] &
b_K\mathbb B_{k-5}^{\curl}(\mathbb R^3)\arrow[r,"\operatorname{curl}"] \arrow[ru,"\mskw"] &
\left(b_KP_{k-6}^{(0)}(\mathbb R^3)\right) \cap (\mathbb R^3)^{\perp} \arrow[r,"\operatorname{div}"] \arrow[ru,"\operatorname{id}"] &
P_{k-2}^{(2)*}(\mathbb R)\cap P_1(\mathbb R) ^{\perp}\arrow[r]& 0.
\end{tikzcd}
\]
Here $\iota$ and $\mskw$ are injections, and $\operatorname{id}$ is a bijection.
\end{lemma}
\begin{proof}
$\iota$ is clearly an injection, and $\operatorname{id}$ is a bijection by construction. For any $\boldsymbol{p}\in \mathbb B_{k-5}^{\curl}(K;\mathbb R^3)$, it holds that $$\mskw (\boldsymbol{a}) \boldsymbol{n}|_F=\boldsymbol{a}\times \boldsymbol{n}|_F=0$$for each face $F$
and $\boldsymbol{a}|_e=0$ for each edge $e$. Therefore, $\mskw$ is also an injection. The anti-commutativity follows from the BGG identities \eqref{eq:bgg-construction-elasticity}.

The top sequence is exact by Corollary~\ref{cor:smoothest-bubble-stokes-complex}.  
For the bottom sequence, only the right part requires verification.  
Given $\boldsymbol{p}\in b_K\mathbb B_{k-5}^{\curl}(K;\mathbb R^3)$ and any constant $\boldsymbol{q}\in\mathbb R^3$, Stokes’ formula yields
\[
\int_K \operatorname{curl}\boldsymbol{p}\cdot \boldsymbol{q}
= \int_K \operatorname{curl}\boldsymbol{q}\cdot \boldsymbol{p}=0,
\]
hence
\[
\operatorname{curl}\big(b_K\mathbb B_{k-5}^{\curl}(K;\mathbb R^3)\big)
   \subset \left(\big(b_K P_{k-6}^{(0)}(K;\mathbb R^3)\big)\cap(\mathbb R^3)^\perp\right)\cap \ker(\div).
\]
By Corollary~\ref{cor:smoother-bubble-stokes-complex}, this image coincides with
\[
\big(b_K P_{k-6}^{(0)}(K;\mathbb R^3)\big)\cap\ker(\operatorname{div}),
\]
which implies $\ker(\operatorname{div})=\ran(\operatorname{curl})$ in the bottom row.  
Finally, surjectivity of $\operatorname{div}$ onto $P_{k-2}^{(2)*}(K;\mathbb R)\cap P_1(K;\mathbb R)^\perp$ again follows from Corollary~\ref{cor:smoother-bubble-stokes-complex} together with Stokes’ formula.
\end{proof}
Now, we introduce a new bubble space for the $\div\div$ operator:\[
\mathbb B_{k}^{\div\div*}(K;\mathbb S)
:= \left\{ \boldsymbol{\sigma}\in P_k(K;\mathbb S) \;:\;
\boldsymbol{n}^T \boldsymbol{\sigma}\,\boldsymbol{n}|_F = 0 \;\;\forall F\in\mathcal F(K),
\;\;\boldsymbol{\sigma}|_e = 0 \;\;\forall e\in\mathcal E(K) \right\}.
\]
\begin{lemma}\label{lemma:decomposition-of-divdiv-bubble}
    It holds that:
    $$
\sym\,\mathbb B_{k}^{\div*}(K;\mathbb M)=\mathbb B_{k}^{\div\div*}(K;\mathbb S), \quad  \mathbb B_{k}^{\div*}(K;\mathbb M)\cap \ker(\sym) = \mskw \mathbb B^{\curl}_k(K;\mathbb R^3).
$$
\end{lemma}
\begin{proof}
{\bf Step 1: $\ker(\sym)=\ran(\mskw)$.} 
Indeed, if $\boldsymbol{\tau}\in\mathbb B_{k}^{\div*}(K;\mathbb M)$ with $\sym\boldsymbol{\tau}=0$, then $\boldsymbol{\tau}=\mskw(\boldsymbol{v})$ for some $\boldsymbol{v}\in P_{k}(K;\mathbb R^3)$, and the boundary condition $\boldsymbol{\tau}\boldsymbol{n}|_F=0$ for each face $F$ implies $\boldsymbol{v}\times \boldsymbol{n}=0$, i.e.\ $\boldsymbol{v}\in\mathbb B_{k}^{\curl}(K;\mathbb R^3)$. The converse is also immediate.

{\bf Step 2: surjectivity of $\sym$.} For any $\boldsymbol{\tau}\in\mathbb B_{k}^{\div*}(K;\mathbb M)$, the condition $\boldsymbol{\tau}\boldsymbol{n}=0$ implies $\boldsymbol{n}^T(\sym\boldsymbol{\tau})\boldsymbol{n}=0$, hence $\sym\boldsymbol{\tau}\in\mathbb B_{k}^{\div\div*}(K;\mathbb S)$. Thus 
\[
\sym\,\mathbb B_{k}^{\div*}(K;\mathbb M)\subset\mathbb B_{k}^{\div\div*}(K;\mathbb S).
\]

For the reverse inclusion, we compare dimensions. By geometric decompositions \eqref{eq:representation of curl bubble sum} and \eqref{eq:representation of div* bubble sum}:
\[
\dim \mathbb B_{k}^{\div*}(K;\mathbb M)=12\,\dim P_{k-3}(F;\mathbb R^2)+\dim P_{k-4}(K;\mathbb M),
\]
\[
\dim (\mskw\,\mathbb B_{k}^{\curl}(\mathbb R^3))=4\,\dim P_{k-3}(F;\mathbb R)+\dim P_{k-4}(K;\mathbb R^3).
\]
Hence
\begin{align*}
    \dim (\sym\,\mathbb B_{k}^{\div*}(K;\mathbb M))
&=\dim \mathbb B_{k}^{\div*}(K;\mathbb M)-\dim (\mskw\,\mathbb B_{k}^{\curl}(K;\mathbb R^3))
\\&=20\,\dim P_{k-3}(F;\mathbb R)+6\dim P_{k-4}(K;\mathbb R).
\end{align*}
On the other hand, it is easy to show using geometric decompositions:
\[
\dim \mathbb B_{k}^{\div\div*}(K;\mathbb S)=20\,\dim P_{k-3}(F;\mathbb R)+\dim P_{k-4}(K;\bS).
\]
Thus $\sym\,\mathbb B_{k}^{\div*}(K;\mathbb M)=\mathbb B_{k}^{\div\div*}(K;\mathbb S)$.
Combining both statements completes the proof.
\end{proof}

\begin{lemma}[Bubble divdiv complex]\label{lemma:bubble-divdiv}
    The following sequence is exact for $k\geq8$:
    \begin{align*}
        0 \xrightarrow{} b_K^3P_{k-11}(K;\mathbb{R}^3)\xrightarrow{\dev\grad} b_K^2P_{k-8}(K;\mathbb{T}) \xrightarrow[]{\sym\curl}b_K \mathbb{B}^{\div\div*}_{k-5}(K;\mathbb{S})  \cap\ker(\div\div)\xrightarrow[]{\div\div} {0}.
    \end{align*}
\end{lemma}
\begin{proof}
It is straightforward to check that the sequence is a complex. Therefore, we focus on the reverse inclusion.

{\bf Step 1: $\ker(\div\div)\subset\ran(\sym\curl)$.}
By Lemma~\ref{lemma:decomposition-of-divdiv-bubble}, given $\boldsymbol{\sigma}\in \mathbb B_{k-5}^{\div\div*}(K;\mathbb S)$ with $\div\div(b_K\boldsymbol{\sigma})=0$, there exists 
$\boldsymbol{\tau}\in \mathbb B_{k-5}^{\div*}(K;\mathbb M)$ such that 
$\boldsymbol{\sigma}=\sym\boldsymbol{\tau}$ and $\div\div(b_K\boldsymbol{\tau})=0$.
By exactness of the bottom complex and the anti-commutativity of the diagram in Lemma~\ref{lemma:bgg-bubble-divdiv}, there is 
$\boldsymbol{p}\in\mathbb B_{k-5}^{\curl}(K;\mathbb R^3)$ with
\[
\div(b_K\boldsymbol{\tau})=\curl(b_K\boldsymbol{p})
= -\,\div\mskw(b_K\boldsymbol{p}).
\]
Then, exactness of the top complex in Lemma~\ref{lemma:bgg-bubble-divdiv} then yields some 
$\boldsymbol{q}\in P_{k-8}(K;\mathbb M)$ such that
\[
b_K\boldsymbol{\tau}=-\,b_K\mskw(\boldsymbol{p})+\curl(b_K^2\boldsymbol{q}).
\]
Applying $\sym$ and using $\sym\mskw=0$ gives
\[
b_K\boldsymbol{\sigma}
=\sym(b_K\boldsymbol{\tau})
=\sym\curl(b_K^2\boldsymbol{q}),
\]
hence $\boldsymbol{\sigma}\in\ran(\sym\curl)$ on the bubble space. This proves 
$\ker(\div\div)\subset\ran(\sym\curl)$.

{\bf Step 2: $\ker(\sym\curl)\subset\ran(\dev\grad)$.}
For any $\boldsymbol{\sigma}\in P_{k-8}(K;\mathbb T)$ and $\sym\curl(b_K^2\boldsymbol{\sigma})=0$, it holds that $$\curl(b_K^2\boldsymbol{\sigma})\in \mathbb B_{k}^{\div*}(K;\mathbb M)\cap \ker(\sym)$$ since the top sequence in Lemma~\ref{lemma:bgg-bubble-divdiv} is a complex. By Lemma~\ref{lemma:decomposition-of-divdiv-bubble},
there exists $\boldsymbol{p}\in \mathbb B^{\curl}_{k-5}(\mathbb R^3)$ such that $\curl(b_K^2\boldsymbol{\sigma})=\mskw(b_K\boldsymbol{p})$. Therefore, by anti-commutativity it holds that $$0=\div\mskw(b_K\boldsymbol{p})=-\curl(b_K\boldsymbol{p}).$$ By exactness of the bottom sequence in Lemma~\ref{lemma:bgg-bubble-divdiv}, there exists $q\in P_{k-8}(K;\mathbb R)$ such that $b_K\boldsymbol{p}=\grad(b_K^2q)$. It then follows from anti-commutativity that $$\curl(b_K^2\boldsymbol{\sigma})=\mskw\grad(b_K^2q) = -\curl(b_K^2q\mathbf{I}).$$ Hence, by exactness of the top sequence in Lemma \ref{lemma:bgg-bubble-divdiv}, there exists $\boldsymbol{u}\in P_{k-11}(\mathbb R^3)$ such that $b_K^2\boldsymbol{\sigma}=-b_K^2q\mathbf{I} +\grad(b_K^3\boldsymbol{u})$. Since $\boldsymbol{\sigma}$ is traceless, $b_K^2\boldsymbol{\sigma}=\dev\grad(b_K^3\boldsymbol{u})$. This completes the proof.
\end{proof}
\begin{remark}
In fact, one cannot apply BGG machinery to the diagram in Lemma~\ref{lemma:bgg-bubble-divdiv} directly to get a bubble divdiv complex, as $\ran(\mskw)^{\perp}$ is \emph{not} the divdiv bubble space $b_K\mathbb B_{k-5}^{\div\div*}(K;\mathbb S)$.  $b_K\mathbb B_{k-5}^{\div\div*}(K;\mathbb S)$ is \emph{not} even a subspace of $b_K\mathbb B_{k-5}^{\div*}(K;\mathbb M)$.
\end{remark}
\subsubsection{BGG construction for bubble conformal complex}
Finally, we construct the smoothest bubble conformal complex using BGG diagram that connects bubble divdiv complex (Lemma~\ref{lemma:bubble-divdiv}) and bubble elasticity complex (Lemma~\ref{lemma:bubble_elasticity}). Recall our $H(\cott)\cap H^2$ bubble space $b_K^2\mathbb B_{k-8}^{2\cott}(K;\bST)$. We parallel Lemma~\ref{lemma:decomposition_inc_bubble_bgg} and enrich $\mathbb B_{k}^{2\cott}(\bST)$ to $\mathbb T$.
\begin{lemma}\label{lemma:decomposition-for-T-cott-bubbles}
Let $$\mathbb B^{2\cott}_{k}(K;\bT):=\{\boldsymbol{\sigma}\in P_k(K;\mathbb T): \boldsymbol{n}\times S^{-1}\sym(\boldsymbol{\sigma}\times\boldsymbol{n})\times\boldsymbol{n}|_F =0\}.$$
   Then, it holds that \begin{align*}
    \mathbb B^{2\cott}_{k}(K;\bT)=\mathbb B^{2\cott}_k(K;\bST) \,\oplus^{\perp}\,\mskw \, P_k(K;\mathbb R^3).
\end{align*}
\end{lemma}
\begin{proof}The proof is similar to the proof of Lemma~\ref{lemma:decomposition_inc_bubble_bgg}.
Let $\boldsymbol{\sigma}\in P_k(K;\bT)$.  
From \eqref{eq:anti-commutativity-inc}, it holds that
\begin{align}
    \boldsymbol{n}\times \big(S^{-1}\sym(\skw(\boldsymbol{\sigma})\times\boldsymbol{n})\big)\times\boldsymbol{n}=    \sym\left(\boldsymbol{n}\times \big(S^{-1}(\skw(\boldsymbol{\sigma})\times\boldsymbol{n})\big)\times\boldsymbol{n}\right)=0.\label{eq:anti-commutativity-cott}
\end{align}
Hence, 
\begin{align}
    \boldsymbol{n}\times \big(S^{-1}\sym(\boldsymbol{\sigma}\times\boldsymbol{n})\big)\times\boldsymbol{n}
= \sym\big(\Pi_F\sym(\boldsymbol{\sigma})\times\boldsymbol{n}\big)=\tr_1^{\cott}\left(\sym(\boldsymbol{\sigma})\right).\label{eq:cott-trace-relations}
\end{align}
Therefore, by \eqref{cinc bubble in tr_1} in Proposition~\ref{prop:bubbles_in_tr1_bubble}, we have \[
\boldsymbol{\sigma}\in\mathbb B^{2\cott}_k(K;\mathbb T)
\iff \sym\boldsymbol{\sigma}\in\mathbb B^{2\cott}_k(K;\bST).
\]
Then, the ``$\supset$'' relation follows \eqref{eq:anti-commutativity-cott} and \eqref{eq:cott-trace-relations}, and ``$\subset$'' follows by decomposing $\boldsymbol{\sigma}\in\mathbb B^{2\cott}_k(K;\mathbb T)$ into $\boldsymbol{\sigma}=\sym \boldsymbol{\sigma}+\skw \boldsymbol{\sigma}$.
\end{proof}
\begin{lemma}[Bubble BGG construction for conformal complex]\label{lemma:bgg-bubble-conformal}
    For $ k\geq 10$, the following two discrete sequences are exact, and the diagram is anti-commutative:
\[
\begin{tikzcd}[column sep=2.5em, row sep=3.2em]
   0 \arrow[r] & b_K^3P_{k-11}(\mathbb R^3)  \arrow[r, "\operatorname{dev\,grad}"]  & b_K^2\mathbb B_{k-8}^{2\cott}(\mathbb T)  \arrow[r, "\operatorname{sym\,curl}"] &  \left( Sb_K\,\mathbb B_{k-5}^{1\inc }(\bS)\right)\cap\ker(\div\div)  \arrow[r, "\operatorname{div\,div}"]  & 0 \\
  0 \arrow[r] & b_K^2P_{k-8}(\mathbb R^3) \arrow[r, "\operatorname{def}"] \arrow[ru, "\mskw"] & \left(b_K \mathbb B_{k-5}^{1\inc }(\bS)\right)\cap\ker(\div\div S) \arrow[r, "\operatorname{inc}"] \arrow[ru, "S"] & \mathbb B_{k-3}^{\div}(\bS\cap\bT)\cap\ker(\div) \arrow[r, "\operatorname{div}"] \arrow[ru, "\Tr"] & 0.
\end{tikzcd}
\]
Here, $\mskw$ is an injection,  $S$ is a bijection, and $\Tr$ is a surjection. 
\end{lemma}
\begin{proof}We first note that $\mskw$ is an injection follows from Lemma~\ref{lemma:decomposition-for-T-cott-bubbles}. $\Tr$ is a surjection and $S$ is a bijection by definition of the diagram. It is also straightforward to show the exactness of the bottom sequence by bubble elasticity complex in Lemma~\ref{lemma:bubble_elasticity} and the anti-commutativity of the diagram, using similar arguments in the proof of exactness in Lemma~\ref{lemma:bgg-elasticity}. Therefore, we focus on showing the exactness of the first sequence in the diagram.

Recall the bubble divdiv complex derived in Lemma~\ref{lemma:bubble-divdiv}. We first show 
$$\left(Sb_K\mathbb B_{k-4}^{1\inc}(K;\bS)\right)\cap\ker(\div\div) \subset \left(b_K\mathbb B_{k-4}^{\div\div*}(K;\bS)\right)\cap\ker(\div\div). $$

For any $\boldsymbol{\sigma} \in \mathbb B_{k-4}^{1\inc}(K;\bS)$ and $\div\div S(b_K\boldsymbol{\sigma}) = 0$, from geometric decompositions of $\mathbb B_{k-4}^{1\inc}(K;\bS)$ \eqref{eq:1inc_geometric_decompositions}, there exists $p\in P_{k-6}(e;\mathbb R)$ such that $\boldsymbol{\sigma}|_e = b_ep|_e\sym(\boldsymbol{n}_+\boldsymbol{n}_-^T)$ for each edge $e=F_+\cap F_-$. By Proposition~\ref{proposition: geometric decompositions for div bubble}, elements in $\mathbb B_{k-3}^{\div}(K;\bST)$ vanish on edges. Therefore, it holds that 
\begin{align*}
0=\inc(b_K\boldsymbol{\sigma})|_e&=-\frac{b_e}{h_+h_-} \left(\boldsymbol{n}_+\times\boldsymbol{\sigma}|_e\times\boldsymbol{n}_- + \boldsymbol{n}_-\times\boldsymbol{\sigma}|_e\times\boldsymbol{n}_+\right) \\
&=-\frac{b_e^2p_e}{h_+h_-}|_e (\boldsymbol{n}_+\times\boldsymbol{n}_-)(\boldsymbol{n}_+\times\boldsymbol{n}_-)^T.
\end{align*}
This leads to $p|_e=0$, hence $\boldsymbol{\sigma}|_e=0$ for each edge $e$. Note that:
\begin{align*}
        \Tr\left(\tr_1^{\inc}({\boldsymbol{\sigma}})\right)|_F&=\Tr\left(\boldsymbol{n}\times \boldsymbol{\sigma}\times\boldsymbol{n}\right)|_F\\
        &=\boldsymbol{n}\cdot \boldsymbol{\sigma}\cdot\boldsymbol{n}-\Tr(\boldsymbol{\sigma})\\&=\boldsymbol{n}\cdot S\boldsymbol{\sigma}\cdot\boldsymbol{n}|_F.
    \end{align*}
Hence, $S (b_K\boldsymbol{\sigma})\in \left(b_K\mathbb B_{k-5}^{\div\div*}(\bS)\right) \cap \ker(\div\div)$. By Lemma~\ref{lemma:bubble-divdiv}, there exists $\boldsymbol{\tau}\in P_{k-8}(\bT)$, such that $$S(b_K\boldsymbol{\sigma})=\sym\curl(b_K^2\boldsymbol{\tau})=b_K\sym\left(\sum_{F\in\mathcal{F}(K)}\frac{2b_F}{h_F}\boldsymbol{\tau}\times\boldsymbol{n}+b_K\curl\boldsymbol{\tau}\right).$$
Since for each face $F$, by definition it holds that $\boldsymbol{n}\times\boldsymbol{\sigma}\times\boldsymbol{n}|_F=0$. Hence, we have $\boldsymbol{n}\times \left(S^{-1}\sym(\boldsymbol{\tau}\times\boldsymbol{n})\right)\times\boldsymbol{n}|_F=0$, which yields $\boldsymbol{\tau}\in \mathbb B_{k-8}^{2\cott}(K;\mathbb T)$ and hence $\ker(\div\div)\subset\ran(\sym\curl)$. The reverse inclusion $\ran(\sym\curl)\subset\ker(\div\div)$ also follows from a similar expansion.

Finally, for any $\boldsymbol{p}\in P_{k-11}(\mathbb R^3)$, we have$$\dev\grad(b_K^3\boldsymbol{p})=b_K^2\left(3\sum_{F\in\mathcal{F}(K)}b_F\dev(\boldsymbol{p}\grad\lambda_F^T)+b_K\dev\grad\boldsymbol{p}\right),$$
and therefore $\dev\grad$ maps to $b_K^2\mathbb B_{k-8}^{2\cott}(\bT)$. Since $\mathbb B_{k-8}^{2\cott}(\bT)\subset P_{k-8}(\bST)$, $\ran(\dev\grad)=\ker(\sym\curl)$ from Lemma~\ref{lemma:bubble-divdiv}. The first sequence is also exact. This completes the proof.
\end{proof}
\begin{theorem}\label{thm:conformal-bubble-seq3}
    The following sequence is exact for $k\geq 10$:
    \begin{align*}
            0 \xrightarrow{} b_K^3P_{k-11}(K;\mathbb{R}^3)\xrightarrow{\dev \Def} b_K^2\mathbb{B}_{k-8}^{2\cott}(K;\bST) \xrightarrow[]{\cott} &\mathbb{B}^{\div}_{k-3}(K;\bST)\cap \ker(\div) \xrightarrow[]{\div} 0.
    \end{align*}
\end{theorem}
\begin{proof}
    We use the BGG diagram in Lemma~\ref{lemma:bgg-bubble-conformal}. By Lemma~\ref{lemma:decomposition-for-T-cott-bubbles}, 
    $$\ran(\mskw)^{\perp}=b_K^2\mathbb{B}_{k-8}^{2\cott}(K;\bST), \quad \ker(\Tr) =\mathbb{B}^{\div}_{k-3}(K;\bST)\cap \ker(\div).  $$
Therefore, the complex follows from the BGG machinery. The resulting complex is also exact, since both complexes in the Lemma~\ref{lemma:bgg-bubble-conformal} are exact. This finishes the proof.
\end{proof}
\section{\texorpdfstring{$H(\operatorname{div})$}{H(div)}-conforming Finite Element Space and Divergence Stability}
\label{div section}
In this section, we first prove the bubble stability result in Theorem~\ref{thm:div-surjective property} by using Theorem~\ref{thm:conformal-bubble-seq3}. We then proceed to construct a conforming finite element pair for $H(\div;\bST)\times L^2(\mathbb R^3)$.  Adapting the stabilization ideas introduced for the Stokes pair in \cite{falk2013stokes,neilan2015discrete}, we establish an inf–sup (divergence) stability estimate for the new space.

\subsection{Supersmoothness in bubble surjectivity}
We note that for any $\boldsymbol{\sigma}\in\mathbb B^{\div}_{k-3}$, $\boldsymbol{\sigma}|_e=0$ for each edge $e\in \mathcal E(K)$ from supersmoothness (Proposition~\ref{proposition: geometric decompositions for div bubble}). Therefore $\mathbb B^{\div}_{k-3}\subset P^{(1)}_{k-3}$.
However, we show that $\mathbb B^{\div}_{k-3}\cap \ker(\div)$ exhibits even higher intrinsic supersmoothness at the vertices, which directly inherits from the image of $\cott$, as a result of Theorem~\ref{thm:conformal-bubble-seq3}.
\begin{lemma}
\label{prop:super-smoothness-div-free-bubble}
        For $k\geq 10$, it holds that
    \begin{align}
        \mathbb{B}_{k-3}^{\div}(K;\bST)\cap\ker(\div)\subset P^{(3)}_{k-3}(K;\bST).
    \end{align}
    As a result, we have
    \begin{align}
    \mathbb{B}_{k-3}^{\div, (s)}(K;\bST)\cap\ker(\div)=\mathbb{B}_{k-3}^{\div}(K;\bST)\cap\ker(\div)
\end{align}
for $1\leq s\leq3$.
\end{lemma}
\begin{proof}
Let $\boldsymbol{\tau} \in \mathbb{B}_{k-3}^{\div}(K;\bST) \cap \ker(\div)$. By Theorem~\ref{thm:conformal-bubble-seq3}, there exists $\boldsymbol{\sigma} \in \mathbb{B}_{k-8}^{2\cott}(K;\bST)$ such that
\[
\boldsymbol{\tau} = \cott(b_K^2 \boldsymbol{\sigma}).
\]
From geometric decompositions of $\mathbb{B}_{k-8}^{2\cott}(K;\bST)$, we have $\boldsymbol{\sigma}(\delta) = 0$ for all vertices $\delta \in \mathcal{V}(K)$, so $\boldsymbol{\sigma} \in P_{k-8}^{(0)}$.
Using the supersmoothness of the bubble function $b_K$, we obtain:
\[
b_K \boldsymbol{\sigma} \in P_{k-4}^{(3)}, \quad b_K^2 \boldsymbol{\sigma} \in P_k^{(6)}.
\]
Since $\cott$ is a third-order differential operator, it follows that
\[
\boldsymbol{\tau} = \cott(b_K^2 \boldsymbol{\sigma}) \in P_{k-3}^{(3)},
\]
 which finishes the proof.
\end{proof}
We are able to give a proof for Theorem~\ref{thm:div-surjective property}.
\begin{proof}[Proof of Theorem~\ref{thm:div-surjective property}]
We work with $\mathbb{B}_{k-3}^{\div, (s)}$ and $P^{(s-1)}_{k-4}$ assuming $k\geq 10$ for convenience. We first note that $\div$ maps $\mathbb{B}_{k-3}^{\div, (s)}(K;\bST)$ to $P^{(s-1)}_{k-4}(K;\mathbb{R}^3)\cap \boldsymbol{CK}^{\perp}$ by Stokes' formula. Therefore, it suffices to check the dimension of the image of $\div$ and $P^{(s-1)}_{k-4}(K;\mathbb{R}^3)\cap \boldsymbol{CK}^{\perp}$.

From the unisolvency of Neilan Stokes velocity element \cite[Section 3]{neilan2015discrete}, sets of DOFs
$$D^{\alpha}\boldsymbol{v},\quad |\alpha|\leq 2 \text{ on } \mathcal{V}(K), \quad\int_K\boldsymbol{v}\cdot P_{2}(K;\mathbb{R}^3)$$
are linearly independent for $\boldsymbol{v}\in P_{k-4}(K;\mathbb{R}^3)$, provided $k\geq 10$.
Since $\boldsymbol{CK}\subset P_2(K;\mathbb R^3)$, it follows that $$\operatorname{dim}P^{(s-1)}_{k-4}(K;\mathbb{R}^3)\cap\boldsymbol{CK}^{\perp}=\operatorname{dim}P_{k-4}(K;\mathbb{R}^3)-10-12\binom{s+2}{3}$$
 for $1\leq s\leq 3$.
From Lemma~\ref{prop:super-smoothness-div-free-bubble} and Theorem~\ref{thm:conformal-bubble-seq3}, we have
\begin{align*}
    \dim \mathbb{B}_{k-3}^{\div, (s)}(K;\bST)\cap\ker(\div)&=\dim \mathbb{B}_{k-3}^{\div}(K;\bST)\cap\ker(\div)\\
    & =\dim \mathbb B_{k-8}^{2\cott}(K;\bST) - \dim P_{k-11}(K;\mathbb R^3)
\end{align*}
for $1\leq s\leq3$. Since $k-3\geq 7\geq 2s+1$ for $1\leq s\leq3$, we also note that $\dim \mathbb B_{k-3}^{\div,(s)}(K;\bST)$ can be obtained via geometric decomposition in Proposition~\ref{proposition: geometric decompositions for div bubble}.
Therefore, through a simple dimension count, we have
\begin{align*}
&\quad\operatorname{dim}\div\mathbb{B}_{k-3}^{\div,(s)}(K;\bST)\\
&= \dim\mathbb{B}_{k-3}^{\div,(s)}(K;\bST)-\dim\mathbb{B}_{k-3}^{\div,(s)}(K;\bST)\cap \ker(\div)\\
& = 8\dim P_{k-6}^{(s-2)}(F;\mathbb R) + 5\dim P^{(s-3)}_{k-7}(K;\mathbb R) -6\dim P_{k-10}(e;\mathbb R) \\
&\quad-12\dim P_{k-11}(F;\mathbb R)-5\dim P_{k-12}(K;\mathbb R) + 3\dim P_{k-11}(K;\mathbb R) \\
&= \operatorname{dim}P_{k-4}^{(s-1)}(K;\mathbb{R}^3)\cap \boldsymbol{CK}^{\perp}-\tfrac{4}{3}s^3 + 4s^2 + \tfrac{28}{3}s - 28
\end{align*}
 for $1\leq s\leq3$. Thus,
\begin{align*}
    &\dim\div\mathbb{B}_{k-3}^{\div,(1)}(K;\bST)=\operatorname{dim}P_{k-4}^{(0)}(K;\mathbb{R}^3)\cap\boldsymbol{CK}^{\perp}-16,\\
    &\dim\div\mathbb{B}_{k-3}^{\div,(2)}(K;\bST)=\operatorname{dim}P_{k-4}^{(1)}(K;\mathbb{R}^3)\cap\boldsymbol{CK}^{\perp}-4,\\
     &\dim\div\mathbb{B}_{k-3}^{\div,(3)}(K;\bST)=\operatorname{dim}P_{k-4}^{(2)}(K;\mathbb{R}^3)\cap\boldsymbol{CK}^{\perp}.
\end{align*}
This completes the proof. 
\end{proof}
\begin{remark}
        Using a standard scaling argument, we can show that for any $\boldsymbol{v}\in P_{k-4}^{(2)}(K;\mathbb{R}^3)\cap\boldsymbol{CK}^{\perp}$, there exists $\boldsymbol{\tau}\in \mathbb{B}_{k-3}^{\div,(3)}(K;\bST)$ such that $\div\boldsymbol{\tau}=\boldsymbol{v}$ and $\lVert\boldsymbol{\tau}\rVert_{H(\div,K)}\leq C \lVert\boldsymbol{v}\rVert_{L^2(K)}$, where $C$ is independent of the mesh size $h$.
    \label{affine constant bubble}
\end{remark}
\subsection{Degrees of freedom}
For $k\geq 10,$ set the shape function as $P_{k-3}(K;\bST)$. A unisolvent set of degrees of freedom for $H(\div)$-conforming symmetric and traceless finite element space $\boldsymbol{\Sigma}^{\div}_{k-3,h}$ is locally given by
\begin{subequations}
\begin{align}
    D^{\alpha}\boldsymbol{\tau}(\delta),\quad &\forall|\alpha|\leq 3 ,\quad \forall\delta \in\mathcal{V}(K).\label{div_1}\\
    \int_{e} \boldsymbol{\tau}:\boldsymbol{q},\quad&\forall\boldsymbol{q}\in P_{k-11}(e;\bST),\quad\forall e\in \mathcal{E}(K).\label{div_2}\\
     \int_{F} \boldsymbol{q}\cdot\boldsymbol{\tau}\cdot\boldsymbol{n},\quad&\forall\boldsymbol{q}\in P_{k-6}^{(1)}(F;\mathbb{R}^3),\quad\forall F\in \mathcal{F}(K).\label{div_3}\\
      \int_{K} \boldsymbol{\tau}:\boldsymbol{q},\quad&\forall\boldsymbol{q}\in \mathbb{B}_{k-3}^{\div,(3)}(K;\bST).\label{div_4}
\end{align}
\end{subequations}
\begin{theorem}
    The above degrees of freedom are unisolvent. Furthermore, DOFs \eqref{div_1}-\eqref{div_3} determine the trace $\boldsymbol{\tau}\cdot\boldsymbol{n}$ uniquely on $\partial K$.
\end{theorem}
\begin{proof}
    The number of DOFs is
    \begin{align*}
        &\quad 4\cdot5\cdot20+30(k-10)+12(\frac{(k-4)(k-5)}{2}-9)+ \frac{5\,k^3 }{6}-\frac{17\,k^2 }{2}+\frac{77\,k}{3}-112\\
        &=\dim P_{k-3}(K;\bST).
    \end{align*}
    We note that $\dim \mathbb B_{k-3}^{\div,(3)}(K;\bST)$ could be determined from Proposition~\ref{proposition: geometric decompositions for div bubble}. Therefore, it suffices to show that $\boldsymbol{\tau}=0$ if all DOFs vanish. Notice the vanishing of \eqref{div_1} and \eqref{div_2} establishes that $\boldsymbol{\tau}\cdot\boldsymbol{n}|_F\in b_F P_{k-6}^{(1)}(F;\mathbb{R}^3)|_F$, where $b_F$ is the barycentric bubble of face $F$. Hence, if \eqref{div_1}-\eqref{div_3} vanish, $\boldsymbol{\tau}\cdot\boldsymbol{n}$ vanishes on $\partial K$ and as a consequence $\boldsymbol{\tau}\in\mathbb{B}_{k-3}^{\div,(3)}(K;\bST).$ Therefore, $\boldsymbol{\tau}\cdot\boldsymbol{n}$ is uniquely determined on faces and by \eqref{div_4} $\boldsymbol{\tau}=0$, which completes the proof. 
\end{proof}
Next, we define the discontinuous vector finite element space $\boldsymbol{V}_{k-4,h}$ with extra smoothness at vertices for $k\geq 10$. Choose the shape function as $P_{k-4}(K;\mathbb{R}^3)$ and a set of local DOFs is given by
\begin{subequations}
\begin{align}
    D^{\alpha}\boldsymbol{v}(\delta),\quad &\forall|\alpha|\leq 2 ,\quad \forall\delta \in\mathcal{V}(K).\label{Discontinous1}\\
      \int_{K} \boldsymbol{v}:\boldsymbol{q},\quad&\forall\boldsymbol{q}\in P^{(2)}_{k-4}(K;\mathbb{R}^3).\label{Discontinous2}
\end{align}
\end{subequations}
The unisolvency is valid by definition.
\subsection{Divergence stability}
At the end of the section, we present the divergence stability of the above finite element spaces, i.e. the operator $\div: \boldsymbol{\Sigma}^{\div}_{k-3,h}\longrightarrow \boldsymbol{V}_{k-4,h}$ is surjective. 
\begin{theorem}
\label{div_stability_FE}
    For any $\boldsymbol{v}\in \boldsymbol{V}_{k-4,h}$ provided $k\geq 11$, there exists $\boldsymbol{\tau}\in \boldsymbol{\Sigma}^{\div}_{k-3,h}$ such that $\div \boldsymbol{\tau}=\boldsymbol{v}$ and $\lVert\boldsymbol{\tau}\rVert_{H(\div,\Omega)}\leq C\lVert \boldsymbol{v}\rVert_{L^2(\Omega)}$, where $C$ is a constant independent of the mesh size $h$.
\end{theorem}
\begin{proof}
    To start with, from the BGG construction on the continuous level \cite{arnold2021complexes}, there exists $\boldsymbol{\tau}_1\in {H}^1(\Omega;\bST)$ such that
    \begin{align}
        \div\boldsymbol{\tau}_1=\boldsymbol{v} \text{ and } \lVert\boldsymbol{\tau}_1\rVert_{H^1(\Omega)}\leq C\lVert \boldsymbol{v}\rVert_{L^2(\Omega)}.
    \end{align}
    Denote by $\boldsymbol{I}_h\boldsymbol{\tau}_1\in {P}_{k-3}(\bST)$ the canonical Scott-Zhang interpolant \cite{scott1990finite}. We then determine $\boldsymbol{\tau'}$ by Neilan Stokes ${H}^1$-conforming elements with extra smoothness at vertices as follows:
    \begin{subequations}
    \begin{align}
        \boldsymbol{\tau'}(\delta)=&\boldsymbol{I}_h\boldsymbol{\tau}_1(\delta),\label{div_construction_vertices}\\
        D^{\alpha}\boldsymbol{\tau'}_{ii}(\delta)=&0,\quad \forall1\leq|\alpha|\leq 3,\nonumber\\
        D^{\alpha}(\partial_i\boldsymbol{\tau'}_{ij})(\delta)=&\tfrac{1}{2}D^{\alpha}\boldsymbol{v}_j(\delta),\quad i\neq j,\quad\forall0\leq|\alpha|\leq2,\nonumber\\
        D^{\alpha}(\partial_k\boldsymbol{\tau'}_{ij})(\delta)=&0,\quad i,j,k\text{ are different from each other},\nonumber\\&\quad\quad\forall0\leq|\alpha|\leq2,\quad \forall \delta\in \mathcal{V}(K).\nonumber\\
        \int_e \boldsymbol{\tau'}:\boldsymbol{q}=&\int_e \boldsymbol{I}_h\boldsymbol{\tau}_1:\boldsymbol{q},\quad \forall \boldsymbol{q}\in P_{k-11}(e;\bST),\quad \forall e\in\mathcal{E}(K).\\
        \int_e \partial_{\boldsymbol{n}_{e{\pm}}}\left(\boldsymbol{\tau'}\right):\boldsymbol{q}=&\int_e \partial_{\boldsymbol{n}_{e{\pm}}}\left(\boldsymbol{I}_h\boldsymbol{\tau}_1\right):\boldsymbol{q},\quad \forall \boldsymbol{q}\in P_{k-10}(e;\bST),\quad \forall e\in\mathcal{E}(K).\\
        \int_F \boldsymbol{\tau'}:\boldsymbol{q}=&\int_F \boldsymbol{\tau}_1:\boldsymbol{q},\quad \forall \boldsymbol{q}\in P_{k-9}(F;\bST),\quad \forall F\in\mathcal{F}(K).\label{div_construction_face}\\
\int_K\boldsymbol{\tau'}:\boldsymbol{q}=&\int_K \boldsymbol{I}_h\boldsymbol{\tau}_1:\boldsymbol{q},\quad \forall \boldsymbol{q}\in P_{k-7}^{(0)}(K;\bST).
    \end{align}
    \end{subequations}
    Note that this is exactly Neilan Stokes element adding extra orders of derivatives at vertices. Since $\boldsymbol{\tau'}$ has $C^3$ smoothness at vertices and satisfies global continuity, it follows that $\boldsymbol{\tau'}\in\boldsymbol{\Sigma}_{k-3,h}^{\div}$. Moreover, from a standard scaling argument similar to \cite[Lemma 4.5]{neilan2015discrete}, we can derive that \begin{align}        \lVert\boldsymbol{\tau'}\rVert_{H^1(\Omega)}\leq C\lVert\boldsymbol{v}\rVert_{L^2(\Omega)}.\label{scaling estimate}
    \end{align} By \eqref{div_construction_vertices}, we have
    \begin{align*}     D^{\alpha}\left(\div\boldsymbol{\tau'}\right)(\delta)=D^{\alpha}\boldsymbol{v}(\delta),\quad\forall0\leq|\alpha|\leq2,\quad \forall \delta\in \mathcal{V}(K).
    \end{align*}
    Note that $\boldsymbol{CK}$ consists of polynomials with degree no more than two, and we assume $k
    \geq 11$. By \eqref{div_construction_face} we obtain
    \begin{align*}
        \int_K\left(\div \boldsymbol{\tau'}- \boldsymbol{v}\right)\cdot\boldsymbol{q}=&-\int_K\left(\boldsymbol{\tau'}-\boldsymbol{\tau}_1\right): \dev\Def\boldsymbol{q}+\int_{\partial K}\boldsymbol{q}\cdot\left(\boldsymbol{\tau'}-\boldsymbol{\tau}_1\right)\cdot\boldsymbol{n}\\=&0,\quad \forall\boldsymbol{q}\in \boldsymbol{CK}.
    \end{align*}
    Thus, from Theorem~\ref{thm:div-surjective property} and Remark \ref{affine constant bubble}, there exists $\boldsymbol{\tau}_2\in \boldsymbol{\Sigma}_{k-3,h}^{\div}$, such that $\boldsymbol{\tau}_2|_K\in\mathbb{B}^{\div,(3)}_{k-3}(K;\bST)$ and 
    \begin{align*}
      \div\boldsymbol{\tau}_2=\div\boldsymbol{\tau'}-\boldsymbol{v} ,\quad\lVert\boldsymbol{\tau}_2\rVert_{H(\div,\Omega)}\leq C\lVert\div\boldsymbol{\tau'}-\boldsymbol{v}\rVert_{L^2(\Omega)}.
    \end{align*}
    Finally, we define $\boldsymbol{\tau}=\boldsymbol{\tau'}-\boldsymbol{\tau}_2$.
    Hence $\div\boldsymbol{\tau}=\boldsymbol{v}$ and by \eqref{scaling estimate} we have
    \begin{align*}  \lVert\boldsymbol{\tau}\rVert_{H(\div,\Omega)}&\leq C\left(\lVert\div\boldsymbol{\tau'}-\boldsymbol{v}\rVert_{L^2(\Omega)}+\lVert\boldsymbol{\tau'}\rVert_{H(\div,\Omega)}\right)\\
      &\leq C\lVert\boldsymbol{v}\rVert_{L^2(\Omega)}. 
    \end{align*}
    This completes the proof.
\end{proof}
\begin{corollary}
    The inf-sup condition \eqref{sec1:eq:inf-sup} holds for $k\geq 11$.
\end{corollary}

\section{\texorpdfstring{$H({\operatorname{cott}})$}{H(cott)}-conforming Finite Element Space}
\label{cinc section}

In this section, we present several bubble conformal complexes with smoothness less than the bubble complex in Theorem~\ref{thm:conformal-bubble-seq3}. We construct a family of ${H}(\cott)$-conforming symmetric and traceless finite element spaces suitable for use in a finite element complex.

Building on the analysis from the previous section, \(\boldsymbol{\Sigma}_{k-3,h}^{\div}\) requires at least \(C^3\) vertex smoothness for \(\div\)-stability. Accordingly, we impose \(C^6\) vertex smoothness for \(\boldsymbol{\Sigma}_{k,h}^{\cott}\) to ensure compatibility in the complex.

To construct \(\boldsymbol{\sigma} \in \boldsymbol{\Sigma}_{k,h}^{\cott}\), we first define degrees of freedom ensuring that \(\tr_1^{\cott}(\boldsymbol{\sigma})|_F\) and \(\boldsymbol{\sigma}|_e\) are single-valued. We then analyze \(\tr_2^{\cott}(\boldsymbol{\sigma})\) and \(\tr_3^{\cott}(\boldsymbol{\sigma})\) in 2D under the simplifying assumption that \(\tr_1^{\cott}(\boldsymbol{\sigma})|_{\partial K} = 0\) and \(\boldsymbol{\sigma}|_e = 0\), which isolates the local behavior on a single tetrahedron.

Inspired by the finite element systems framework \cite{christiansen2010finite,christiansen2018generalized}, which emphasizes that restrictions to lower-dimensional sub-cells should themselves form finite element complexes, we derive unisolvent degrees of freedom for $H(\div_F\div_F)$- and $H(\operatorname{rot}_F)$-conforming symmetric and traceless tensors on faces. These results are then used to design face and edge DOFs to ensure $\tr_2^{\cott}(\boldsymbol{\sigma})$ and $\tr_3^{\cott}(\boldsymbol{\sigma})$ are single-valued for \(\boldsymbol{\Sigma}_{k,h}^{\cott}\) in \eqref{cinc-DOFs}.
\subsection{Bubble conformal complex with less smoothness}

\begin{theorem}
\label{thm:bubble-seq-0}
For $k\geq10$, the following sequence is exact:
$$0 \xrightarrow{} b_KP_{k-3}(K;\mathbb{R}^3)\xrightarrow{\dev \Def} \mathbb{B}_{k}^{\cott}(K;\bST) \xrightarrow[]{\cott} \mathbb{B}^{\div}_{k-3}(K;\bST)\cap \ker(\div) \xrightarrow[]{\div} 0.$$
\end{theorem}
By the trace complexes in Section~\ref{trace section}, the above is a complex. 
Moreover, since $b_K^2\mathbb B_{k-8}^{2\cott}(\bST)\subset\mathbb B_k^{\cott}(\bST)$, 
the surjectivity of $\cott$ follows from Theorem~\ref{thm:conformal-bubble-seq3}.  
Hence, the only nontrivial part is the exactness of the first two terms: 
we must show $\ker(\cott)=\ran(\dev\Def)$ inside $\mathbb B_k^{\cott}(K;\bST)$.
\begin{lemma}{{\rm (Polynomial conformal complex \cite{vcap2023bounded})}}
    \label{polynomial_conformal}The following polynomial sequence is exact:
    \begin{align}\label{polynomial-complex}
       \boldsymbol{CK} \xrightarrow{\subset} P_{k+1}(K;\mathbb{R}^3)\xrightarrow{\dev \Def} P_k(K;\bST) \xrightarrow[]{\cott} P_{k-3}(K;\bST) \xrightarrow[]{\div}P_{k-4}(K;\mathbb{R}^3) \xrightarrow[]{}0.
    \end{align}
\end{lemma}
\begin{lemma}
\label{lemma:characterizations of CK}
$\boldsymbol{v}\in\boldsymbol{CK} $ can be uniquely determined by the following degrees of freedom:
\begin{align}
\int_e\boldsymbol{t}_e\cdot\curl\boldsymbol{v},&\quad e\in \mathcal{E}(K),\label{CK_1}\\
    \int_F\boldsymbol{v}\cdot\boldsymbol{n},&\quad F\in \mathcal{F}(K).\label{CK_2}
\end{align}
\end{lemma}
\begin{proof}
    Recall our definition of infinitesimal rigid motions $\boldsymbol{RM}$ in \eqref{def:RM}, the lowest order Raviart-Thomas space $\boldsymbol{RT}$ in \eqref{def:RT}, and the conformal Killing fields \eqref{def:CK}:
    \[
\boldsymbol{CK}=\ker(\dev\Def)
=\bigl\{(\boldsymbol{x}\!\cdot\!\boldsymbol{x})\boldsymbol{a}
-2(\boldsymbol{a}\!\cdot\!\boldsymbol{x})\boldsymbol{x}
+\boldsymbol{b}\times\boldsymbol{x}
+c\,\boldsymbol{x}
+\boldsymbol{d}:\;
\boldsymbol{a},\boldsymbol{b},\boldsymbol{d}\in\mathbb{R}^3,\ c\in\mathbb{R}\bigr\}.
\]
It holds that
\[
\nabla\times(\boldsymbol{b}\times\boldsymbol{x})=2\boldsymbol{b},
\qquad
\nabla\times\Big((\boldsymbol{x}\!\cdot\!\boldsymbol{x})\boldsymbol{a}
-2(\boldsymbol{a}\!\cdot\!\boldsymbol{x})\boldsymbol{x}\Big)
=-4\,\boldsymbol{a}\times\boldsymbol{x}.
\]
Hence \(\curl\boldsymbol{CK}=\boldsymbol{RM}\) and
\(\boldsymbol{CK}\cap\ker(\curl)=\boldsymbol{RT}\).
Therefore, from the unisolvency of the lowest order Nédélec and Raviart-Thomas element, $\boldsymbol{CK}$ can be uniquely determined by \eqref{CK_1} and \eqref{CK_2}.
\end{proof}

\begin{proof}[Proof of Theorem~\ref{thm:bubble-seq-0}]
We show $\ker(\cott) = \ran(\dev\Def)$.

\medskip\noindent\textbf{Step 1: Using the polynomial conformal complex.}
Let $\boldsymbol\sigma\in\mathbb B_{k}^{\cott}(K;\bST)$ with $\cott\boldsymbol\sigma=0$.
Since $\boldsymbol\sigma$ is polynomial, the polynomial conformal complex
(Lemma~\ref{polynomial_conformal}) yields $\boldsymbol v\in P_{k+1}(K;\mathbb R^3)$ such that
\[
\dev\Def\boldsymbol v=\boldsymbol\sigma,
\qquad \boldsymbol v \ \text{unique up to}\ \boldsymbol{CK}.
\]
We fix this ambiguity by imposing the following moments
\begin{align}
\int_e \boldsymbol t_e\cdot\curl\boldsymbol v = 0 \quad &\forall e\in\mathcal E(K), \label{eq:gauge-1}\\
\int_F \boldsymbol v\cdot \boldsymbol n = 0 \quad &\forall F\in\mathcal F(K). \label{eq:gauge-2}
\end{align}
By Lemma~\ref{lemma:characterizations of CK}, 
\eqref{eq:gauge-1} and \eqref{eq:gauge-2} uniquely determine $\boldsymbol v$.

\medskip\noindent\textbf{Step 2: Enforcing boundary vanishing conditions.}
Since $\boldsymbol\sigma\in\mathbb B_k^{\cott}$, all face traces $\tr_i(\boldsymbol\sigma)$ vanish on $\partial K$.
Substituting $\boldsymbol\sigma=\dev\Def\boldsymbol v$ into the trace identities of Lemma~\ref{trace complex identities part 1},
for every $F\in\mathcal F(K)$,
\begin{align}
\tr_1^{\cott}(\boldsymbol\sigma)|_F &= \tfrac12\Def_F(\boldsymbol v\times\boldsymbol n)-\tfrac12\sym\curl_F(\boldsymbol v\Pi_F)=0, \label{eq:tr1}\\
\tr_2^{\cott}(\boldsymbol\sigma)|_F &= -\Def_F\curl_F(\boldsymbol v\cdot\boldsymbol n)=0, \label{eq:tr2}\\
\tr_3^{\cott}(\boldsymbol\sigma)|_F &= \tfrac12\hess_F\big(\boldsymbol n\cdot\curl\boldsymbol v\big)=0. \label{eq:tr3}
\end{align}
From the edge identities \eqref{trace identity symcurl} and \eqref{new_form_trace_2}, for any edge $e\subset F$, we have
\begin{align}
\tfrac{1}{2}\,\partial_{\boldsymbol{t}_e}\!\left(\boldsymbol{t}_e\cdot\curl\boldsymbol{v}\right)
&=\boldsymbol{t}_e\cdot(\sym\curl\boldsymbol{\sigma})\cdot\boldsymbol{t}_e \nonumber\\
&=-\,\boldsymbol{n}_{F,e}\cdot\tr_2^{\cott}(\boldsymbol{\sigma})|_{F}\cdot\boldsymbol{n}_{F,e}
+\partial_{\boldsymbol{t}_{F,e}}\!\big(\boldsymbol{n}_{F,e}\cdot\boldsymbol{\sigma}\cdot\boldsymbol{n}\big).
\label{eq:edge-condition}
\end{align}
By Lemma~\ref{edge vanish 1}, $\boldsymbol\sigma|_e=0$ on every edge $e\in\mathcal E(K)$, so both terms on the right of \eqref{eq:edge-condition} vanish. Hence,
\[
\partial_{\boldsymbol{t}_e}\!\left(\boldsymbol{t}_e\cdot\curl\boldsymbol{v}\right)=0
\qquad\forall e\in\mathcal E(K).
\]
Combined with the condition \eqref{eq:gauge-1}, we obtain
\begin{equation}\label{eq:curl-vertex-vanish}
\boldsymbol t_e\cdot\curl\boldsymbol v\equiv 0\ \text{on }e,\qquad
(\curl\boldsymbol v)(\delta)=0\ \text{for every vertex }\delta\in\mathcal V(K).
\end{equation}

\medskip\noindent\textbf{Step 3: Facewise de~Rham reduction and affine boundary.}
From \eqref{eq:tr2} we have $\boldsymbol v\cdot\boldsymbol n\in\ker(\Def_F\curl_F)$ on each face $F$; thus $\boldsymbol v\cdot\boldsymbol n$ is a polynomial of degree $\le 2$.
From \eqref{eq:tr3} and \eqref{eq:curl-vertex-vanish} we get $\boldsymbol n\cdot\curl\boldsymbol v\in\ker(\hess_F)$ and vanishing at all vertices of $F$, hence
$\boldsymbol n\cdot\curl\boldsymbol v= \operatorname{rot}_F(\boldsymbol v\Pi_F)=0$ on $F$.
By the exactness of the 2D polynomial de~Rham complex, there is $u_F\in P_{k+2}(F)$ such that
\[
\boldsymbol v\Pi_F=\grad_F u_F \quad\text{on }F.
\]
Inserting into \eqref{eq:tr1} gives $\sym\curl_F(\boldsymbol v\Pi_F)=0$, i.e.
$\boldsymbol v\Pi_F\in\ker(\sym\curl_F)=\Pi_F(\boldsymbol{RT})$ \cite{chen2022finite}. 
Therefore, $\boldsymbol v$ is linear along every edge and at most quadratic facewise, hence facewise affine.

Let $F_+,F_-$ be adjacent faces with common edge $e$. Writing
$\boldsymbol v\Pi_{F_\pm}=a_\pm\,\boldsymbol x\Pi_{F_\pm}+\boldsymbol b_{F_\pm}\Pi_{F_\pm}$ on $F_\pm$,
$\boldsymbol v\cdot\boldsymbol t_e$ being single-valued along $e$ yields $a_+=a_-$. Hence, there exists a constant $a$ such that
\[
(\boldsymbol v-a\boldsymbol x)\Pi_F\ \text{is constant on every }F\in\mathcal F(K).
\]
Thus $\boldsymbol v-a\boldsymbol x$ is constant along every edge; combined with facewise affinity, 
$\boldsymbol v-a\boldsymbol x$ is constant on each face. 
Equivalently, there exists $\boldsymbol b\in\mathbb R^3$ with
\begin{equation}\label{eq:ax+b}
\boldsymbol v - a\,\boldsymbol x - \boldsymbol b \equiv 0 \qquad\text{on }\partial K.
\end{equation}
Combined with \eqref{eq:gauge-2}, we conclude that $a=0$ and $\boldsymbol b=0$. Therefore, $\boldsymbol{v}|_{\partial K}=0$, which completes the proof.
\end{proof}

\begin{corollary}\label{cor:bubble-conformal-seq-2}For $k\geq 10$, the following sequence is exact:
$$0 \xrightarrow{} b_K^2P_{k-7}(K;\mathbb{R}^3)\xrightarrow{\dev \Def} b_K\mathbb{B}_{k-4}^{1\cott}(K;\bST) \xrightarrow[]{\cott} \mathbb{B}^{\div}_{k-3}(K;\bST)\cap \ker(\div) \xrightarrow[]{\div} 0.$$
\end{corollary}
\begin{proof}
    The surjectivity of $\cott$ follows from Theorem \ref{thm:conformal-bubble-seq3}. Suppose $\boldsymbol{\sigma} \in b_K \mathbb{B}_{k-4}^{1\cott}(K; \bST)$ and $\cott \boldsymbol{\sigma} = 0$. From Theorem~\ref{thm:bubble-seq-0}, there exists $\boldsymbol{p} \in P_{k-3}(K; \mathbb{R}^3)$ such that
\[
\boldsymbol{\sigma} = \dev \Def (b_K \boldsymbol{p}).
\]
We expand:
\[
\dev \Def (b_K \boldsymbol{p}) = b_K \dev \Def \boldsymbol{p} + \sum_{F \in \mathcal{F}(K)} b_F \dev \sym(\grad \lambda_F \boldsymbol{p}^T).
\]
Since $\boldsymbol{\sigma}$ vanishes on $\partial K$, the second term must vanish on each face, which implies $\boldsymbol{p}|_{\partial K} = 0$. Therefore, $\boldsymbol{p} \in b_{K}P_{k-7}(K; \mathbb R^3)$, and $\boldsymbol{\sigma} \in \dev \Def (b_K^2 P_{k-7}(K; \mathbb{R}^3))$. This completes the proof.
\end{proof}
We note that Lemma~\ref{edge vanish 1} implies that for any \(\boldsymbol{\sigma} \in \mathbb{B}_{k}^{\cott}(K;\bST)\), we have \(D^\alpha \boldsymbol{\sigma}(\delta) = 0\) for all \(|\alpha| \leq 1\), \(\delta \in \mathcal{V}(K)\) due to edge vanishing. To align with the imposed \(C^6\) vertex smoothness of \(\boldsymbol{\Sigma}_{k,h}^{\cott}\), we define the enhanced bubble space as
\[
\mathbb{B}_{k}^{\cott, (6)}(K;\bST) := P_{k}^{(6)}(K;\bST) \cap \mathbb{B}_{k}^{\cott}(K;\bST),
\]
for \(k \geq 13\). We present a version of bubble complex that forms the foundation of our conforming element design.
\begin{theorem}\label{conformal_theorem_bubble_extra}
    The following bubble sequence with extra smoothness at vertices is exact for $k\geq 13$:
        \begin{align*}
         0 \xrightarrow{} b_KP_{k-3}^{(4)}(K;\mathbb{R}^3)\xrightarrow{\dev \Def} \mathbb{B}_{k}^{\cott,(6)}(K;\bST) \xrightarrow[]{\cott} \mathbb{B}^{\div, (3)}_{k-3}(K;\bST)\xrightarrow[]{\div} P_{k-4}^{(2)}(K;\mathbb{R}^3)\cap\boldsymbol{CK}^{\perp}\xrightarrow[]{}0.   
     \end{align*}
\end{theorem}

We present a lemma to aid our proof of Theorem~\ref{conformal_theorem_bubble_extra}. A constructive proof is given in Appendix \ref{appendix: devdef lemma}.

\begin{lemma}
\label{devdef lemma}Let $\alpha$, $\beta$ be multi-indices such that $|\alpha|\geq 3$ and $|\beta|=|\alpha|-1$. Assume $\boldsymbol{u}$ is a sufficiently smooth three-dimensional vector. Then elements of $D^{\alpha}\boldsymbol{u}$ can be expressed as linear combinations of elements in $D^{\beta}\left(\dev\Def\boldsymbol{u}\right)$.\end{lemma}
\begin{remark}
Lemma~\ref{devdef lemma} relates to conformal Korn inequalities \cite{feireisl2009singular,fuchs2009application,fuchs2010generalizations}. It does not hold for \(|\alpha| \leq 2\), but one has the identity:
\begin{align}
    \partial_{ij} u_m = \partial_i (\Def \boldsymbol{u})_{mj} + \partial_j (\Def \boldsymbol{u})_{mi} - \partial_m (\Def \boldsymbol{u})_{ij}.\label{def identity}
\end{align}
\end{remark}
\begin{proof}[Proof of Theorem~\ref{conformal_theorem_bubble_extra}]
It is straightforward to check that the sequence is a complex. We have shown $\div$ is a surjection in Theorem~\ref{thm:div-surjective property}. We prove the exactness in two steps.

\textbf{Step 1: $\ran(\cott)=\ker(\div)$.}  
Since \(b_K^2 \mathbb{B}_{k-8}^{2\cott}(K;\bST) \subset P_k^{(6)}(K;\bST)\), we obtain the hierarchy:
\[
b_K^2 \mathbb{B}_{k-8}^{2\cott}(K;\bST) \subset \mathbb{B}_{k}^{\cott,(6)}(K;\bST) \subset \mathbb{B}_{k}^{\cott}(K;\bST).
\]
Hence,
\[
\cott\mathbb{B}_{k}^{\cott,(6)}(K;\bST) = \mathbb{B}_{k-3}^{\div}(K;\bST) \cap \ker(\div) = \mathbb{B}_{k-3}^{\div,(3)}(K;\bST) \cap \ker(\div),
\]
by Theorem~\ref{thm:conformal-bubble-seq3} and Lemma~\ref{prop:super-smoothness-div-free-bubble}, since $\cott$ is already surjective under the smaller subspace $b_K^2 \mathbb{B}_{k-8}^{2\cott}(K;\bST) $.

\textbf{Step 2: $\ker(\cott)=\ran(\dev\Def)$.}  
Let \(\boldsymbol{\sigma} \in \mathbb{B}_k^{\cott,(6)} \cap \ker(\cott)\). By Theorem~\ref{thm:bubble-seq-0}, there exists \(\boldsymbol{p} \in P_{k-3}(K;\mathbb{R}^3)\) such that \(\boldsymbol{\sigma} = \dev\Def(b_K \boldsymbol{p})\). Since \(b_K \boldsymbol{p} \in P_{k+1}^{(2)}\) from supersmoothness of $b_K$ and \(D^\alpha(b_K \boldsymbol{p})\) can be expressed in terms of \(D^\beta(\boldsymbol{\sigma})\) with $|\beta|=|\alpha|-1$ for \(|\alpha| \geq 3\) (by Lemma~\ref{devdef lemma}), we conclude \(b_K \boldsymbol{p} \in P_{k+1}^{(7)}\), i.e., \(\boldsymbol{p} \in P_{k-3}^{(4)}\) from geometric decompositions (Proposition~\ref{prop:geometric_decompositions}). Thus, the sequence is exact.
\end{proof}
\subsection{Characterizations of traces in two dimensions}
In this subsection, we will give sets of DOFs for two-dimensional elements, which can guide our construction to determine $\tr_2^{\cott}(\boldsymbol{\sigma})$ and $\tr_3^{\cott}(\boldsymbol{\sigma})$ on $\partial K$. The key lies in the design of ${H}(\div_F\div_F)$- and ${H}(\operatorname{rot}_F)$-conforming symmetric and traceless tensors in two dimensions. Since we have assumed $\boldsymbol{\sigma}$ has $C^6$ continuity at vertices, we are going to give vertex derivative DOFs for $\tr_2^{\cott}(\boldsymbol{\sigma})$ and $\tr_3^{\cott}(\boldsymbol{\sigma})$ up to 5th and 4th order respectively. Denote $\mathbb{S}_F$ and $\mathbb{T}_F$ as symmetric and traceless matrices in $\Pi_F\mathbb{M}\Pi_F$ respectively. In this subsection, $\perp$ denotes $L^2$ orthogonality on face $F$.
\subsubsection{The $H(\operatorname{div}_F\div_F)$-conforming element}
Traces of ${H}(\div_F\div_F)$-conforming symmetric tensors in two dimensions \cite{chen2022divdivanddiv,chen2022finitedivdiv2dim,hu2022conformingdivdiv} are given as: 
\begin{align*} \operatorname{tr}_{e, 1}^{\div_F\div_F}(\boldsymbol{\sigma}) & :=\boldsymbol{n}_{F, e} \cdot \boldsymbol{\sigma} \cdot \boldsymbol{n}_{F, e}, \\ \operatorname{tr}_{e, 2}^{\div_F\div_F}(\boldsymbol{\sigma}) & :=\partial_{\boldsymbol{t}_{F,e}}\left(\boldsymbol{t}_{F, e} \cdot \boldsymbol{\sigma} \cdot \boldsymbol{n}_{F, e}\right)+\boldsymbol{n}_{F, e} \cdot \operatorname{div}_F \boldsymbol{\sigma} \end{align*}
for $e\in \mathcal{E}(F)$. The canonical symmetric $H(\div_F\div_F)$ bubble space with domain restricted on $F$, is defined as
$$\mathbb{B}^{\div_F\div_F}_k(F;\mathbb{S}_F)|_F:=\left\{\boldsymbol{\sigma} \in {P}_k(F ; \mathbb{S}_F)|_F: \operatorname{tr}_{e, 1}^{\div_F\div_F}(\boldsymbol{\sigma})|_e=\operatorname{tr}_{e, 2}^{\div_F\div_F}(\boldsymbol{\sigma})|_e=0, \forall e \in \mathcal{E}(F), \boldsymbol{\sigma}(\delta)={0} ,\forall \delta \in \mathcal{V}(F)\right\}.$$ 

\begin{lemma}[Two dimensional bubble divdiv complex \cite{chen2022finitedivdiv2dim, chen2022finite}]\label{two dim divdiv S bubble complex}
    The following sequence is exact: 
    $${0} \xrightarrow[]{\subset} b_F {P}_{k-2}\left(F ; \Pi_F\mathbb{R}^3\right) |_F\xrightarrow[]{\sym\curl_F} \mathbb{B}^{\div_F\div_F}_k(F;\mathbb{S}_F)|_F \xrightarrow[]{\div_F\div_F} {P}_{k-2}(F;\mathbb{R})|_F \cap {P}_1(F;
\mathbb{R})|_F^{\perp} \rightarrow 0.$$
\end{lemma}
\begin{lemma}{\em (Bubble Stokes complex \cite{vogelius1983right})}
\label{two dim stokes bubble complex}
The following sequence is exact: 
    $${0} \xrightarrow[]{\subset} b_F^2 {P}_{k-5}\left(F ; \mathbb{R}\right)|_F \xrightarrow[]{\grad_F} b_F P_{k-3}(F;\Pi_F\mathbb{R}^3) |_F\xrightarrow[]{\operatorname{rot}_F} {P}_{k-1}^{(0)}(F;\mathbb{R})|_F\cap\mathbb R^{\perp} \rightarrow 0.$$
\end{lemma}
Next, denote the symmetric and traceless ${H}(\div_F\div_F)$ bubble space with extra smoothness at vertices as
$$\mathbb{B}^{\div_F\div_F,(5)}_k(F;\mathbb{S}_F\cap\mathbb{T}_F)|_F:=P_k^{(5)}(F;\mathbb{S}_F\cap\mathbb{T}_F)|_{F}\cap\mathbb{B}^{\div_F\div_F}_k(F;\mathbb{S}_F)|_F.$$
The kernel and image of $\div_F\div_F$ of this space can be characterized via the following Lemma:
\begin{lemma}\label{lem:face-complex} 
    The following sequence is exact:
    $${0} \xrightarrow[]{\subset} b_F^2 {P}_{k-4}^{(3)}\left(F ; \mathbb{R}\right)|_F \xrightarrow[]{\Def_F\curl_F} \mathbb{B}^{\div_F\div_F,(5)}_k(F;\mathbb{S}_F\cap\mathbb{T}_F)|_F\xrightarrow[]{\div_F\div_F} {P}_{k-2}^{(3)}(F;\mathbb{R})|_F \cap {P}_1^{+}(F;
\mathbb{R})|_F^{\perp} \rightarrow 0,$$
\label{two dim divdiv ST bubble complex extra smoothness}
where$${P}_1^{+}(F;
\mathbb{R})|_F:={P}_1(F;
\mathbb{R})|_F\oplus\{(\Pi_F\boldsymbol{x})\cdot(\Pi_F\boldsymbol{x})\}.$$
\end{lemma}
We give a sketch of proof. A full proof is given in Appendix \ref{appendix: two dim divdiv ST bubble complex extra smoothness}.
\begin{proof}[Sketch of proof]
\textbf{Step 1: It is a complex.} The zero derivatives at vertices follow a pattern of 7-5-3, which match with the differential operator.
Note that for $u\in b_F^2 {P}_{k-4}^{(3)}\left(F ; \mathbb{R}\right)|_F $, $\grad_F u\in b_FP_{k-2}^{(4)}(F;\Pi_F\mathbb R^3)|_F$. From Lemma~\ref{two dim divdiv S bubble complex}, it holds that $$\Def_F\curl_F u=\sym\curl_F(\grad_Fu)\in \mathbb B_k^{\div_F\div_F,(5)}(F;\mathbb S_F\cap\mathbb T_F)|_F.$$
 Since $\Pi_F\boldsymbol{\sigma}\Pi_F\in \bS_F\cap\bT_F$ is traceless, the $L^2$ orthogonality with respect to ${P}_1^{+}(F;
\mathbb{R})$ comes from integration by parts:
$$\int_F\div_F\div_F\boldsymbol\sigma{q}=\int_F\Pi_F\boldsymbol{\sigma}\Pi_F:\nabla_F^2{q},\quad \forall q\in P_1^{+}(F;
\mathbb{R})|_F. $$

{\textbf {Step 2: Exactness in the middle.}}    If $\div_F\div_F\boldsymbol{\sigma}=0$, by Lemma~\ref{two dim divdiv S bubble complex}, there exists $\boldsymbol{v}\in b_FP_{k-2}(F;\Pi_F\mathbb{R}^3)|_F$ such that $\boldsymbol{\sigma}=\sym\curl_F\boldsymbol{v}$. Since $\boldsymbol{\sigma}$ is traceless, $\operatorname{rot}_F\boldsymbol{v}=0$. Therefore, from Lemma~\ref{two dim stokes bubble complex}, there exists $u\in P_{k-4}(F;\mathbb{R})|_F$ such that $\boldsymbol{v}=\grad_F(b_F^2u)$ hence $ \boldsymbol{\sigma}=\Def_F\curl_F (b_F^2u).$ $u$ inherits the extra smoothness of $\boldsymbol{\sigma}$ at vertices.  

{\textbf{Step 3: Surjectivity of $\div_F\div_F$.}}
By definition of $\mathbb B_k^{\div_F\div_F,(5)}(F;\mathbb S_F\cap\mathbb T_F)|_F$ we can give a lower bound of its dimension via vanishing of a set of degrees of freedoms. Since the kernel of $\div_F\div_F$ is established in {\textbf{Step 2}}, we therefore have a lower bound on the dimension of image space. The result follows from a dimension count.
\end{proof}

With the help of the bubble complex, we are ready to present a set of DOFs for $H(\div_F\div_F$-conforming symmetric and traceless finite element space in two dimensions. 
\begin{proposition}
\label{prop:divdiv_conforming_elements}
   For $k\geq 11$, choose the shape function as $P_{k}(F;\mathbb{S}_F\cap\mathbb{T}_F)|_F$. The following set of DOFs for the $H(\div_F\div_F;\bS_F\cap\bT_F)$-conforming finite element space is unisolvent:
\begin{subequations}
\begin{align}
D_F^{\alpha}\boldsymbol{\sigma}(\delta)&,\quad \forall 0\leq|\alpha|\leq5,\quad\forall \delta\in \mathcal{V}(F).\label{divdiv face vertex}\\
        \int_e\tr_{e,1}^{\div_F\div_F}(\boldsymbol{\sigma})q&,\quad\forall q\in P_{k-12}(e;\mathbb{R}),\quad\forall e\in\mathcal{E}(F).\label{divdiv face edge 1}\\
\int_e\tr_{e,2}^{\div_F\div_F}(\boldsymbol{\sigma})q&,\quad\forall q\in P_{k-11}(e;\mathbb{R}),\quad\forall e\in\mathcal{E}(F).\label{divdiv face edge 2}\\
\int_e\div_F\div_F\boldsymbol{\sigma}q&,\quad\forall q\in {P}_{k-10}(e;\mathbb{R}),\quad\forall e\in\mathcal{E}(F).\label{divdiv face edge 3}\\
\int_F\div_F\div_F\boldsymbol{\sigma}q&,\quad\forall q\in {P}_{k-5}^{(1)}(F;\mathbb{R})|_F \cap {P}_1^{+}(F;
\mathbb{R})|_F^{\perp}.\label{divdiv face 1}\\
\int_F\boldsymbol{\sigma}:\Def_F\curl_F(b_F^2{q})&,\quad\forall {q}\in {P}_{k-4}^{(3)}(F;\mathbb{R})|_F.\label{divdiv face 2}
\end{align}
\end{subequations}  
\end{proposition}
\begin{proof}
The total number of degrees of freedom is,
\begin{align*}
    &\quad 3\cdot 2\cdot \binom{7}{5} + 3\cdot\left(3k-30\right) +\binom{k-3}{2}-4 -3\cdot\binom{3}{1} +\binom{k-2}{2}-3\cdot\binom{5}{3}\\
    & = \dim P_{k}(F;\mathbb{S}_F\cap\mathbb{T}_F).
\end{align*}
Note that the vanishing of \eqref{divdiv face vertex}-\eqref{divdiv face edge 2} implies $\boldsymbol{\sigma}\in\mathbb B_k^{\div_F\div_F,(5)}(F;\mathbb S_F\cap\mathbb T_F)$. Using Lemma~\ref{two dim divdiv ST bubble complex extra smoothness} and the vanishing of \eqref{divdiv face edge 3}, we see that $\div_F\div_F\boldsymbol{\sigma}\in \left[b_F{P}_{k-5}^{(1)}(F;\mathbb{R})\right]|_F \cap {P}_1^{+}(F;
\mathbb{R})|_F^{\perp}$. Then, the vanishing of \eqref{divdiv face 1} further implies that $\div_F\div_F\boldsymbol{\sigma}=0$. Finally, from the exactness of bubble complexes in Lemma~\ref{two dim divdiv ST bubble complex extra smoothness}, $\boldsymbol{\sigma} \in \mathbb{B}^{\div_F\div_F,(5)}_k(F;\mathbb{S}_F \cap \mathbb{T}_F)\cap \ker(\div_F\div_F)$ must be of the form $\boldsymbol{\sigma} = \Def_F\curl_F(b_F^2 q)$ for some $q \in P_{k-4}^{(3)}(F; \mathbb{R})$.
$\boldsymbol{\sigma}=0$ follows from the vanishing of \eqref{divdiv face 2}. This completes the proof.
\end{proof}

\subsubsection{The $H(\operatorname{rot}_F)$-conforming element}
\begin{proposition}
\label{prop:rot_conforming_elements}
    For $k\geq9$, choose the shape function as $P_{k}(F;\mathbb{S}_F\cap\mathbb{T}_F)|_F$. The following set of DOFs for the $H(\operatorname{rot}_F;\bS_F\cap\bT_F)$-conforming finite element space in two dimensions is unisolvent:
\begin{subequations}
 \begin{align}
D_F^{\alpha}\boldsymbol{\sigma}(\delta)&,\quad \forall 0\leq|\alpha|\leq4,\quad\forall \delta\in \mathcal{V}(F).\label{vertex rot 1}\\
\int_e\boldsymbol{\sigma}:\boldsymbol{q}&,\quad\forall \boldsymbol{q}\in P_{k-10}(e;\mathbb{S}_F\cap\mathbb{T}_F),\quad\forall e\in\mathcal{E}(F).\label{edge rot 1}\\
\int_e\operatorname{rot}_F\boldsymbol{\sigma}\cdot\boldsymbol{q}&,\quad\forall \boldsymbol{q}\in P_{k-9}(e;\Pi_F\mathbb{R}^3),\quad\forall e\in\mathcal{E}(F).\label{edge rot 2}\\        \int_F\boldsymbol{\sigma}:\boldsymbol{q}&,\quad\forall q\in {P}_{k-6}^{(0)}(F;\mathbb{S}_F\cap\mathbb{T}_F)|_F.\label{face rot 1}
\end{align}
 \end{subequations}
\end{proposition}
\begin{proof}
    The total number of DOFs is
    \begin{align*}
        &\quad 3\cdot 2\cdot\binom{6}{2} + 3\cdot\left(2k-18+2k-16\right) + 2[\binom{k-4}{2}-3]=\dim P_k(F;\bS_F\cap\bT_F)|_F.
    \end{align*}
    Note that the vanishing of \eqref{vertex rot 1}-\eqref{edge rot 2} implies $\boldsymbol{\sigma}\in P^{(4)}_{k}(F;\bS_F\cap\bT_F)$, $\boldsymbol{\sigma}|_e=0$, and $\operatorname{rot}_F\boldsymbol{\sigma}|_e=0$ for each edge $e\in \mathcal{E}(F)$. On each edge $e\in \mathcal{E}(F)$, it holds that
    \begin{align*}
        0=\operatorname{rot}_F\boldsymbol{\sigma}&=-\div_F(\boldsymbol{n}\times\boldsymbol{\sigma})\\
        &=-\partial_{\boldsymbol{t}_{F,e}}\left(\boldsymbol{t}_{F,e}\cdot\left(\boldsymbol{n}\times\boldsymbol{\sigma}\right)\right)-\partial_{\boldsymbol{n}_{F,e}}(\boldsymbol{n}_{F,e}\cdot(\boldsymbol{n}\times\boldsymbol{\sigma}))\\
        &=-\partial_{\boldsymbol{t}_{F,e}}(\boldsymbol{n}_{F,e}\cdot\boldsymbol{\sigma})+\partial_{\boldsymbol{n}_{F,e}}(\boldsymbol{t}_{F,e}\cdot\boldsymbol{\sigma})\\
        &=\partial_{\boldsymbol{n}_{F,e}}(\boldsymbol{t}_{F,e}\cdot\boldsymbol{\sigma}),
    \end{align*}
    where the last equality follows from $\boldsymbol{\sigma}|_e=0$.
    Therefore, $\nabla_F(\boldsymbol{t}_{F,e}\cdot\boldsymbol{\sigma})=0$. Since $\boldsymbol{\sigma}$ is both symmetric and traceless, it holds that
    $$\boldsymbol{t}_{F,e}\cdot \boldsymbol{\sigma}\cdot \boldsymbol{t}_{F,e} = -\boldsymbol{n}_{F,e}\cdot \boldsymbol{\sigma}\cdot \boldsymbol{n}_{F,e}, \quad \boldsymbol{n}_{F,e}\cdot \boldsymbol{\sigma}\cdot \boldsymbol{t}_{F,e} = \boldsymbol{t}_{F,e}\cdot \boldsymbol{\sigma}\cdot \boldsymbol{n}_{F,e}, \quad e\in \mathcal{E}(F).$$
    Therefore, the edge trace $\boldsymbol{\sigma}\cdot\boldsymbol{t}_{F,e}|_e$ determines $\boldsymbol{\sigma}|_e$. It follows that $\nabla_F\boldsymbol{\sigma}|_e=0$ for each edge $e$. Hence, $\boldsymbol{\sigma}\in b_F^2P^{(0)}_{k-6}(F;\bS_F\cap\bT_F)|_F$ from Proposition~\ref{prop:geometric_decompositions}. Finally, the vanishing of \eqref{face rot 1} implies $\boldsymbol{\sigma}=0$. This completes the proof.
\end{proof}
\subsubsection{Traces of traces}
\begin{lemma}\label{lem:traces-of-traces}
    Assume $\boldsymbol{\sigma}\in \mathbb B_k^{\tr_1}(K;\bST)$ and $\boldsymbol{\sigma}|_e=0$, $\forall e\in\mathcal{E}(K)$. Then on edge $e\subset F$, it holds that
    \begin{align}      \tr_{e,1}^{\div_F\div_F}\left(\tr_2^{\cott}(\boldsymbol{\sigma})\right)&=-\boldsymbol{t}_e\cdot(\sym\curl\boldsymbol{\sigma})\cdot\boldsymbol{t}_e,\label{trace 2 edge trace 1}\\  \tr_{e,2}^{\div_F\div_F}\left(\tr_2^{\cott}(\boldsymbol{\sigma})\right)&=\boldsymbol{n}_{F,e}\cdot\left(2\partial_{\boldsymbol{t}_e}(\sym\curl\boldsymbol{\sigma})\cdot\boldsymbol{t}_e-\nabla\left(\boldsymbol{t}_e\cdot(\sym\curl\boldsymbol{\sigma})\cdot\boldsymbol{t}_e\right)\right),\label{trace 2 edge trace 2}\\
\boldsymbol{t}_{F,e}\cdot\tr_3^{\cott}(\boldsymbol{\sigma})\cdot\boldsymbol{t}_{F,e}&=\boldsymbol{n}\cdot\left(2\partial_{\boldsymbol{t}_e}(\sym\curl\boldsymbol{\sigma})\cdot\boldsymbol{t}_e-\nabla\left(\boldsymbol{t}_e\cdot(\sym\curl\boldsymbol{\sigma})\cdot\boldsymbol{t}_e\right)\right),\label{trace 3 edge trace 1}\\
\boldsymbol{n}_{F,e}\cdot\tr_3^{\cott}(\boldsymbol{\sigma})\cdot\boldsymbol{t}_{F,e}&=-\boldsymbol{t}_e\cdot\nabla\times(\sym\curl\boldsymbol{\sigma})\cdot\boldsymbol{t}_e-\tfrac{1}{2}\partial_{\boldsymbol{t}_e}(\boldsymbol{t}_e\cdot\div\boldsymbol{\sigma}).\label{trace 3 edge trace 2}
    \end{align}
\end{lemma}
\begin{proof}
    The domain of our arguments in the following discussion is restricted to the edge $e$ contained within the face $F$.
   Recalling the new representations of $\tr_2^{\cott}(\cdot)$ as given in \eqref{new_form_trace_2}, we can derive the following expression for term \eqref{trace 2 edge trace 1}:
    \begin{align*}
        \tr_{e,1}^{\div_F\div_F}(\tr_2^{\cott}(\boldsymbol\sigma))=-\boldsymbol{t}_{F,e}\cdot(\sym\curl\boldsymbol{\sigma})\cdot\boldsymbol{t}_{F,e}+\partial_{\boldsymbol{t}_{F,e}}(\boldsymbol{n}\cdot\boldsymbol{\sigma}\cdot\boldsymbol{n}_{F,e}).
    \end{align*}
    Hence, \eqref{trace 2 edge trace 1} holds since $\boldsymbol{\sigma}|_e=0$. Note that we also have
    \begin{align*}
        \tr_{e,2}^{\div_F\div_F}(\tr_2^{\cott}(\boldsymbol{\sigma}))&=2\partial_{\boldsymbol{t}_{F,e}}\left(\boldsymbol{t}_{F,e}\cdot\tr_2^{\cott}(\boldsymbol{\sigma})\cdot\boldsymbol{n}_{F,e}\right)+\partial_{\boldsymbol{n}_{F,e}}\left(\boldsymbol{n}_{F,e}\cdot\tr_2^{\cott}(\boldsymbol{\sigma})\cdot\boldsymbol{n}_{F,e}\right)\\        &=2\partial_{\boldsymbol{t}_{F,e}}\left(\boldsymbol{n}_{F,e}\cdot(\sym\curl\boldsymbol{\sigma})\cdot\boldsymbol{t}_{F,e}+\tfrac{1}{2}\partial_{\boldsymbol{t}_{F,e}}(\boldsymbol{t}_{F,e}\cdot\boldsymbol{\sigma}\cdot\boldsymbol{n})-\tfrac{1}{2}\partial_{\boldsymbol{n}_{F,e}}(\boldsymbol{n}_{F,e}\cdot\boldsymbol{\sigma}\cdot\boldsymbol{n})\right)\\
        &\quad\quad +\partial_{\boldsymbol{n}_{F,e}}\left(-\boldsymbol{t}_{F,e}\cdot(\sym\curl\boldsymbol{\sigma})\cdot\boldsymbol{t}_{F,e}+\partial_{\boldsymbol{t}_{F,e}}(\boldsymbol{n}_{F,e}\cdot\boldsymbol{\sigma}\cdot\boldsymbol{n})\right)\\       &=\boldsymbol{n}_{F,e}\cdot\left(2\partial_{\boldsymbol{t}_e}(\sym\curl\boldsymbol{\sigma})\cdot\boldsymbol{t}_e-\nabla\left(\boldsymbol{t}_e\cdot(\sym\curl\boldsymbol{\sigma})\cdot\boldsymbol{t}_e\right)\right).
    \end{align*}
    This proves identity \eqref{trace 2 edge trace 2}.
    
    Identity \eqref{trace 3 edge trace 1} is directly acquired from \eqref{new_form_trace_3}. Using \eqref{new_form_trace_3}, we also have
 $$\boldsymbol{n}_{F,e}\cdot\tr_3^{\cott}(\boldsymbol{\sigma})\cdot\boldsymbol{t}_{F,e}=-\boldsymbol{t}_{F,e}\cdot\nabla\times(\sym\curl\boldsymbol{\sigma})\cdot\boldsymbol{t}_{F,e}+\partial_{\boldsymbol{t}_{F,e}}\left(\boldsymbol{n}_{F,e}\cdot\left(\sym\curl\boldsymbol{\sigma}\right)\cdot\boldsymbol{n}\right).$$
Since 
$\boldsymbol{\sigma}|_e=0$ and  $\boldsymbol{t}_{F,e}\cdot\boldsymbol{\sigma}\cdot\boldsymbol{n}_{F,e}=-\boldsymbol{t}_{F,e}\cdot\tr_1^{\cott}(\boldsymbol{\sigma})|_F\cdot\boldsymbol{t}_{F,e}=0$ on $F$, it holds that
 $$\boldsymbol{t}_{F,e}\cdot\div\boldsymbol{\sigma}=\partial_{\boldsymbol{n}}(\boldsymbol{t}_{F,e}\cdot\boldsymbol{\sigma}\cdot \boldsymbol{n})+\partial_{\boldsymbol{n}_{F,e}}(\boldsymbol{t}_{F,e}\cdot\boldsymbol{\sigma}\cdot \boldsymbol{n}_{F,e})+\partial_{\boldsymbol{t}_{F,e}}(\boldsymbol{t}_{F,e}\cdot\boldsymbol{\sigma}\cdot \boldsymbol{t}_{F,e})=\partial_{\boldsymbol{n}}(\boldsymbol{t}_{F,e}\cdot\boldsymbol{\sigma}\cdot \boldsymbol{n}).$$
 A direct expansion shows that:
 \begin{align*}   \boldsymbol{n}_{F,e}\cdot\left(\sym\curl\boldsymbol{\sigma}\right)\cdot\boldsymbol{n}&=\tfrac{1}{2}\left(\partial_{\boldsymbol{t}_{F,e}}\left(\boldsymbol{n}\cdot\boldsymbol{\sigma}\cdot\boldsymbol{n}\right)-\partial_{\boldsymbol{n}}\left(\boldsymbol{t}_{F,e}\cdot\boldsymbol{\sigma}\cdot\boldsymbol{n}\right)\right)\\
 &\quad+\tfrac{1}{2}\left(\partial_{\boldsymbol{n}_{F,e}}\left(\boldsymbol{t}_{F,e}\cdot\boldsymbol{\sigma}\cdot\boldsymbol{n}_{F,e}\right)-\partial_{\boldsymbol{t}_{F,e}}\left(\boldsymbol{n}_{F,e}\cdot\boldsymbol{\sigma}\cdot\boldsymbol{n}_{F,e}\right)\right)\\
 &=-\tfrac{1}{2}\boldsymbol{t}_{F,e}\cdot\div\boldsymbol{\sigma}.
 \end{align*}
This proves the identity \eqref{trace 3 edge trace 2}.
\end{proof}
\begin{remark}
 \label{edge trace remark}
    Denote the edge traces of $\boldsymbol{\sigma}$ by
\begin{align}    
\tr_{e,1}^{\cott}(\boldsymbol{\sigma})&:=\boldsymbol{t}_e\cdot(\sym\curl\boldsymbol{\sigma})\cdot\boldsymbol{t}_e,\\
    \tr_{e,2}^{\cott}(\boldsymbol{\sigma})&:=2\partial_{\boldsymbol{t}_e}(\sym\curl\boldsymbol{\sigma})\cdot\boldsymbol{t}_e-\nabla\left(\boldsymbol{t}_e\cdot(\sym\curl\boldsymbol{\sigma})\cdot\boldsymbol{t}_e\right),\\
    \tr_{e,3}^{\cott}(\boldsymbol{\sigma})&:=-\boldsymbol{t}_e\cdot\nabla\times(\sym\curl\boldsymbol{\sigma})\cdot\boldsymbol{t}_e-\tfrac{1}{2}\partial_{\boldsymbol{t}_e}(\boldsymbol{t}_e\cdot\div\boldsymbol{\sigma}).
\end{align}
 If $\boldsymbol{\sigma}\in\mathbb B_k^{\tr_1}(K;\bST)\cap P_{k}^{(6)}(K;\bST)$ and $\boldsymbol{\sigma}|_e=0$ for each edge $e$, $\Tr\left(\tr_3^{\cott}(\boldsymbol{\sigma})\right)|_{\partial K}=0$ follows from Lemma~\ref{trace_couple}. Then it holds that
\begin{align*}
     \tr_2^{\cott}(\boldsymbol{\sigma})|_F&\in P^{(5)}_{k-1}(F;\bS_F\cap\bT_F)|_F,\\
          \tr_3^{\cott}(\boldsymbol{\sigma})|_F&\in P^{(4)}_{k-2}(F;\bS_F\cap\bT_F)|_F.
 \end{align*}Let $\boldsymbol{n}_{e\pm}$ be the two orthogonal unit vectors perpendicular to the tangential vector $\boldsymbol{t}_e$ of the edge $e$. If we further assume that $\tr_{e,1}^{\cott}(\boldsymbol{\sigma})|_e=\boldsymbol{n}_{e\pm}\cdot\tr_{e,2}^{\cott}(\boldsymbol{\sigma})|_e=\tr_{e,3}^{\cott}(\boldsymbol{\sigma})|_e=0$,  the above lemma establishes that
 \begin{align}
     \tr_2^{\cott}(\boldsymbol{\sigma})|_F&\in \mathbb B^{\div_F\div_F,(5)}_{k-1}(F;\mathbb S_F\cap\mathbb T_F)|_F.
 \end{align}
 Since $\tr_3^{\cott}(\boldsymbol{\sigma})\cdot\boldsymbol{t}_{F,e}|_{F,e}=0$ implies $\tr_3^{\cott}(\boldsymbol{\sigma})|_{F,e}=0$ (see the proof of Proposition~\ref{prop:rot_conforming_elements}), we have
 \begin{align}   
 \tr_3^{\cott}(\boldsymbol{\sigma})|_F&\in b_FP_{k-5}^{(2)}(F;\mathbb S_F\cap\mathbb T_F)|_F.
 \end{align}
 The above arguments shall be utilized below in the design of DOFs for the $H(\cott)$-conforming element.
\end{remark}
\subsection{Degrees of freedom}
Inspired by the DOFs given for two-dimensional elements, now we are ready to give a set of DOFs for the ${H}(\cott)$-conforming symmetric and traceless finite element space $\boldsymbol{\Sigma}_{k,h}^{\cott}$.
For $k\geq 13$, choose the shape function as $P_k(K;\bST)$ and the DOFs for $\boldsymbol{\Sigma}_{k,h}^{\cott}$ are given as follows:
\begin{subequations}
\label{cinc-DOFs}
\begin{align}
D^{\alpha}\boldsymbol{\sigma}(\delta)&,\quad \forall 0\leq|\alpha|\leq6,\quad\forall \delta\in \mathcal{V}(K).\label{cinc vertex 1}\\
\int_e\boldsymbol{\sigma}:\boldsymbol{q}&,\quad\forall q\in P_{k-14}(e;\bST),\quad\forall e\in\mathcal{E}(K).\label{cinc edge 1}\\ 
 \int_e\tr_{e,1}^{\cott}(\boldsymbol{\sigma}){q}&,\quad\forall {q}\in P_{k-13}(e;\mathbb{R}),\quad\forall e\in\mathcal{E}(K).\label{cinc edge 2}\\     \int_e\boldsymbol{n}_{e\pm}\cdot\tr_{e,2}^{\cott}(\boldsymbol{\sigma}){q}&,\quad\forall {q}\in P_{k-12}(e;\mathbb{R}),\quad\forall e\in\mathcal{E}(K).\label{cinc edge 3}\\   
  \int_e\tr_{e,3}^{\cott}(\boldsymbol{\sigma}){q}&,\quad\forall {q}\in P_{k-12}(e;\mathbb{R}),\quad\forall e\in\mathcal{E}(K).\label{cinc edge 4}\\  
\int_e\cott\boldsymbol{\sigma}:\boldsymbol{q}&,\quad\forall \boldsymbol{q}\in P_{k-11}(e;\bST),\quad\forall e\in\mathcal{E}(K).\label{cinc edge 5}\\
\int_F\tr_1^{\cott}(\boldsymbol{\sigma}):\boldsymbol{q}&,\quad\forall \boldsymbol{q}\in {P}_{k-3}^{(4)}(F;\mathbb{S}_F\cap\mathbb{T}_F),\quad\forall F\in\mathcal{F}(K).\label{cinc face 1}\\
 \int_F\boldsymbol{n}\cdot\cott\boldsymbol{\sigma}\cdot\boldsymbol{n}q&,\quad\forall q\in {P}_{k-6}^{(1)}(F;\mathbb{R}) \cap {P}_1^{+}(F;
\mathbb{R})^{\perp},\quad\forall F\in\mathcal{F}(K).\label{cinc face 2}\\
\int_F\tr_2^{\cott}(\boldsymbol{\sigma}):\Def_F\curl_F(b_F^2{q})&,\quad\forall {q}\in {P}_{k-5}^{(3)}(F;\mathbb{R}),\quad\forall F\in\mathcal{F}(K).\label{cinc face 3}\\
\int_F\tr_3^{\cott}(\boldsymbol{\sigma}):\boldsymbol{q}&,\quad\forall \boldsymbol{q}\in {P}_{k-8}^{(0)}(F;\mathbb{S}_F\cap\mathbb{T}_F),\quad\forall F\in\mathcal{F}(K).\label{cinc face 4}\\
\int_K\boldsymbol{\sigma}:\boldsymbol{q}&,\quad\forall \boldsymbol{q}\in \mathbb{B}_{k}^{\cott,(6)}(K;\bST).\label{cinc element 2}
\end{align}
\end{subequations}

\begin{theorem}
    For $k\geq 13$, the above DOFs are unisolvent for $P_k(K;\bST)$. Furthermore, $\boldsymbol{\Sigma}_{k,h}^{\cott}$ satisfies the ${H}(\cott)$-conformity.
\end{theorem}
\begin{proof}
We first compute the dimension of $\mathbb{B}_{k}^{\cott,(6)}(K;\bST)$ using Theorem \ref{conformal_theorem_bubble_extra} and Theorem \ref{thm:conformal-bubble-seq3}:
    \begin{align*}
        \dim\mathbb{B}_{k}^{\cott,(6)}(K;\bST)&=\dim P_{k-3}^{(4)}(K;\mathbb{R}^3)+\dim\mathbb{B}^{\div, (3)}_{k-3}(K;\bST)\cap\ker(\div)\\
        & = \dim P_{k-3}^{(4)}(K;\mathbb{R}^3) + \dim \mathbb B_{k-8}^{2\cott}(K;\bST) - \dim P_{k-11}(K;\mathbb R^3)\\
        &=\frac{5k^3}{6} - 7k^2 + \frac{127k}{6} - 399.
    \end{align*}
    The number of local DOFs is
    \begin{align*}        
    &\quad 4\cdot5\cdot84+6(14k-160) \\&\quad+4\left[2\left[\binom{k-1}{2}-3\binom{6}{2}\right]+\binom{k-4}{2}-4-3\binom{3}{2}+\binom{k-3}{2}-3\binom{5}{2}+2\left[\binom{k-6}{2}-3\right]\right]\\
    &\quad+\frac{5k^3}{6} - 7k^2 + \frac{127k}{6} - 399=\dim P_{k}(K;\bST).
    \end{align*}
For $\boldsymbol{\sigma}\in P_{k}(K;\bST)$, we aim to demonstrate that the vanishing of DOFs \eqref{cinc vertex 1}-\eqref{cinc face 4} implies $\boldsymbol{\sigma}\in \mathbb{B}_k^{\cott,(6)}(K;\bST).$ The proof is divided into two parts:

\textbf{Step 1: $\tr_1^{\cott}(\boldsymbol{\sigma
    })|_{\partial K}$ vanishes.} From the vanishing of \eqref{cinc vertex 1}, \eqref{cinc edge 1} and \eqref{cinc face 1}, $\boldsymbol{\sigma}|_e=0$ for each edge $e$ and $\boldsymbol{\sigma}\in \mathbb{B}_{k}^{\tr_1}(K;\bST)\cap P_{k}^{(6)}(K;\bST)$. 

\textbf{Step 2: $\tr_2^{\cott}(\boldsymbol{\sigma
    })|_{\partial K}$ and $\tr_3^{\cott}(\boldsymbol{\sigma
    })|_{\partial K}$ vanish.} From the vanishing of \eqref{cinc edge 2}-\eqref{cinc edge 4} and Remark \ref{edge trace remark}, 
     \begin{align}    \tr_2^{\cott}(\boldsymbol{\sigma})|_F&\in \mathbb B^{\div_F\div_F,(5)}_{k-1}(F;\mathbb S_F\cap\mathbb T_F)|_F,\label{conformity trace 2 in bubble}\\
 \tr_3^{\cott}(\boldsymbol{\sigma})|_F&\in b_FP_{k-5}^{(2)}(F;\mathbb S_F\cap\mathbb T_F)|_F.\label{conformity trace 3 in bubble}
 \end{align}
 From the vanishing of \eqref{cinc edge 5} and Lemma~\ref{trace complex identity part 2}, for each edge $e\subset F$ we have
 \begin{align}
0=\boldsymbol{n}\cdot\cott\boldsymbol{\sigma} \cdot\boldsymbol{n}|_e&=-\div_F\div_F\tr_2^{\cott}\left(\boldsymbol{\sigma}\right)|_{F,e},\label{conformity trace 2 edge vanish}\\    0=\boldsymbol{n}\times\cott\boldsymbol{\sigma} \cdot\boldsymbol{n}|_e&=\operatorname{rot}_F\tr_3^{\cott}\left(\boldsymbol{\sigma}\right)|_{F,e}.\label{conformity trace 3 edge vanish}
\end{align}
With Lemma~\ref{two dim divdiv ST bubble complex extra smoothness}, we have
\begin{align}   \div_F\div_F\tr_2^{\cott}(\boldsymbol{\sigma})|_F\in b_FP_{k-6}^{(1)}(F;\mathbb R)|_F\cap {P}_1^{+}(F;
\mathbb{R})|_F^{\perp}.\label{conformity divdiv in bubble}
\end{align}
Combining \eqref{conformity trace 2 in bubble} and \eqref{conformity divdiv in bubble} with the vanishing of \eqref{cinc face 2}, we obtain that
 \begin{align}
 \div_F\div_F\tr_2^{\cott}(\boldsymbol{\sigma})|_F=\boldsymbol{n}\cdot\cott\boldsymbol{\sigma}\cdot\boldsymbol{n}|_F=0
 \end{align}
 for each face $F$.
Hence,
\begin{align}    \tr_2^{\cott}(\boldsymbol{\sigma})|_F&\in \mathbb B^{\div_F\div_F,(5)}_{k-1}(F;\mathbb S_F\cap\mathbb T_F)|_F\cap \ker(\div_F\div_F).
 \end{align}
 Moreover, \eqref{conformity trace 3 in bubble} and \eqref{conformity trace 3 edge vanish} imply that
 \begin{align}     \tr_3^{\cott}(\boldsymbol{\sigma})|_F&\in b_F^2P_{k-8}^{(0)}(F;\mathbb S_F\cap\mathbb T_F)|_F.
 \end{align}
 It follows from Lemma~\ref{two dim divdiv ST bubble complex extra smoothness} and the vanishing of \eqref{cinc face 3} and \eqref{cinc face 4} that $\tr_2^{\cott}(\boldsymbol{\sigma})|_{\partial K}=\tr_3^{\cott}(\boldsymbol{\sigma})|_{\partial K}=0$.

The preceding arguments demonstrate that $\boldsymbol{\sigma}\in \mathbb{B}^{\cott,(6)}_k(K;\bST)$, therefore, unisolvency is established. Moreover, it is noteworthy that all traces are uniquely determined by the DOFs on sub-simplices, and any function in $\boldsymbol{\Sigma}_{k,h}^{\cott}$ is single-valued on edges. Consequently, in line with Theorem~\ref{sufficient conforming condition}, we conclude that $\boldsymbol{\Sigma}_{k,h}^{\cott}$ satisfies the ${H}(\cott)$-conformity. This completes the proof.
\end{proof}
\section{A Finite Element Conformal Complex}
\label{finite element complex section}
We are finally able to present a discrete finite element conformal complex. The ${H}(\cott)$-conforming finite element space $\boldsymbol{\Sigma}_{k,h}^{\cott}$ (DOFs \eqref{cinc vertex 1}-\eqref{cinc element 2}), ${H}(\div)$-conforming finite element space $\boldsymbol{\Sigma}_{k-3,h}^{\div}$ (DOFs \eqref{div_1}-\eqref{div_4}) and discontinuous finite element space $\boldsymbol{V}_{k-4,h}$ (DOFs \eqref{Discontinous1}-\eqref{Discontinous2}) were defined in the previous sections. To ensure the complex property, we choose the ${H}^1$-conforming finite element space $\boldsymbol{U}_{k+1,h}$ as the Neilan Stokes element \cite{neilan2015discrete} with $C^7$ smoothness at vertices. For $k\geq 14$, choose the shape function as $P_{k+1}(K;\mathbb{R}^3)$ and the DOFs for $\boldsymbol{U}_{k+1,h}$ are defined as follows:
\begin{subequations}
     \begin{align}
      D^{\alpha}\boldsymbol{u}(\delta),&\quad 0\leq|\alpha|\leq7,\quad \forall \delta\in \mathcal{V}(K).\\
        \int_e \boldsymbol{u}\cdot\boldsymbol{q},&\quad \forall \boldsymbol{q}\in P_{k-15}(e;\mathbb{R}^3),\quad \forall e\in\mathcal{E}(K).\\
        \int_e \partial_{\boldsymbol{n}_{e{\pm}}}\left(\boldsymbol{u}\right)\cdot\boldsymbol{q},&\quad \forall \boldsymbol{q}\in P_{k-14}(e;\mathbb{R}^3),\quad \forall e\in\mathcal{E}(K).\\
        \int_F \boldsymbol{u}\cdot\boldsymbol{q},&\quad \forall \boldsymbol{q}\in P_{k-5}^{(3)}(F;\mathbb{R}^3),\quad \forall F\in\mathcal{F}(K).\\       \int_K\boldsymbol{u}\cdot\boldsymbol{q},&\quad \forall \boldsymbol{q}\in P_{k-3}^{(4)}(K;\mathbb{R}^3).
    \end{align}
    \end{subequations}
    Note that any $\boldsymbol{u}\in\boldsymbol{U}_{k+1,h}$ has $C^7$ continuity at vertices, $C^1$ continuity on edges and $C^0$ continuity globally.

    {We note that the finite element space $\boldsymbol{U}_{k+1,h}$ satisfies the exactness property $\dev\Def \boldsymbol{U}_{k+1,h} = \boldsymbol{\Sigma}_{k,h}^{\cott}\cap \ker(\cott)$. This will be verified in Theorem~\ref{thm:finite-element-complex} below, using the exactness of the continuous and polynomial conformal complexes \eqref{eq:sec1:conformal-complex}, \eqref{polynomial-complex} with argument for the vertex and edge supersmoothness. The exactness of the entire complex is a result of a dimension count.}
    
    \begin{lemma}
        Let $\boldsymbol{u}$ be a sufficiently smooth three-dimensional vector. It holds that
        \label{cinc edge trace identities}
        \begin{align}
            \tr_{e,1}^{\cott}(\dev\Def\boldsymbol{u})&=\tfrac{1}{2}\partial_{\boldsymbol{t}_e}\left(\boldsymbol{t}_e\cdot\curl\boldsymbol{u}\right),\label{trace e 1 cinc devdef}\\
            \tr_{e,2}^{\cott}(\dev\Def\boldsymbol{u})&=\tfrac{1}{2}\partial^2_{\boldsymbol{t}_e}\left(\curl\boldsymbol{u}\right),\\
            \tr_{e,3}^{\cott}(\dev\Def\boldsymbol{u})&=-\tfrac{1}{3}\partial^2_{\boldsymbol{t}_e}\left(\div\boldsymbol{u}\right).
        \end{align}
    \end{lemma}
    \begin{proof}
        The first two identities follow directly from \eqref{trace identity symcurl}.  
        As for the third identity, since $\curl\curl\boldsymbol{u}=\nabla(\div\boldsymbol{u})-\Delta\boldsymbol{u}$, we have
        \begin{align*}
            \tr_{e,3}^{\cott}(\dev\Def\boldsymbol{u})&=-\boldsymbol{t}_e\cdot\nabla\times\left(\tfrac{1}{2}\Def\left(\curl\boldsymbol{u}\right)\right)\cdot\boldsymbol{t}_e-\tfrac{1}{2}\partial_{\boldsymbol{t}_e}\left(\boldsymbol{t}_e\cdot\left(\Def\boldsymbol{u}-\tfrac{1}{3}(\div\boldsymbol{u})\mathbf{I}\right)\cdot\nabla\right)\\
            &=-\tfrac{1}{4}\partial_{\boldsymbol{t}_e}\left(\curl\curl\boldsymbol{u}+\Delta\boldsymbol{u}\right)\cdot\boldsymbol{t}_e-\tfrac{1}{4}\partial_{\boldsymbol{t}_e}^2\left(\div\boldsymbol{u}\right)+\tfrac{1}{6}\partial^2_{\boldsymbol{t}_e}\left(\div\boldsymbol{u}\right)\\
            &=-\tfrac{1}{3}\partial^2_{\boldsymbol{t}_e}\left(\div\boldsymbol{u}\right).
        \end{align*}
        This completes the proof.
    \end{proof}
    \begin{theorem}
        \label{thm:finite-element-complex}
        The following sequence is exact for $k\geq 14$:
        \begin{align}
            \boldsymbol{CK} \xrightarrow[]{\subset} \boldsymbol{U}_{k+1,h}\xrightarrow{\dev \Def} \boldsymbol{\Sigma}^{\cott}_{k,h} \xrightarrow[]{\cott} \boldsymbol{\Sigma}^{\div}_{k-3,h}\xrightarrow[]{\div} \boldsymbol{V}_{k-4,h}\xrightarrow[]{}0.  \label{finite element complex}
        \end{align}
    \end{theorem}
    \begin{proof}
First, we verify that \eqref{finite element complex} indeed forms a complex. We first note that $\div\boldsymbol{\Sigma}^{\div}_{k-3,h} = \boldsymbol{V}_{k-4,h}$ was established in Theorem~\ref{div_stability_FE}. For each $\boldsymbol{\sigma}\in \boldsymbol{\Sigma}^{\cott}_{k,h}$, since $\boldsymbol{\sigma}$ has $C^6$ continuity at vertices (DOFs \eqref{cinc vertex 1}), $\cott\boldsymbol{\sigma}|_e$ is single-valued (DOFs \eqref{cinc edge 5}), the inclusion $\cott \boldsymbol{\Sigma}^{\cott}_{k,h} \subset \boldsymbol{\Sigma}^{\div}_{k-3,h}$ follows from the conformity of $\boldsymbol{\Sigma}^{\cott}_{k,h}$ and Lemma~\ref{trace complex identity part 2}. Furthermore, the inclusion $\dev\Def \boldsymbol{U}_{k+1,h} \subset \boldsymbol{\Sigma}_{k,h}^{\cott}$ follows from Lemma~\ref{trace complex identities part 1} and Lemma~\ref{cinc edge trace identities}.

Next we show that {$\dev\Def \boldsymbol{U}_{k+1,h}=\boldsymbol{\Sigma}_{k,h}^{\cott}\cap\ker(\cott)$.} Utilizing the conformal complex at the continuous level \eqref{eq:sec1:conformal-complex} and the polynomial conformal complex \eqref{polynomial_conformal}, we know that for any $\boldsymbol{\sigma}\in \boldsymbol{\Sigma}_{k,h}^{\cott}\cap\ker(\cott)$ there exists $\boldsymbol{u}$ satisfying $\boldsymbol{u}|_K\in P_{k+1}(K;\mathbb{R}^3)$, globally continuous and $\boldsymbol{\sigma}=\dev\Def\boldsymbol{u}$. It remains to show extra smoothness at edges ($C^1$) and vertices  ($C^7$).
        
From the identity    
\begin{align}            \boldsymbol{u}\nabla^T=\grad\boldsymbol{u}=\dev\Def\boldsymbol{u}+\tfrac{1}{2}\operatorname{mskw}(\curl\boldsymbol{u})+\tfrac{1}{3}(\div\boldsymbol{u})\mathbf{I},
        \label{gradient identity}
        \end{align}
        we obtain that
        \begin{align}
            \tfrac{1}{2}\curl\boldsymbol{u}\times\boldsymbol{t}_e+\tfrac{1}{3}(\div\boldsymbol{u})\boldsymbol{t}_e=\partial_{\boldsymbol{t}_e}(\boldsymbol{u})-\left(\dev\Def\boldsymbol{u}\right)\cdot\boldsymbol{t}_e.\label{eq:exactness_identity}
        \end{align}
        Since both $\boldsymbol{u}$ and $\boldsymbol{\sigma}=\dev\Def\boldsymbol{u}$ are single-valued on all edges (DOFs \eqref{cinc edge 1}), it follows that $\curl\boldsymbol{u}\times\boldsymbol{t}_e$ and ${\div}\boldsymbol{u}$ are single-valued on all edges. Next, consider two adjacent elements, $K$ and $K'$ that share a face $F$. For any $\delta\in\mathcal{V}(F)$, since $\curl\boldsymbol{u}\times\boldsymbol{t}_e|_e$ is single-valued for each edge $e$, we have $(\curl\boldsymbol{u})(\delta)|_K = (\curl\boldsymbol{u})(\delta)|_{K'}$, which establishes that $\curl\boldsymbol{u}$ is single-valued at all vertices. 
        
        Note that $\tr^{\cott}_{e,1}(\boldsymbol{\sigma})=\tr_{e,1}^{\cott}(\dev\Def\boldsymbol{u})=\tfrac{1}{2}\partial_{\boldsymbol{t}_e}(\boldsymbol{t}_e\cdot\curl\boldsymbol{u})$ is single-valued on all edges (DOFs \eqref{trace e 1 cinc devdef}). Since $\curl\boldsymbol{u}$ is single-valued on all vertices,  $\boldsymbol{t}_e\cdot\curl\boldsymbol{u}$ is single-valued on all edges.  
        
        Thus, $\curl\boldsymbol{u}$ and $\div\boldsymbol{u}$ are both single-valued on all edges. By \eqref{gradient identity}, $\grad\boldsymbol{u}$ is single-valued on all edges. This establishes $C^1$ smoothness of $\boldsymbol{u}$ on the edges and at the vertices. It remains to show extra vertex smoothness.

        Suppose $K$ and $K'$ share an edge $e$ and vertex $\delta\subset e$. Following \eqref{eq:exactness_identity}, we have
        \begin{align}
            \tfrac{1}{2}\nabla\left(\curl\boldsymbol{u}\times\boldsymbol{t}_e\right)+\tfrac{1}{3}\nabla\left((\div\boldsymbol{u})\boldsymbol{t}_e\right)=\nabla\partial_{\boldsymbol{t}_e}(\boldsymbol{u})-\nabla\left(\dev\Def\boldsymbol{u}\cdot\boldsymbol{t}_e\right). \label{eq:exactness_vertex1}
        \end{align}
        Since $\nabla\boldsymbol{u}$ is single-valued on all edges, we have $\nabla\partial_{\boldsymbol{t}_e}\boldsymbol{u}|_K=\nabla\partial_{\boldsymbol{t}_e}\boldsymbol{u}|_{K'}$ on $e$. Additionally, $\nabla\left(\dev\Def\boldsymbol{u}\cdot \boldsymbol{t}_e\right)=\nabla(\boldsymbol{\sigma}\cdot\boldsymbol{t}_e)$ is single-valued at all vertices. Thus, it follows from \eqref{eq:exactness_vertex1} that
        $$\nabla(\div\boldsymbol{u})(\delta)|_{K}=\nabla(\div\boldsymbol{u})(\delta)|_{K'}.$$
        This argument shows that $\nabla(\div\boldsymbol{u})$ is single-valued at all vertices. Hence, $D^{\alpha}(\Def\boldsymbol{u})$ is single-valued at vertices for all multi-indices $|\alpha|=1$. By \eqref{def identity}, we obtain that $D^{\alpha}\boldsymbol{u}$ is single-valued for $\forall|\alpha|=2.$  Since $D^{\alpha}\dev\Def\boldsymbol{u}=D^{\alpha}\boldsymbol{\sigma}$ is single-valued for $2\leq|\alpha|\leq 6$ (DOFs \eqref{cinc vertex 1}), utilizing Lemma~\ref{devdef lemma} we conclude that $D^{\alpha}\boldsymbol{u}$ is single-valued for $3\leq|\alpha|\leq 7.$ Therefore, $\boldsymbol{u}$ has $C^7$ smoothness at all vertices and $\boldsymbol{u}\in\boldsymbol{U}_{k+1,h}.$

        We have checked the surjective property of the divergence operator. To prove the exactness of the complex, it suffices to adopt a counting argument. Let $|\mathcal{V}|$, $|\mathcal{E}|$, $|\mathcal{F}|$ and $|\mathcal{T}|$ be the number of vertices, edges, faces and tetrahedrons in the triangulation respectively. Counting the dimensions of the finite element spaces, we have
        \begin{align*}
            \dim \boldsymbol{U}_{k+1,h}&=360|\mathcal{V}|+(9k-120)|\mathcal{E}|+(\frac{3k^2}{2}-\frac{21k}{2}-72)|\mathcal{F}|+(\frac{k^3}{2} - \frac{3k^2}{2}+ k - 420)|\mathcal{T}|,\\
            \dim \boldsymbol{\Sigma}_{k,h}^{\cott}&=420|\mathcal{V}|+(14k-160)|\mathcal{E}|+(3k^2 - 24k - 79)|\mathcal{F}|+(\frac{5k^3 }{6}-7k^2 +\frac{127k}{6}-399)|\mathcal{T}|,\\
            \dim \boldsymbol{\Sigma}_{k-3,h}^{\div}&=100|\mathcal{V}|+5(k-10)|\mathcal{E}|+(\frac{3k^2 }{2}-\frac{27k}{2}+3)|\mathcal{F}|+(\frac{5k^3 }{6}-\frac{17k^2 }{2}+\frac{77k}{3}-112)|\mathcal{T}|,\\
            \dim \boldsymbol{V}_{k-4,h}&=30|\mathcal{V}|+\left(\frac{(k-1)(k-2)(k-3)}{2} - 120\right)|\mathcal{T}|.
        \end{align*}
        By Euler's formula, we obtain
        $$\dim\boldsymbol{U}_{k+1,h}-\dim \boldsymbol{\Sigma}_{k,h}^{\cott}+\dim \boldsymbol{\Sigma}_{k-3,h}^{\div}-\dim \boldsymbol{V}_{k-4,h}=10(|\mathcal{V}|-|\mathcal{E}|+|\mathcal{F}|-|\mathcal{T}|)=\dim\boldsymbol{CK}.$$
        This finishes the proof.
    \end{proof}

\section{Conclusions and Outlook}
In this paper, we constructed a conforming finite element conformal complex. In particular, we investigated the intrinsic supersmoothness of a conforming, inf-sup stable, and balanced $H(\div,\Omega;\bST)$-$L^2(\mathbb R^3)$ finite element pair through conformal bubble complexes. The bubble version of the BGG diagrams can also be applied to study the intrinsic supersmoothness of other tensor-valued finite element pairs, such as $H^1(\mathbb M)$-$L^2(\mathbb R^3)$ with $\mathbb M = \bT, \bS, \bST$, which require higher smoothness than Neilan’s Stokes pair \cite{neilan2015discrete}.

Although high-order polynomials and supersmoothness need not pose an obstacle for practical computation with appropriate numerical schemes (cf.~\cite{ainsworth2023computing} for a discussion of the Argyris element), we hope that the technical tools developed in this work, including integration-by-parts formulas and bubble complexes, will be useful for future developments of simpler discretizations of the conformal deformation and conformal Hessian complexes \cite{arnold2021complexes,vcap2023bgg}. In particular, for the elasticity (Riemannian deformation) complex, Regge calculus \cite{regge1961general} can be interpreted as a finite element method fitting into a discrete complex \cite{christiansen2011linearization}, and schemes inspired by discrete mechanics appear in \cite{hauret2007diamond}. These constructions involve distributions with clear discrete geometric and topological interpretations, and they hold promise for broader applications. 

At the same time, significant challenges remain in constructing a distributional version of the conformal complex (see Remark~\ref{rmk:distributional} for the piecewise constant case). Nevertheless, the approach developed here, potentially in connection with discrete conformal geometry, opens up avenues for further investigation. In particular, it remains open whether a low-order distributional conformal complex with clear geometric and     
  topological interpretation, analogous to Regge calculus for the elasticity complex, can be constructed. 

\section*{Acknowledgements}
The work of KH was supported by a Royal Society University Research Fellowship (URF$\backslash$R1$\backslash$221398) and an ERC Starting Grant (project 101164551, GeoFEM). The work of T.L. was supported by NSFC project 123B2014. The authors would like to thank Prof. Jun Hu at Peking University for his helpful discussions.
    \appendix
    \section{Proofs}
    \label{technical proofs appendix}
    \subsection{Characterization of the conformal Killing fields}
    \label{appendix:characterization_of_CK}
\begin{proposition}
    It holds that
        \[
\boldsymbol{CK}:=\ker(\dev\Def)
=\bigl\{(\boldsymbol{x}\!\cdot\!\boldsymbol{x})\boldsymbol{a}
-2(\boldsymbol{a}\!\cdot\!\boldsymbol{x})\boldsymbol{x}
+\boldsymbol{b}\times\boldsymbol{x}
+c\,\boldsymbol{x}
+\boldsymbol{d}:\;
\boldsymbol{a},\boldsymbol{b},\boldsymbol{d}\in\mathbb{R}^3,\ c\in\mathbb{R}\bigr\}.
\]
\end{proposition}
\begin{proof}
Set \(\phi:=\tfrac13\div \boldsymbol u\). Then \(\dev\Def \boldsymbol u=0\) is equivalent to
\begin{equation}\label{eq:CK1-new}
\Def \boldsymbol u=\phi\,\mathbf I .
\end{equation}

\textbf{Step 1: Affineness of \(\phi\).}
Taking divergence of \eqref{eq:CK1-new} yields
\[
\tfrac12(\Delta \boldsymbol u+\nabla\div \boldsymbol u)=\nabla\phi,
\quad\text{hence}\quad
\Delta \boldsymbol u=-\nabla\phi
\]
since \(\div \boldsymbol u=3\phi\).
Apply \(\Def\) to the first equation and \(\Delta\) to \eqref{eq:CK1-new}; using the commutation of \(\sym\nabla\) and \(\Delta\),
\[
-\nabla^2\phi=\Delta\phi\,\mathbf I .
\]
Taking the trace gives \(-\Delta\phi=3\,\Delta\phi\), so \(\Delta\phi=0\) and therefore \(\nabla^2\phi=0\).
Thus \(\phi\) is affine:
\begin{equation}\label{eq:phi-affine-new}
\phi(\boldsymbol{x})=c-2\,\boldsymbol{a}\!\cdot\!\boldsymbol{x}
\quad\text{for some }c\in\mathbb R,\ \boldsymbol a\in\mathbb R^3.
\end{equation}

\textbf{Step 2: Explicit reconstruction.}
Define
\[
\boldsymbol u_{\boldsymbol{a},\boldsymbol{b},c,\boldsymbol{d}}(\boldsymbol{x})
:=(\boldsymbol{x}\!\cdot\!\boldsymbol{x})\boldsymbol{a}
-2(\boldsymbol{a}\!\cdot\!\boldsymbol{x})\boldsymbol{x}
+\boldsymbol{b}\times\boldsymbol{x}
+c\,\boldsymbol{x}
+\boldsymbol{d}.
\]
A direct computation shows
\[
\sym\grad\!\big((\boldsymbol{x}\!\cdot\!\boldsymbol{x})\boldsymbol{a}
-2(\boldsymbol{a}\!\cdot\!\boldsymbol{x})\boldsymbol{x}\big)
=-2(\boldsymbol{a}\!\cdot\!\boldsymbol{x})\,\mathbf I,\]
\[
\sym\grad(\boldsymbol{b}\times\boldsymbol{x})=\mathbf 0,\quad
\sym\grad(c\,\boldsymbol{x})=c\,\mathbf I,\quad
\sym\grad\boldsymbol{d}=\mathbf 0,
\]
hence
\[
\sym\grad u_{\boldsymbol{a},\boldsymbol{b},c,\boldsymbol{d}}
=\bigl(c-2\,\boldsymbol{a}\!\cdot\!\boldsymbol{x}\bigr)\mathbf I,
\]
which matches \eqref{eq:CK1-new}–\eqref{eq:phi-affine-new}. Conversely, given a solution \(u\) of \eqref{eq:CK1-new}, pick \(\boldsymbol a,c\) from \eqref{eq:phi-affine-new} and set
\(w:=u-u_{\boldsymbol{a},\mathbf 0,c,\mathbf 0}\).
Then \(\sym\grad w=0\), so \(w\) is a rigid motion:
\(w(\boldsymbol x)=\boldsymbol{b}\times\boldsymbol{x}+\boldsymbol{d}\).
Therefore, \(u\) has precisely the stated form.
\end{proof}
    \subsection{Proof of Lemma~\ref{integration by parts}}
    \label{integration by parts appendix}
    First, we introduce a simple lemma to aid our proof:
    \begin{lemma}
    Suppose $\boldsymbol{u}$, $\boldsymbol{v}$ are three-dimensional vectors defined on a tetrahedron $K$. Let $\boldsymbol{n}_+$ and $\boldsymbol{n}_-$ be the outward pointing unit vectors of face $F_+$ and $F_-$ respectively. Let $e=F_+\cap F_-$. Then it holds that:
    \label{lemma for Green's identity}
\begin{align*}
    \int_e \left(\boldsymbol{n}_{+}\cdot\boldsymbol{u}\right)\left(\boldsymbol{n}_{F_+,e}\cdot\boldsymbol{v}\right)+\left(\boldsymbol{n}_{-}\cdot\boldsymbol{u}\right)\left(\boldsymbol{n}_{F_-,e}\cdot\boldsymbol{v}\right)=\int_e \left(\boldsymbol{n}_{F_+,e}\cdot\boldsymbol{u}\right)\left(\boldsymbol{n}_{+}\cdot\boldsymbol{v}\right)+\left(\boldsymbol{n}_{F_-,e}\cdot\boldsymbol{u}\right)\left(\boldsymbol{n}_{-}\cdot\boldsymbol{v}\right).
\end{align*}
\end{lemma}
\begin{proof}
Set $\boldsymbol{n}_{F_+,e}=k_1\boldsymbol{n}_++k_2\boldsymbol{n}_-$, then we must have $\boldsymbol{n}_{F_-,e}=k_2\boldsymbol{n}_++k_1\boldsymbol{n}_-$. The proof immediately follows if we replace $\boldsymbol{n}_{F_+,e}$ and $\boldsymbol{n}_{F_-,e}$ with expansions under basis $\boldsymbol{n}_+$ and $\boldsymbol{n}_-$.
\end{proof}

Now we are ready to prove Lemma~\ref{integration by parts}. With a slight abuse of notation, we use $\mathcal{K}(F)$, $\mathcal{E}(F)$, and $\mathcal{V}(F)$ to denote tetrahedrons sharing a common face $F$, edges belonging to $F$, and vertices incident to $F$ in the mesh $\mathcal{T}_h$. Similarly, $\mathcal{K}(e)$, $\mathcal{F}(e)$, and $\mathcal{V}(e)$ refer to tetrahedrons sharing an edge $e$, faces that share an edge $e$, and vertices associated with edge $e$ within the mesh $\mathcal{T}_h.$
    \begin{proof}[Proof of Lemma~\ref{integration by parts}]
Through Green's identity of the $\curl$ operator, we have
\begin{align}
    \int_K \cott\boldsymbol{\sigma}:\boldsymbol{\tau}=\int_K\curl S^{-1}\inc\boldsymbol{\sigma}:\boldsymbol{\tau}=\int_KS^{-1}\inc\boldsymbol{\sigma}:\curl\boldsymbol{\tau}+\int_{\partial K}S^{-1}\inc\boldsymbol{\sigma}:\boldsymbol{\tau}\times\boldsymbol{n}.
    \label{int_by_parts curl}
\end{align}
We note that $\boldsymbol{\tau}$ is symmetric, therefore it holds that $\tr(\curl\boldsymbol{\tau})=\tr(\boldsymbol{\tau}\times\boldsymbol{n})=0$. Since $\inc\boldsymbol{\sigma}=\curl\left(\curl\boldsymbol{\sigma}\right)^T$ is symmetric, it follows
\begin{align}   \int_K\cott\boldsymbol{\sigma}:\boldsymbol{\tau}&=\int_K \curl\left(\curl\boldsymbol{\sigma}\right)^T:\left(\curl\boldsymbol{\tau}\right)^T+\int_{\partial K}\inc\boldsymbol{\sigma}:\boldsymbol{\tau}\times\boldsymbol{n}\nonumber\\
&=\int_K\left(\curl\boldsymbol{\sigma}\right)^T:\inc\boldsymbol{\tau}+\int_{\partial K}\left(\curl\boldsymbol{\sigma}\right)^T:\left(\curl\boldsymbol{\tau}\right)^T\times\boldsymbol{n}+\int_{\partial K}\inc\boldsymbol{\sigma}:\boldsymbol{\tau}\times\boldsymbol{n}.
\end{align}
Hence by symmetry, we obtain
\begin{align}   &\quad\int_K\cott\boldsymbol{\sigma}:\boldsymbol{\tau}-\int_K\cott\boldsymbol{\tau}:\boldsymbol{\sigma}\nonumber\\
&={\int_{\partial K}\left(\curl\boldsymbol{\sigma}\right)^T:\left(\curl\boldsymbol{\tau}\right)^T\times\boldsymbol{n}}+{\int_{\partial K}\inc\boldsymbol{\sigma}:\boldsymbol{\tau}\times\boldsymbol{n}}{-\int_{\partial K}\inc\boldsymbol{\tau}:\boldsymbol{\sigma}\times\boldsymbol{n}}\nonumber.\\
&= \underbrace{\int_{\partial K}\left(\boldsymbol{n}\cdot\left(\nabla\times\boldsymbol{\sigma}\right)\Pi_F\right)\cdot\left(\boldsymbol{n}\cdot\left(\nabla\times\boldsymbol{\tau}\right)\times\boldsymbol{n}\right)}_{I_{1,n}}+\underbrace{\int_{\partial K}\boldsymbol{n}\times\left(\nabla\times\boldsymbol{\sigma}\right)\Pi_F:\boldsymbol{n}\times\left(\nabla\times\boldsymbol{\tau}\right)\times\boldsymbol{n}}_{I_{1,t}}\nonumber\\
&\quad+ \underbrace{\int_{\partial K}\left(\boldsymbol{n}\cdot\left(-\nabla\times\boldsymbol{\sigma}\times\nabla\right)\Pi_F\right)\cdot\left(\boldsymbol{n}\cdot\boldsymbol{\tau}\times\boldsymbol{n}\right)}_{I_{2,n}}+\underbrace{\int_{\partial K}\boldsymbol{n}\times\left(-\nabla\times\boldsymbol{\sigma}\times\nabla\right)\Pi_F:\boldsymbol{n}\times\boldsymbol{\tau}\times\boldsymbol{n}}_{I_{2,t}}\nonumber\\
&\quad\underbrace{-\int_{\partial K}\left(\boldsymbol{n}\cdot\left(-\nabla\times\boldsymbol{\tau}\times\nabla\right)\Pi_F\right)\cdot\left(\boldsymbol{n}\cdot\boldsymbol{\sigma}\times\boldsymbol{n}\right)}_{I_{3,n}}\quad\underbrace{-\int_{\partial K}\Pi_F\inc\boldsymbol{\tau}\Pi_F:\Pi_F\boldsymbol{\sigma}\times\boldsymbol{n}}_{I_{3,t}},\nonumber
\end{align}
where the last equality follows from splitting integration terms on $\partial K$ into normal and tangential parts.
Next, we deal with all the normal parts through Stokes' formula on planes. Note that
\begin{align*}
    I_{1,n}&=-\int_{\partial K}\left(\boldsymbol{n}\cdot\left(\nabla\times\boldsymbol{\sigma}\right)\Pi_F\right)\cdot\left(\nabla_F\cdot\left(\boldsymbol{n}\times\boldsymbol{\tau}\times\boldsymbol{n}\right)\right)\\
    &=\underbrace{\int_{\partial K}\nabla_F\left(\boldsymbol{n}\cdot\left(\nabla\times\boldsymbol{\sigma}\right)\Pi_F\right):\boldsymbol{n}\times\boldsymbol{\tau}\times\boldsymbol{n}}_{I_{1,n}'}\\
    &\quad-\underbrace{\sum_{F\in\mathcal{F}(K)}\sum_{e\in\mathcal{E}(F)}\int_e \left(\boldsymbol{n}\cdot\left(\nabla\times\boldsymbol{\sigma}\right)\Pi_F\right)\cdot\left(\boldsymbol{n}_{F,e}\cdot\left(\boldsymbol{n}\times\boldsymbol{\tau}\times\boldsymbol{n}\right)\right)}_{E_1},
\end{align*}
\begin{align*}
    I_{2,n}&=\int_{\partial K}\left(\nabla_F\cdot\left(\boldsymbol{n}\times\left(\boldsymbol{\sigma}\times\nabla\right)\Pi_F\right)\right)\cdot\left(\boldsymbol{n}\cdot\boldsymbol{\tau}\times\boldsymbol{n}\right)\\
    &=\underbrace{-\int_{\partial K} \boldsymbol{n}\times\left(\boldsymbol{\sigma}\times\nabla\right)\Pi_F:\nabla_F\left(\boldsymbol{n}\cdot\boldsymbol{\tau}\times\boldsymbol{n}\right)}_{I_{2,n}'}\\
    &\quad+\underbrace{\sum_{F\in\mathcal{F}(K)}\sum_{e\in\mathcal{E}(F)}\int_e\left(\boldsymbol{n}_{F,e}\cdot\left(\boldsymbol{n}\times\left(\boldsymbol{\sigma}\times\nabla\right)\Pi_F\right)\right)\cdot\left(\boldsymbol{n}\cdot\boldsymbol{\tau}\times\boldsymbol{n}\right)}_{E_2},
\end{align*}
and by similar arguments, 
\begin{align*}
    I_{3,n}
    &=\underbrace{\int_{\partial K} \boldsymbol{n}\times\left(\boldsymbol{\tau}\times\nabla\right)\Pi_F:\nabla_F\left(\boldsymbol{n}\cdot\boldsymbol{\sigma}\times\boldsymbol{n}\right)}_{I_{3,n}'}\\
    &\quad\underbrace{-\sum_{F\in\mathcal{F}(K)}\sum_{e\in\mathcal{E}(F)}\int_e\left(\boldsymbol{n}_{F,e}\cdot\left(\boldsymbol{n}\times\left(\boldsymbol{\tau}\times\nabla\right)\Pi_F\right)\right)\cdot\left(\boldsymbol{n}\cdot\boldsymbol{\sigma}\times\boldsymbol{n}\right)}_{E_3}.
\end{align*}
With this notation, $\int_K\cott\boldsymbol{\sigma}:\boldsymbol{\tau}-\int_K\cott\boldsymbol{\tau}:\boldsymbol{\sigma}=\sum_{i=1}^3(I_{i,n}'+E_i+I_{i,t})$. By the symmetry of $\inc\boldsymbol{\sigma}$, it holds that
\begin{align}
    I_{3,t}=-\int_{\partial K}\sym\left(\Pi_F\boldsymbol{\sigma}\times\boldsymbol{n}\right):\Pi_F\inc\boldsymbol{\tau}\Pi_F=-\int_{\partial K}\tr_1^{\cott}(\boldsymbol{\sigma}):\Pi_F\inc\boldsymbol{\tau}\Pi_F. \label{int term trace1}
\end{align}
Notice that $I_{1,n}'$ and $I_{2,t}$ share a symmetric term $\boldsymbol{n}\times\boldsymbol{\tau}\times\boldsymbol{n}$. Then from \eqref{cross_product_identity_1} we have
\begin{align}
    &\quad I_{1,n}'+I_{2,t}\nonumber\\
    &=\int_{\partial K}-\nabla_F\left(\boldsymbol{n}\cdot\left(\boldsymbol{\sigma}\times\nabla\right)\Pi_F\right)+\partial_{\boldsymbol{n}}\left(\Pi_F\left(\boldsymbol{\sigma}\times\nabla\right)\Pi_F\right)+\nabla_F\left(\boldsymbol{n}\cdot\left(\nabla\times\boldsymbol{\sigma}\right)\Pi_F\right):\boldsymbol{n}\times\boldsymbol{\tau}\times\boldsymbol{n}\nonumber\\
    &=\int_{\partial K}2\Def_F\left(\boldsymbol{n}\cdot\sym\curl\boldsymbol{\sigma}\Pi_F\right)-\partial_{\boldsymbol{n}}\left(\Pi_F\sym\curl\boldsymbol{\sigma}\Pi_F\right):\boldsymbol{n}\times\boldsymbol{\tau}\times\boldsymbol{n}=\int_{\partial K}\tr_3^{\cott}(\boldsymbol{\sigma}):\boldsymbol{n}\times\boldsymbol{\tau}\times\boldsymbol{n}.\label{int term trace3}
\end{align}
Let's consider the rest of the terms on faces. Through a simple rotation ($
  \boldsymbol{U}:\boldsymbol{V}\times\boldsymbol{n}=-\boldsymbol{U}\times\boldsymbol{n}:\boldsymbol{V}$ for any matrix $\boldsymbol{U}$ and $\boldsymbol{V}$) and transposition, we have
\begin{align*}
    I_{2,n}'&=\int_{\partial K}\Pi_F\left(\nabla\times\boldsymbol{\sigma}\right)\times\boldsymbol{n}:\left(\boldsymbol{n}\times\boldsymbol{\tau}\cdot\boldsymbol{n}\right)\nabla_F^T,\\
    I_{3,n}'&=\int_{\partial K}\boldsymbol{n}\times(\nabla\times\boldsymbol{\tau})\times\boldsymbol{n}:(\Pi_F\boldsymbol{\sigma}\cdot\boldsymbol{n})\nabla_F^T.
\end{align*}
Therefore, we have
\begin{align}
    \nonumber&\quad I_{1,t}+I_{2,n}'+I_{3,n}'\\
    &\nonumber=\int_{\partial K}\boldsymbol{n}\times\left(\nabla\times\boldsymbol{\sigma}\right)\Pi_F+\left(\Pi_F\boldsymbol{\sigma}\cdot\boldsymbol{n}\right)\nabla_F^T:\boldsymbol{n}\times\left(\nabla\times\boldsymbol{\tau}\right)\times\boldsymbol{n}\\
    &\quad +\int_{\partial K}\Pi_F\left(\nabla\times\boldsymbol{\sigma}\right)\times\boldsymbol{n}:\left(\boldsymbol{n}\times\boldsymbol{\tau}\cdot\boldsymbol{n}\right)\nabla_F^T\nonumber\\
    &=\nonumber\int_{\partial K}\boldsymbol{n}\times\left(\nabla\times\boldsymbol{\sigma}\right)\Pi_F+\left(\Pi_F\boldsymbol{\sigma}\cdot\boldsymbol{n}\right)\nabla_F^T:\boldsymbol{n}\times\left(\nabla\times\boldsymbol{\tau}\right)\times\boldsymbol{n}+\left(\Pi_F\boldsymbol{\tau}\cdot\boldsymbol{n}\right)\nabla_F^T\times\boldsymbol{n}\\
    &\quad-\int_{\partial K}\left(\Pi_F\boldsymbol{\sigma}\cdot\boldsymbol{n}\right)\nabla_F^T:\left(\Pi_F\boldsymbol{\tau}\cdot\boldsymbol{n}\right)\nabla_F^T\times\boldsymbol{n}.\nonumber
\end{align}
Note that by \eqref{cross_product_identity_1}, 
\begin{align}
\boldsymbol{n}\times\left(\nabla\times\boldsymbol{\tau}\right)\Pi_F+\left(\Pi_F\boldsymbol{\tau}\cdot\boldsymbol{n}\right)\nabla_F^T=2\Def_F\left(\boldsymbol{n}\cdot\boldsymbol{\tau}\Pi_F\right)-\partial_{\boldsymbol{n}}\left(\Pi_F\boldsymbol{\tau}\Pi_F\right)
\end{align}
yields a symmetric formulation. Since $\nabla_F\cdot\left(\nabla_F\times\boldsymbol{n}\right)=0$, it holds that
\begin{align}
    \nonumber &\quad I_{1,t}+I_{2,n}'+I_{3,n}'\\
    &=-\int_{\partial K}\tr_2^{\cott}(\boldsymbol{\sigma}):2\Def_F\left(\boldsymbol{n}\cdot\boldsymbol{\tau}\Pi_F\right)-\partial_{\boldsymbol{n}}\left(\Pi_F\boldsymbol{\tau}\Pi_F\right)\label{int term trace2}\\
    &\quad \underbrace{-\sum_{F\in\mathcal{F}(K)}\sum_{e\in\mathcal{E}(F)}\int_e\left(\Pi_F\boldsymbol{\sigma}\cdot\boldsymbol{n}\right)\cdot\left(\Pi_F\boldsymbol{\tau}\cdot\boldsymbol{n}\right)\nabla_F^T\times\boldsymbol{n}\cdot\boldsymbol{n}_{F,e}}_{E_4}\nonumber.
\end{align}
Hence by \eqref{int term trace1}, \eqref{int term trace3} and \eqref{int term trace2}, we obtain face integral terms in the lemma. It remains to deal with edge integrals $E_i$, $i=1,2,3,4$.

On a fixed face $F$, It holds that
\begin{align*}
    \sum_{e\in\mathcal{E}(F)}\int_e\left(\Pi_F\boldsymbol{\sigma}\cdot\boldsymbol{n}\right)\cdot\left(\Pi_F\partial_{\boldsymbol{t}_{F,e}}\boldsymbol{\tau}\cdot\boldsymbol{n}\right)=-\sum_{e\in\mathcal{E}(F)}\int_e\left(\Pi_F\partial_{\boldsymbol{t}_{F,e}}\boldsymbol{\sigma}\cdot\boldsymbol{n}\right)\cdot\left(\Pi_F\boldsymbol{\tau}\cdot\boldsymbol{n}\right),
\end{align*}
due to the cancellations of the vertex value terms.
Thus, we have
\begin{align}
    E_4=\sum_{F\in\mathcal{F}(K)}\sum_{e\in\mathcal{E}(F)}\int_e\left(\Pi_F\partial_{\boldsymbol{t}_{F,e}}\boldsymbol{\sigma}\cdot\boldsymbol{n}\right)\cdot\left(\Pi_F\boldsymbol{\tau}\cdot\boldsymbol{n}\right),
\end{align}
which yields term \eqref{int edge 1} in the lemma.

By decomposing terms into tangential parts and normal parts of faces, we can obtain
\begin{align}   
&\quad E_1+E_2\nonumber\\
    &=\sum_{F\in\mathcal{F}(K)}\sum_{e\in \mathcal{E}(F)}\int_e\left(\boldsymbol{n}\cdot(\nabla\times\boldsymbol{\sigma})\cdot\boldsymbol{n}_{F,e}\right)\left(\boldsymbol{t}_{F,e}\cdot\boldsymbol{\tau}\cdot\boldsymbol{t}_{F,e}\right) -\left(\boldsymbol{n}\cdot(\nabla\times\boldsymbol{\sigma})\cdot\boldsymbol{t}_{F,e}\right)\left(\boldsymbol{n}_{F,e}\cdot\boldsymbol{\tau}\cdot\boldsymbol{t}_{F,e}\right)\nonumber \\ 
      &\quad+\sum_{F\in\mathcal{F}(K)}\sum_{e\in \mathcal{E}(F)}\int_e\left(\boldsymbol{n}_{F,e}\cdot(\nabla\times\boldsymbol{\sigma})\cdot\boldsymbol{t}_{F,e}\right)\left(\boldsymbol{n}\cdot\boldsymbol{\tau}\cdot\boldsymbol{t}_{F,e}\right)-\left(\boldsymbol{t}_{F,e}\cdot(\nabla\times\boldsymbol{\sigma})\cdot\boldsymbol{t}_{F,e}\right)\left(\boldsymbol{n}_{F,e}\cdot\boldsymbol{\tau}\cdot\boldsymbol{n}\right) \nonumber \\
      &= \sum_{F\in\mathcal{F}(K)}\sum_{e\in \mathcal{E}(F)}\int_e\left(\boldsymbol{n}\cdot(\nabla\times\boldsymbol{\sigma})\cdot\boldsymbol{n}_{F,e}\right)\left(\boldsymbol{t}_{F,e}\cdot\boldsymbol{\tau}\cdot\boldsymbol{t}_{F,e}\right)-\left(\boldsymbol{t}_{F,e}\cdot(\nabla\times\boldsymbol{\sigma})\cdot\boldsymbol{t}_{F,e}\right)\left(\boldsymbol{n}_{F,e}\cdot\boldsymbol{\tau}\cdot\boldsymbol{n}\right), \label{E1+E2}
\end{align}
where the last identity follows from Lemma~\ref{lemma for Green's identity}. Note that \eqref{E1+E2} is exactly term \eqref{int edge 2} and \eqref{int edge 3}.

Finally, a decomposition of $E_3$ gives 
\begin{align}
    E_3=\sum_{F\in\mathcal{F}(K)}\sum_{e\in \mathcal{E}(F)}\int_e\left(\boldsymbol{n}_{F,e}\cdot\boldsymbol{\sigma}\cdot\boldsymbol{n}\right)\left(\boldsymbol{t}_{F,e}\cdot(\nabla\times\boldsymbol{\tau})\cdot\boldsymbol{t}_{F,e}\right)-(\boldsymbol{n}\cdot\boldsymbol{\sigma}\cdot\boldsymbol{t}_{F,e})(\boldsymbol{n}_{F,e}\cdot(\nabla\times\boldsymbol{\tau})\cdot\boldsymbol{t}_{F,e}).
\end{align}
Utilizing Lemma~\ref{lemma for Green's identity} again, we obtain \eqref{int edge 4} and \eqref{int edge 5}. This completes the proof.
\end{proof}
\subsection{Proof of Lemma~\ref{edge_basis}}
\label{edge basis appendix}
\begin{proof}
Let $\boldsymbol{\sigma} \in \bST$ be expressed as a linear combination of the five basis tensors:
\[
\boldsymbol{\sigma} = 
u_1\,\sym(\boldsymbol{t}_e\boldsymbol{n}_+^T) +
u_2\,\sym(\boldsymbol{t}_e\boldsymbol{n}_-^T) +
u_3\,\dev(\boldsymbol{n}_+\boldsymbol{n}_+^T) +
u_4\,\dev(\boldsymbol{n}_-\boldsymbol{n}_-^T) +
u_5\,\dev\sym(\boldsymbol{n}_+\boldsymbol{n}_-^T),
\]
where $u_1, \dots, u_5$ are scalar coefficients. We aim to show that if $\boldsymbol{\sigma} = 0$, then all $u_i = 0$.

We utilize the following vanishing inner products, each of which must hold if $\boldsymbol{\sigma} = 0$:
\[
\boldsymbol{v} \cdot \boldsymbol{\sigma} \cdot \boldsymbol{w} = 0,
\quad \text{for } \boldsymbol{v}, \boldsymbol{w} \in \left\{
\boldsymbol{t}_e, \boldsymbol{n}_{F_+,e}, \boldsymbol{n}_{F_-,e}
\right\}.
\]

From these conditions, we derive the following system:
\begin{alignat*}{2}
&\text{From } \boldsymbol{t}_e \cdot \boldsymbol{\sigma} \cdot \boldsymbol{n}_{F_+,e} = 0 
&&\quad\Rightarrow\quad \tfrac{1}{2} (\boldsymbol{n}_- \cdot \boldsymbol{n}_{F_+,e})\, u_2 = 0, \\
&\text{From } \boldsymbol{t}_e \cdot \boldsymbol{\sigma} \cdot \boldsymbol{n}_{F_-,e} = 0 
&&\quad\Rightarrow\quad \tfrac{1}{2} (\boldsymbol{n}_+ \cdot \boldsymbol{n}_{F_-,e})\, u_1 = 0, \\
&\text{From } \boldsymbol{t}_e \cdot \boldsymbol{\sigma} \cdot \boldsymbol{t}_e = 0 
&&\quad\Rightarrow\quad u_3 + u_4 + (\boldsymbol{n}_+ \cdot \boldsymbol{n}_-)\, u_5 = 0, \\
&\text{From } \boldsymbol{n}_{F_+,e} \cdot \boldsymbol{\sigma} \cdot \boldsymbol{n}_{F_+,e} = 0 
&&\quad\Rightarrow\quad u_3 + \left(1 - 3(\boldsymbol{n}_- \cdot \boldsymbol{n}_{F_+,e})^2\right) u_4 + (\boldsymbol{n}_+ \cdot \boldsymbol{n}_-)\, u_5 = 0, \\
&\text{From } \boldsymbol{n}_{F_-,e} \cdot \boldsymbol{\sigma} \cdot \boldsymbol{n}_{F_-,e} = 0 
&&\quad\Rightarrow\quad \left(1 - 3(\boldsymbol{n}_+ \cdot \boldsymbol{n}_{F_-,e})^2\right) u_3 + u_4 + (\boldsymbol{n}_+ \cdot \boldsymbol{n}_-)\, u_5 = 0, \\
&\text{From } \boldsymbol{n}_{F_+,e} \cdot \boldsymbol{\sigma} \cdot \boldsymbol{n}_{F_-,e} = 0 
&&\quad\Rightarrow\quad \left(\boldsymbol{n}_{F_+,e} \cdot \boldsymbol{n}_{F_-,e} \right) u_3
+ \left(\boldsymbol{n}_{F_+,e} \cdot \boldsymbol{n}_{F_-,e} \right) u_4\\&&&\quad \quad
\quad\quad+ \left[ (\boldsymbol{n}_{F_+,e} \cdot \boldsymbol{n}_{F_-,e})(\boldsymbol{n}_+ \cdot \boldsymbol{n}_-)
- \tfrac{3}{2} (\boldsymbol{n}_{F_+,e} \cdot \boldsymbol{n}_-) (\boldsymbol{n}_{F_-,e} \cdot \boldsymbol{n}_+)
\right] u_5 = 0.
\end{alignat*}

Since the angle $\theta$ between $F_+$ and $F_-$ satisfies $\sin \theta \neq 0$, it follows that
\[
\boldsymbol{n}_- \cdot \boldsymbol{n}_{F_+,e} \neq 0,
\quad \boldsymbol{n}_+ \cdot \boldsymbol{n}_{F_-,e} \neq 0.
\]
Thus, the first two equations immediately imply
\[
u_1 = u_2 = 0.
\]

Next, using the remaining three equations, we subtract and combine them to eliminate $u_5$ and find
\[
u_3 = u_4 = 0.
\]

Finally, if $\theta \neq \pi/2$, then $\boldsymbol{n}_+ \cdot \boldsymbol{n}_- \neq 0$, so the relation
\[
u_3 + u_4 + (\boldsymbol{n}_+ \cdot \boldsymbol{n}_-) u_5 = 0
\]
gives $u_5 = 0$. If $\theta = \pi/2$, we may instead use the $\boldsymbol{n}_{F_+,e} \cdot \boldsymbol{\sigma} \cdot \boldsymbol{n}_{F_-,e}$ relation (the sixth equation) to deduce $u_5 = 0$.

Hence, all coefficients vanish, and $\boldsymbol{\sigma} = 0$ implies $u_i = 0$ for all $i$. This completes the proof.
\end{proof}

\subsection{Proof of Lemma~\ref{face_basis}}
\label{face basis appendix}
\begin{proof}
We follow the same routine as in the proof of Lemma~\ref{edge_basis}. Suppose
\[
\boldsymbol{\sigma}
= u_1\,\sym(\boldsymbol{n}\boldsymbol{t}_{F,1}^T)
+ u_2\,\sym(\boldsymbol{n}\boldsymbol{t}_{F,2}^T)
+ u_3\,\sym(\boldsymbol{t}_{F,1}\boldsymbol{t}_{F,2}^T)
+ u_4\,(\boldsymbol{t}_{F,1}\boldsymbol{t}_{F,1}^T - \boldsymbol{t}_{F,2}\boldsymbol{t}_{F,2}^T)
+ u_5\,\dev(\boldsymbol{n}\boldsymbol{n}^T)
\]
for some constants \( u_i \in \mathbb{R} \), \( 1 \leq i \leq 5 \). We aim to show \( \boldsymbol{\sigma} = 0 \) implies \( u_i = 0 \) for all \( i \).

Evaluating the six linearly independent scalar quantities
\[
\boldsymbol{t}_{F,1}\cdot\boldsymbol{\sigma}\cdot\boldsymbol{t}_{F,2},\quad
\boldsymbol{t}_{F,1}\cdot\boldsymbol{\sigma}\cdot\boldsymbol{t}_{F,1},\quad
\boldsymbol{t}_{F,2}\cdot\boldsymbol{\sigma}\cdot\boldsymbol{t}_{F,2},\quad
\boldsymbol{n}\cdot\boldsymbol{\sigma}\cdot\boldsymbol{n},\quad
\boldsymbol{t}_{F,1}\cdot\boldsymbol{\sigma}\cdot\boldsymbol{n},\quad
\boldsymbol{t}_{F,2}\cdot\boldsymbol{\sigma}\cdot\boldsymbol{n},
\]
and using that \( \boldsymbol{\sigma} = 0 \), we obtain the system:
\begin{alignat*}{2}
\boldsymbol{t}_{F,1} \cdot \boldsymbol{\sigma} \cdot \boldsymbol{t}_{F,2} &= 0 &\quad&\Rightarrow\quad u_3 = 0, \\
\boldsymbol{t}_{F,1} \cdot \boldsymbol{\sigma} \cdot \boldsymbol{t}_{F,1} &= 0 &\quad&\Rightarrow\quad u_4 - \tfrac{1}{3} u_5 = 0, \\
\boldsymbol{t}_{F,2} \cdot \boldsymbol{\sigma} \cdot \boldsymbol{t}_{F,2} &= 0 &\quad&\Rightarrow\quad -u_4 - \tfrac{1}{3} u_5 = 0, \\
\boldsymbol{n} \cdot \boldsymbol{\sigma} \cdot \boldsymbol{n} &= 0 &\quad&\Rightarrow\quad \tfrac{2}{3} u_5 = 0, \\
\boldsymbol{t}_{F,1} \cdot \boldsymbol{\sigma} \cdot \boldsymbol{n} &= 0 &\quad&\Rightarrow\quad u_1 = 0, \\
\boldsymbol{t}_{F,2} \cdot \boldsymbol{\sigma} \cdot \boldsymbol{n} &= 0 &\quad&\Rightarrow\quad u_2 = 0.
\end{alignat*}

Solving the above system yields \( u_i = 0 \) for all \( 1 \leq i \leq 5 \), completing the proof.
\end{proof}

\subsection{Proof of Lemma~\ref{devdef lemma}}
\label{appendix: devdef lemma}
\begin{proof}
Note that $\Def \boldsymbol{u}=\tfrac12(\nabla \boldsymbol{u}+\nabla \boldsymbol{u}^\top)$ so that $\tr(\Def \boldsymbol{u})=\div \boldsymbol{u}$ and
$\dev\Def \boldsymbol{u}=\Def \boldsymbol{u}-\tfrac13(\div \boldsymbol{u})\mathbf I$.
Assume
\[
D^\beta(\dev\Def \boldsymbol{u})=0\qquad\text{for all }|\beta|=2.
\]
We will prove $D^\beta(\div \boldsymbol{u})=0$ for all $|\beta|=2$. Then
$D^\beta(\Def \boldsymbol{u})=0$ for $|\beta|=2$, and a final differentiation of the identity in Step~1 yields $D^\alpha \boldsymbol{u}=0$ for all $|\alpha|=3$.

\medskip\noindent\textbf{Step 1:}
For any $i,j,m\in\{1,2,3\}$,
\begin{equation}\label{eq:basic}
\partial_{ij}u_m
=\partial_i(\Def \boldsymbol{u})_{mj}
+\partial_j(\Def \boldsymbol{u})_{mi}
-\partial_m(\Def \boldsymbol{u})_{ij}.
\end{equation}
This follows by expanding $(\Def \boldsymbol{u})_{ab}=\tfrac12(\partial_a u_b+\partial_b u_a)$ and commuting partial derivatives.

\medskip\noindent\textbf{Step 2:}
Insert $\Def \boldsymbol{u}=\dev\Def \boldsymbol{u}+\tfrac13(\div \boldsymbol{u})\mathbf I$ into \eqref{eq:basic}:
\[
\partial_{ij}u_m
=\partial_i(\dev\Def \boldsymbol{u})_{mj}
+\partial_j(\dev\Def \boldsymbol{u})_{mi}
-\partial_m(\dev\Def \boldsymbol{u})_{ij}
+\tfrac13\!\big(\partial_i(\div \boldsymbol{u})\delta_{mj}
+\partial_j(\div \boldsymbol{u})\delta_{mi}
-\partial_m(\div \boldsymbol{u})\delta_{ij}\big).
\]
Differentiating once more shows that third derivatives of $\boldsymbol{u}$ depend only on second derivatives of $\div \boldsymbol{u}$, because all second derivatives of $\dev\Def \boldsymbol{u}$ vanish by hypothesis. Hence, it suffices to prove $D^\beta(\div \boldsymbol{u})=0$ for all $|\beta|=2$.

\medskip\noindent{\textbf{Step 3}: Off–diagonal second derivatives of $\div \boldsymbol{u}$ vanish.}
Let $(i,j,m)$ be a permutation of $(1,2,3)$, so $j\neq m$. Using
$\div \boldsymbol{u}=3\big((\Def \boldsymbol{u})_{ii}-(\dev\Def \boldsymbol{u})_{ii}\big)$, we get
\[
\partial_{jm}\div \boldsymbol{u}
=3\Big(\partial_{jm}(\Def \boldsymbol{u})_{ii}-\partial_{jm}(\dev\Def \boldsymbol{u})_{ii}\Big)
=3\,\partial_i(\partial_{jm}u_i),
\]
since $(\Def \boldsymbol{u})_{ii}=\partial_i u_i$ and $D^\beta(\dev\Def \boldsymbol{u})=0$ for $|\beta|=2$.
From \eqref{eq:basic} with indices $(j,m,i)$ and then applying $\partial_i$,
\[
\partial_i\partial_{jm}u_i
=\partial_{ij}(\Def \boldsymbol{u})_{im}
+\partial_{im}(\Def \boldsymbol{u})_{ij}
-\partial_{ii}(\Def \boldsymbol{u})_{jm}.
\]
Replace $\Def \boldsymbol{u}$ by $\dev\Def \boldsymbol{u}+\tfrac13(\div \boldsymbol{u})\mathbf I$. The $\dev\Def$-terms vanish after two derivatives by the hypothesis; the trace terms involve $\delta_{im},\delta_{ij},\delta_{jm}$, all zero for a permutation. Hence $\partial_i\partial_{jm}u_i=0$, so $\partial_{jm}\div \boldsymbol{u}=0$ for $j\neq m$.

\medskip\noindent{\textbf{Step 4}: Diagonal second derivatives of $\div \boldsymbol{u}$ vanish.}
Fix distinct $i,j,m$. The identities
\begin{align*}
(\partial_{ii}+\partial_{jj})\,\partial_i u_i
&=\partial_{ii}(\partial_i u_i-\partial_j u_j)
  +\partial_{ij}(\partial_j u_i+\partial_i u_j),\\
(\partial_{jj}+\partial_{mm})\,\partial_i u_i
&=\partial_{jm}(\partial_j u_m+\partial_m u_j)
 +\partial_{jj}(\partial_i u_i-\partial_m u_m)
 +\partial_{mm}(\partial_i u_i-\partial_j u_j)
\end{align*}
rewrite, using $\partial_\ell u_r+\partial_r u_\ell=2(\Def \boldsymbol{u})_{\ell r}$ and
$\partial_i u_i-\partial_j u_j=(\dev\Def \boldsymbol{u})_{ii}-(\dev\Def \boldsymbol{u})_{jj}$, as linear combinations of second derivatives of $\dev\Def \boldsymbol{u}$; hence each vanishes by the hypothesis. By symmetry (replace $j$ by $m$),
$(\partial_{ii}+\partial_{mm})\,\partial_i u_i=0$ as well. Therefore,
\[
\partial_{ii}\partial_i u_i
=\tfrac12\!\left[(\partial_{ii}+\partial_{jj})\partial_i u_i
               +(\partial_{ii}+\partial_{mm})\partial_i u_i
               -(\partial_{jj}+\partial_{mm})\partial_i u_i\right]=0.
\]
Finally,
\[
\partial_{ii}\div \boldsymbol{u}
=3\Big(\partial_{ii}\partial_i u_i-\partial_{ii}(\dev\Def \boldsymbol{u})_{ii}\Big)
=3\,\partial_{ii}\partial_i u_i=0,
\]
since $D^\beta(\dev\Def \boldsymbol{u})=0$ for $|\beta|=2$. Together with Step~3, this gives $D^\beta(\div \boldsymbol{u})=0$ for all $|\beta|=2$.

\medskip
As noted at the start, $D^\beta(\Def \boldsymbol{u})=0$ for all $|\beta|=2$, and differentiating \eqref{eq:basic} once more yields $D^\alpha \boldsymbol{u}=0$ for all $|\alpha|=3$. This completes the proof.
\end{proof}

\subsection{Proof of Lemma~\ref{two dim divdiv ST bubble complex extra smoothness}}
\label{appendix: two dim divdiv ST bubble complex extra smoothness}
\begin{proof}
We first verify the above sequence is indeed a complex. Note that for any $u\in b_F^2P_{k-4}^{(3)}(F;\mathbb R)$, $$\tr(\Def_F\curl_Fu)=\div_F\curl_Fu=0.$$
Furthermore, $\grad_F u\in b_FP_{k-2}^{(4)}(F;\Pi_F\mathbb R^3)$. Therefore, from Lemma~\ref{two dim divdiv S bubble complex}, it follows that $$\Def_F\curl_F u=\sym\curl_F(\grad_Fu)\in \mathbb B_k^{\div_F\div_F,(5)}(F;\mathbb S_F\cap\mathbb T_F).$$

 Suppose $\boldsymbol{\sigma}\in \mathbb{B}^{\div_F\div_F,(5)}_k(F;\mathbb{S}_F\cap\mathbb{T}_F)|_F$. By integration by parts (cf.~\cite{chen2022divdivanddiv}), we obtain 
$$\int_F\div_F\div_F\boldsymbol\sigma{q}=\int_F\Pi_F\boldsymbol{\sigma}\Pi_F:\nabla_F^2{q},\quad \forall q\in P_1^{+}(F;
\mathbb{R})|_F. $$
On face $F$, it holds that $$\nabla_F^2P_1(F;\mathbb{R})|_F=0,\quad\nabla_F^2\left((\Pi_F\boldsymbol{x})\cdot(\Pi_F\boldsymbol{x})\right)=2\Pi_F\mathbf{I}\Pi_F.$$
Since $\Pi_F\boldsymbol{\sigma}\Pi_F$ is traceless, we have
$$\Pi_F\boldsymbol{\sigma}\Pi_F:\Pi_F\mathbf{I}\Pi_F=\Pi_F\boldsymbol{\sigma}\Pi_F:\mathbf{I}=\tr\left(\Pi_F\boldsymbol{\sigma}\Pi_F\right)=0.$$
Hence
$$\int_F\div_F\div_F\boldsymbol\sigma{q}=0,\quad \forall {q}\in P_1^{+}(F;
\mathbb{R})|_F $$and the above sequence is a complex.

    If $\div_F\div_F\boldsymbol{\sigma}=0$, by Lemma~\ref{two dim divdiv S bubble complex}, there exists $\boldsymbol{v}\in b_FP_{k-2}(F;\Pi_F\mathbb{R}^3)|_F$ such that $\boldsymbol{\sigma}=\sym\curl_F\boldsymbol{v}$. Since $\boldsymbol{\sigma}$ is traceless, $\operatorname{rot}_F\boldsymbol{v}=0$. Therefore, from Lemma~\ref{two dim stokes bubble complex}, there exists $u\in P_{k-4}(F;\mathbb{R})|_F$ such that $$\boldsymbol{v}=\grad_F(b_F^2u),\quad \boldsymbol{\sigma}=\Def_F\curl_F (b_F^2u).$$    
    By definition,    $$D_F^{\alpha}\Def_F\curl_F(b_F^2u)(\delta)=0,\quad \forall 2\leq|\alpha|\leq5,\quad\forall\delta\in\mathcal{V}(F),$$ and thus
    $$D_F^{\alpha}(b_F^2u)(\delta)=0,\quad \forall 4\leq|\alpha|\leq7,\quad\forall\delta\in\mathcal{V}(F).$$ Hence, $u\in P^{(3)}_{k-4}(F;\mathbb{R})$ and it follows that $$\mathbb{B}^{\div_F\div_F,(5)}_k(F;\mathbb{S}_F\cap\mathbb{T}_F)|_F\cap\ker(\div_F\div_F)=\Def_F\curl_F\left[b_F^2P_{k-4}^{(3)}(F;\mathbb{R})|_F\right].$$

    It remains to show $\div_F\div_F$ is surjective. This is a consequence of a dimension count. Suppose that $\boldsymbol{\tau}$ satisfies the following conditions:
    \begin{align*}
        D_F^{\alpha}\boldsymbol{\tau}(\delta)&=0,\quad \forall 0\leq|\alpha|\leq5,\quad\forall \delta\in \mathcal{V}(F).\\
        \int_e\tr_{e,1}^{\div_F\div_F}(\boldsymbol{\tau})q&=0,\quad\forall q\in P_{k-12}(e;\mathbb{R}),\quad\forall e\in\mathcal{E}(F).\\
        \int_e\tr_{e,2}^{\div_F\div_F}(\boldsymbol{\tau})q&=0,\quad\forall q\in P_{k-11}(e;\mathbb{R}),\quad\forall e\in\mathcal{E}(F).
        \end{align*}
        Then $\boldsymbol{\tau}\in\mathbb{B}^{\div_F\div_F,(5)}_k(F;\mathbb{S}_F\cap\mathbb{T}_F)|_F$. It follows that
        \begin{align*}           &\quad\dim\div_F\div_F\mathbb{B}^{\div_F\div_F,(5)}_k(F;\mathbb{S}_F\cap\mathbb{T}_F)|_F\\
        &=\dim\mathbb{B}^{\div_F\div_F,(5)}_k(F;\mathbb{S}_F\cap\mathbb{T}_F)|_F-\dim P_{k-4}^{(3)}(F;\mathbb{R})|_F\\
        &\geq \dim P_k(F;\mathbb{S}_F\cap\mathbb{T}_F)-6\binom{7}{2}-3(2k-21)-\dim P_{k-4}^{(3)}(F;\mathbb{R})|_F\\
        &=\dim {P}_{k-2}^{(3)}(F;\mathbb{R})|_F \cap {P}_1^{+}(F;\mathbb{R})|_F^{\perp},
        \end{align*}
        which completes the proof. 
\end{proof}
 \bibliographystyle{plain}
\bibliography{ref}
\end{document}